\documentclass[article]{article}
\usepackage{color}
\usepackage{graphicx}
\usepackage{amsmath}
\usepackage{amssymb}
\usepackage{mathrsfs}
\usepackage[numbers,sort&compress]{natbib}
\usepackage{amsthm}
\usepackage{extarrows}
\usepackage{tikz}
\usepackage{float}
\usepackage{appendix}
\usepackage{hyperref}
\hypersetup{hypertex=true,
            colorlinks=true,
            linkcolor=blue,
            anchorcolor=blue,
            citecolor=red}
\usepackage{indentfirst}
\newtheorem{theorem}{Theorem}[section]
\newtheorem{lemma}{Lemma}[section]
\newtheorem{corollary}{Corollary}[section]
\newtheorem{proposition}{Proposition}[section]

\newtheorem{remark}{Remark}[section]
\newtheorem{RHP}{RHP}[section]
\newtheorem{Dbar}{$\bar{\partial}$-Problem}[section]
\usepackage{caption}
\usepackage{float}
\usepackage{subfigure}
\makeatletter
\@addtoreset{equation}{section}
\makeatother

\begin{document}
\title{Long time asymptotics for the nonlocal mKdV equation with finite density initial data}
\author{Xuan Zhou$^{a}$, Engui Fan$^{a}$\thanks{Corresponding author, with e-mail address
as faneg@fudan.edu.cn} }
\footnotetext[1]{ \  School of Mathematical Sciences, Fudan University, Shanghai 200433, P.R. China.}

\date{ }
\baselineskip=16pt
\maketitle
\begin{abstract}
\baselineskip=16pt
In this paper, we consider the Cauchy problem for an integrable real nonlocal (also called reverse-space-time) mKdV equation with nonzero
boundary conditions
\begin{align*}
    &q_t(x,t)-6\sigma q(x,t)q(-x,-t)q_{x}(x,t)+q_{xxx}(x,t)=0, \\
    &q(x,0)=q_{0}(x),\lim_{x\to \pm\infty} q_{0}(x)=q_{\pm},
\end{align*}
where $|q_{\pm}|=1$ and $q_{+}=\delta q_{-}$, $\sigma\delta=-1$. Based on the spectral analysis of
the Lax pair, we express the solution of the Cauchy problem of  the nonlocal mKdV equation in terms of a Riemann-Hilbert problem.
In a fixed space-time solitonic region $-6<x/t<6$, we apply $\bar{\partial}$-steepest descent method to analyze the long-time
asymptotic behavior of the solution $q(x,t)$.  We find that the long time asymptotic behavior of $q(x,t)$ can be characterized with an $N(\Lambda)$-soliton on discrete spectrum and leading order term $\mathcal{O}(t^{-1/2})$ on continuous spectrum up to an residual error order $\mathcal{O}(t^{-1})$.
\end{abstract}
\baselineskip=16pt

\newpage
\tableofcontents

\section{Introduction}
Since the inverse scattering transform (IST) technique, one of the most powerful tool to investigate solitons of nonlinear integrable models, was firstly presented by Gardner, Greene, Kruskal and Miurra \cite{GGKM}. The dressing method and Riemann-Hilbert (RH) approach are modern versions of IST \cite{ve1973,ve1974}. The RH approach allows getting the solution $q(x,t)$ of the original nonlinear problem by solving a RH problem with specific jump matrix on a given curve. Generally, the Cauchy problems of nonlinear integrable systems can be solved by suing IST or RH method only in the case of refectioness potentials. So a natural idea is to study the asymptotic behavior of solutions in general case.

The analysis of the long time behavior of the solution of the initial value problem for a nonlinear integrable equation, the so-called nonlinear version of the steepest-decent method has proved to be extremely efficient. Being developed by many researchers, the method was first
carried out with IST method by Manakov in 1974 \cite{Manakov1974}. Being developed by many researchers, the method has finally been put into a rigorous shape by Deift and Zhou in 1993 \cite{DZ1993}.  They developed rigorous analytic method to present the long time asymptotic representation of the solution for defocusing mKdV equation by deforming contours to reduce the original RH problem to a model whose solution can be derived in terms of parabolic cylinder functions. Then this method has been widely applied to the focusing NLS equation, KdV equation, Camassa-Holm equation, Degasperis-Procesi, Fokas-Lenells equation etc \cite{NLS0,NLS1,KDV,CH,DP,FL}.

Recently, McLaughlin and Miller have developed a nonlinear steepest descent-method for the asymptotic analysis of RH problems based on the analysis of $\bar{\partial}$-problems, rather than the asymptotic analysis of singular integrals on contours \cite{MM1,MM2}. When it is applied to integrable systems, the $\bar{\partial}$-steepest descent method also has displayed some advantages, such as avoiding delicate estimates involving $L^{p}$ estimates of Cauchy projection operators, and leading the non-analyticity in the RH problem reductions to a $\bar{\partial}$--problem in some sectors of the complex plane. Dieng and McLaughlin obtained sharp asymptotics of solutions of the defocusing nonlinear NLS equation\cite{DM}, based on $\bar{\partial}$-method and under essentially minimal regularity assumptions on initial data. Cussagna and Jenkins study the defocusing NLS equation with finite density initial data \cite{CJ}. Borghese, Jenkins and McLaughlin have studied the Cauchy problem for the focusing NLS equation using the $\bar{\partial}$-generalization of the nonlinear steepest descent method, which has been conjectured for a long time \cite{BJ}; Jenkins and Liu studied the derivative NLS equation for generic initial data in a weighted Sobolev space\cite{JL}.

In this paper, we investigate the long-time asymptotic behavior for the Cauchy problem of an integrable real nonlocal mKdV
equation under nonzero boundary conditions
\begin{align}
    &q_t(x,t)-6\sigma q(x,t)q(-x,-t)q_{x}(x,t)+q_{xxx}(x,t)=0, \label{nmkdv}\\
    &q(x,0)=q_{0}(x),\lim_{x\to \pm\infty} q_{0}(x)=q_{\pm},
\end{align}
where $|q_{\pm}|=1$ and $q_{+}=\delta q_{-}$, $\sigma\delta=-1$. \eqref{nmkdv} was introduced in \cite{nmkdv1,nmkdv2}, where $\sigma=\pm1$ denote the defocusing and focusing  cases, respectively, $q(x,t)$ is a real function. \eqref{nmkdv} can be regarded as the integrable nonlocal extension of the local mKdV equation
\begin{equation}
q_t(x,t)-6\sigma q^{2}(x,t)q_{x}(x,t)+q_{xxx}(x,t)=0,\label{lmkdv}
\end{equation}
which appears in various of physical fields \cite{eg}. From the perspective of their structure, \eqref{lmkdv} can be translated to \eqref{nmkdv}
by replacing $q^{2}(x,t)$ with the \emph{PT}-symmetric term $q(x,t)q(-x,-t)$ (In the real field, the definition of \emph{PT} symmetry is given by \emph{P} : $x\rightarrow-x$ and \emph{T} : $t\rightarrow-t$) \cite{PT}. The general form of the nonlocal mKdV equation was shown to appear in the nonlinear oceanic and atmospheric dynamical system \cite{eg2}. Particularly, (\romannumeral1) for the \emph{PT}-symmetric case $q(-x,-t)=q(x,t)$, \eqref{nmkdv} reduces to the usual mKdV equation \eqref{lmkdv}; (\romannumeral2) for the anti-\emph{PT} -symmetric case
$q(-x,-t)=-q(x,t)$, the focusing (defocusing) nonlocal mKdV equation \eqref{nmkdv} reduces to the defocusing (focusing) mKdV equation;
(\romannumeral3) for $q(-x,-t)=-q(-x,-t)$, the focusing (defocusing) nonlocal mKdV equation \eqref{nmkdv} reduces to the defocusing (focusing) nonlocal mKdV equation \cite{ZY}.

The Darboux transformation was used to seek for soliton solutions of the focusing \eqref{nmkdv} \cite{Darboux}, and the IST for the focusing \eqref{nmkdv} with ZBC was presented \cite{IST}. In \cite{ZY}, Zhang and Yan rigorously analyzed dynamical behaviors of solitons and their interactions for four distinct cases of the reflectionless potentials for both focusing and defocusing nonlocal mKdV equations with NZBCs. Moreover, the long-time asymptotics for the nonlocal defocusing mKdV equation with decaying initial data were investigated via Deift-Zhou steepest-descent method \cite{HF}

In the present paper, we study the long-time behavior of the Cauchy problem for nonlocal mKdV equation (\ref{nmkdv}) with finite density type initial data in the case of $\sigma\delta=-1$. The rest of the paper is organized as follows. In Section \ref{RHconstruct}, we get down to the spectral analysis on the Lax pair. Based on Jost solutions and scattering data, the RH problem for $m(z)$ is established. In Section 3, we give the distribution of saddle points for different $\xi$ and signature table. In Section 4,we conduct several deformations to obtain a mixed $\bar{\partial}$-RH problem for $m^{(2)}(z)$. In Section 5, we decompose $m^{(2)}(z)$ into a pure RH problem for $m^{(2)}_{RHP}$
and a pure $\bar{\partial}$-Problem for $m^{(3)}(z)$. Furthermore, The $m^{(2)}_{RHP}$ can be obtained via an
modified reflectionless RH problem $m^{\Lambda}(z|\sigma_{d}^{\Lambda})$ and an inner model $m^{(lo)}(z)$ for the stationary phase point
$\zeta_{k}$ which are approximated by parabolic cylinder model. In Section 6, we derive the long-time behavior of nonlocal mKdV equation (\ref{nmkdv}) in the soliton region.
\hspace*{\parindent}

\section{The spectral analysis and the RH problem}\label{RHconstruct}
\subsection{Notations}
We recall  some notations. $\sigma_1, \sigma_2,\sigma_3$ are classical Pauli matrices
as follows
\begin{equation*}
    \sigma_1=\begin{bmatrix} 0 & 1 \\ 1 & 0 \end{bmatrix},\quad
    \sigma_2=\begin{bmatrix} 0 & -i \\ i & 0 \end{bmatrix}, \quad
    \sigma_3=\begin{bmatrix} 1 & 0 \\ 0 & -1 \end{bmatrix}.
\end{equation*}

We introduce the Japanese bracket $\langle x\rangle:=\sqrt{1+|x|^2}$ and the normed spaces:
\begin{itemize}
    \item A weighted $L^{p,s}(\mathbb{R})$ is defined by
    \begin{equation*}
        L^{p,s}(\mathbb{R})=\{u\in L^{p}(\mathbb{R}):\langle x\rangle^s u(x)\in L^{p}(\mathbb{R})\}.
    \end{equation*}
    And $\Vert u \Vert_{L^{p,s}(\mathbb{R})}:=\Vert\langle x \rangle^{s}u \Vert_{L^{p}(\mathbb{R})}$.
    \item A Sobolev space is defined by
    \begin{equation*}
        W^{m,p}(\mathbb{R})=\{u\in L^{p}(\mathbb{R}): \partial ^{j}u(x)\in L^{p}(\mathbb{R}) \quad{\rm for}\quad j=0,1,2,\dots,m\}.
    \end{equation*}
    And $\Vert u \Vert_{W^{m,p}(\mathbb{R})}:=\sum_{j=0}^{m}\Vert \partial^{j}u \Vert_{L^{p}(\mathbb{R})}$. Additionally, we are
    used to expressing $H^{m}(\mathbb{R}):=W^{m,2}(\mathbb{R})$.
    \item A weighted Sobolev space is defined by
    \begin{equation*}
        H^{m,s}(\mathbb{R}):=L^{2,s}(\mathbb{R})\cap H^{m}(\mathbb{R}).
    \end{equation*}
\end{itemize}

In this paper, we use $a\lesssim b$ to express $\exists c=c(\xi)>0$, s.t. $a\leqslant cb$.

\subsection{The Lax pair and spectral analysis}

The nonlocal mKdV equation \eqref{nmkdv} posses the nonlocal Lax pair
\begin{equation}\label{lax pair}
     \Phi_x=X\Phi, \quad \Phi_t=T\Phi,
\end{equation}
where
\begin{equation}
\begin{split}
& X=ik\sigma_3+Q,\\
& T=[4k^{2}+2\sigma q(x,t)q(-x,-t)]X-2ik\sigma_{3}Q_{x}+[Q_{x},Q]-Q_{xx},\\
& Q=\begin{bmatrix} 0 & q(x,t) \\  \sigma q(-x,-t) & 0 \end{bmatrix}, \sigma=\pm1,
\end{split}
\end{equation}
and $\Phi=\Phi(x,t;k)$ is a matrix eigenfunction, $k$ is an iso-spectral parameter.

Taking $q(x,t)=q_{\pm}$ in the Lax pair \eqref{lax pair}, we get the spectral problems
\begin{equation}\label{asy lax}
    \phi_x=X_{\pm}\phi, \quad \phi_t=T_{\pm}\phi,
\end{equation}
with $Q_{\pm}=\lim_{x\to \pm\infty}Q(x,t)=\begin{bmatrix} 0 & q_{\pm} \\  - q_{\pm} & 0 \end{bmatrix}$, we have the fundamental matrix solution of Eq. (\ref{asy lax}) as
\begin{align}
 \phi_{\pm}(x,t;k)=\left\{
        \begin{aligned}
    &E_{\pm}(k)e^{it\theta (x,t;k)\sigma_{3}}, \quad k\neq\pm i,\\
    &I+[x+(4k^{2}-2)t]X_{\pm}(k), \quad k=\pm i,
    \end{aligned}
        \right.
\end{align}
where
 \begin{equation}
    E_{\pm}(k)=\begin{bmatrix} 1 & \frac{iq_{\pm}}{k+\lambda} \\  \frac{iq_{\pm}}{k+\lambda} & 1 \end{bmatrix}, \hspace{0.5cm} \lambda^{2}-k^{2}=1, \hspace{0.5cm} \theta(x,t;k)=\lambda [x+(4k^{2}-2)t].
 \end{equation}
To avoid multi-valued case of eigenvalue $\lambda$, we introduce a uniformization variable
\begin{equation}
    z=k+\lambda,
\end{equation}
and obtain two single-valued functions
\begin{equation}
    k=\frac{1}{2}(z-\frac{1}{z}),\hspace{0.5cm}\lambda=\frac{1}{2}(z+\frac{1}{z}).
\end{equation}
We can define two domains $D_{+}$, $D_{-}$ and their boundary $\Sigma$ on $z$-plane by
\begin{align*}
    &D_{+}=\{z\in \mathbb{C}: (|z|-1) Imz>0\},\\
    &D_{-}=\{z\in \mathbb{C}: (|z|-1) Imz<0\},\\
    &\Sigma=\mathbb{R}\cup \{z\in \mathbb{C}: |z|=1\}.
\end{align*}
which are shown in the Figure \ref{analregion&spectrumsdis}.\\
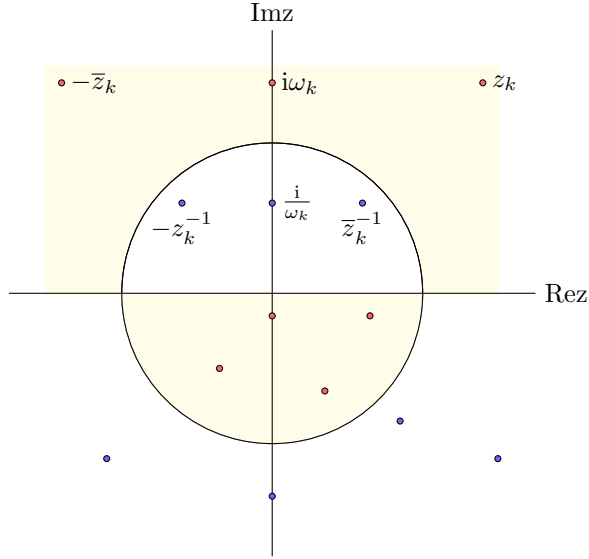
\begin{figure}[H]
\begin{center}
\begin{tikzpicture}[node distance=2cm]
\draw [fill=pink,ultra thick,color=yellow!10] (0,0) rectangle (3,3);
\draw [fill=pink,ultra thick,color=yellow!10] (0,0) rectangle (-3,3);
\draw [fill=white] (0,0) circle [radius=2];
\filldraw[color=yellow!10](0,0)-- (-2,0) arc (180:270:2);
\filldraw[color=yellow!10](0,0)-- (0,-2) arc (270:360:2);
\draw [](-3.5,0)--(3.5,0)  node[right, scale=1] {Rez};
\draw [](0,-3.5)--(0,3.5)  node[above, scale=1] {Imz};
\draw[fill=red!60] (2.8,2.8) circle [radius=0.04];
\draw[fill=red!60] (-2.8,2.8) circle [radius=0.04];
\draw[fill=red!60] (0,2.8) circle [radius=0.04];
\draw[fill=blue!60] (0,1.2) circle [radius=0.04];
\draw[fill=blue!60] (1.2,1.2) circle [radius=0.04];
\draw[fill=blue!60] (-1.2,1.2) circle [radius=0.04];

\draw[fill=red!60] (-0.7,-1) circle [radius=0.04];
\draw[fill=red!60] (0.7,-1.3) circle [radius=0.04];
\draw[fill=red!60] (0,-0.3) circle [radius=0.04];
\draw[fill=red!60] (1.3,-0.3) circle [radius=0.04];
\draw[fill=blue!60] (-2.2,-2.2) circle [radius=0.04];
\draw[fill=blue!60] (3,-2.2) circle [radius=0.04];
\draw[fill=blue!60] (1.7,-1.7) circle [radius=0.04];
\draw[fill=blue!60] (0,-2.7) circle [radius=0.04];

\draw (0,0)circle(2cm);
\node  [right]  at (2.8,2.8) {$z_k$};
\node  [right]  at (-2.8,2.8) {$-\overline{z}_{k}$};
\node  [right]  at (0,2.8) {$\mathrm{i}\omega_{k}$};
\node  [right]  at (0,1.2) {$\frac{\mathrm{i}}{\omega_{k}}$};
\node  [below]  at (1.2,1.2) {$\overline{z}_{k}^{-1}$};
\node  [below]  at (-1.2,1.2) {$-z_{k}^{-1}$};
\end {tikzpicture}
\end{center}
\caption{The complex $z$-plane showing the discrete spectrums [zeros of scattering data $s_{11}(z)$ (red) in yellow region and those
of scattering data $s_{22}(z)$ (blue) in white region]. The yellow and white regions stand for $D_{+}$ and $D_{-}$, respectively.}
\label{analregion&spectrumsdis}
\end{figure}

We know that the Jost solutions $\Phi_{\pm}(x,t,z)$ satisfy
\begin{equation}
   \Phi_{\pm}(x,t,z)\sim E_{\pm}(z)e^{\mathrm{i}t\theta(x,t,z)\sigma_{3}}.
\end{equation}
For convenience, we introduce the modified Jost solutions $\mu_{\pm}(x,t,z)$ by eliminating the exponential oscillations
\begin{equation}
   \mu_{\pm}(x,t,z)=\Phi_{\pm}(x,t,z)e^{-\mathrm{i}t\theta(x,t,z)\sigma_{3}},
\end{equation}
such that
\begin{equation}
   \lim_{x\to \pm\infty}\mu_{\pm}(x,t,z)=E_{\pm}(z),
\end{equation}
and $\mu_{\pm}$ admit the Volterra type integral equations
\begin{align}
 \mu_{\pm}(x,t;z)=E_{\pm}(z)+\left\{
        \begin{aligned}
    &\int_{\pm\infty}^{x}E_{\pm}(z)e^{\mathrm{i}\lambda(x-y)\sigma_{3}}[E_{\pm}^{-1}(z)\Delta Q_{\pm}(y,t)\mu_{\pm}(y,t,z)]dy, \quad k\neq\pm i,\\
    &\int_{\pm\infty}^{x}[I+(x-y)X_{\pm}]\Delta Q_{\pm}(y,t)\mu_{\pm}(y,t,z)dy, \quad k=\pm i,
    \end{aligned}
        \right.
\label{mu}
\end{align}
where $\Delta Q_{\pm}(x,t)= Q(x,t)-Q_{\pm}$. For convenience, we let $\Sigma_{0}=\Sigma/{\pm \mathrm{i}}$. Similarly to \cite{CJ}, we have the following propositions:
\begin{proposition}\label{analydiff}
    Given $n\in\mathbb{N}_0$, let $q\mp q_{\pm}\in$ $L^{1,n+1}(\mathbb{R})$, $q'\in$ $W^{1,1}(\mathbb{R})$.
    \begin{itemize}
    \item $\mu_{+,1}$ and $\mu_{-,2}$ can be analytically extended to $D_{+}$ and continuously extended to $D_{+}\cup \Sigma_{0}$, $\mu_{-,1}$ and $\mu_{+,2}$ can be analytically extended to $D_{-}$ and continuously extended to $D_{-}\cup \Sigma_{0}$;

    \item The map $q\rightarrow$ $\frac{\partial^n}{\partial z^n}\mu_{\pm,i}(z)$ $(i=1,2, n\geq 0)$ are Lipschitz continuous, specifically, for any $x_0\in\mathbb{R}$, $\mu_{-,1}(z)$ and $\mu_{+,2}(z)$ are continuously differentiable mappings:
        \begin{align}
    &\partial_z^n\mu_{+,1}: \bar{D}_+\setminus\{0, \pm \mathrm{i}\} \rightarrow L^{\infty}_{loc}\{\bar{D}_+\setminus\{0, \pm \mathrm{i}\}, C^1([x_0,\infty),\mathbb{C}^2)\cap W^{1,\infty}([x_0,\infty),\mathbb{C}^2)\}, \\
    &\partial_z^n\mu_{-,2}: \bar{D}_+\setminus\{0, \pm \mathrm{i}\}\rightarrow L^{\infty}_{loc}\{\bar{D}_+\setminus\{0, \pm \mathrm{i}\}, C^1((-\infty,x_0],\mathbb{C}^2)\cap W^{1,\infty}((-\infty,x_0],\mathbb{C}^2)\},
    \end{align}
    $\mu_{+,1}(z)$ and $\mu_{-,2}(z)$ are continuously differentiable mappings:
    \begin{align}
    &\partial_z^n\mu_{-,1}: \bar{D}_-\setminus\{0, \pm \mathrm{i}\} \rightarrow L^{\infty}_{loc}\{\bar{D}_-\setminus\{0, \pm \mathrm{i}\}, C^1((-\infty,x_0],\mathbb{C}^2)\cap W^{1,\infty}((-\infty,x_0],\mathbb{C}^2)\},\\
    &\partial_z^n\mu_{+,2}: \bar{D}_-\setminus\{0, \pm \mathrm{i}\} \rightarrow L^{\infty}_{loc}\{\bar{D}_-\setminus\{0, \pm \mathrm{i}\}, C^1([x_0,\infty),\mathbb{C}^2)\cap W^{1,\infty}([x_0,\infty),\mathbb{C}^2)\}.
    \end{align}
    \item Let $K$ be a compact neighborhood of $\{-\mathrm{i},\mathrm{i}\}$ in $\bar{D}_+\setminus\{0\}$. Set $x^{\pm}=\max\{\pm x,0\}$, then there exists a constant $C$ such that for $z\in K$ we have
    \begin{equation}
    |\mu_{+,1}(z)-(1,\mathrm{i}z^{-1})^{\rm T}|\leq C\langle x^-\rangle e^{C\int_x^\infty \langle y-x\rangle|q-1|dy}\|q-\tilde{q}\|_{L^{1,1}(x,\infty)},
    \end{equation}
    i.e., the map $z\rightarrow \mu_{+,1}(z)$ extends as a continuous map to the points $\pm\mathrm{i}$ with values in $C^1([x_0,\infty),\mathbb{C})\cap W^{1,n}([x_0,\infty),\mathbb{C})$ for any preassigned $x_0\in\mathbb{R}$. Moreover, the map $q\rightarrow \mu_1^+(z)$ is locally Lipschitz continuous from:
    \begin{equation}\label{prop31}
    L^{1,1}(\mathbb{R})\rightarrow L^{\infty}(\bar{D}_+\setminus\{0\},C^1([x_0,\infty),\mathbb{C})\cap W^{1,\infty}([x_0,\infty),\mathbb{C}).
    \end{equation}
    Analogous statements hold for $\mu_{+,2}$ and for $\mu_{-,j}$ $(j=1,2)$.
    Furthermore, the maps $z\rightarrow \partial_z^n\mu_{+,1}(z)$ and $q\rightarrow \partial_z^n\mu_{+,1}(z)$ also satisfy:
    \begin{equation}
    |\partial_z^n\mu_{+,1}(z)|\leq F_n\left[(1+|x|)^{n+1}\| q-1\|_{L^{1,n+1}(x,\infty)}\right],\hspace{0.5cm}z\in K.
    \end{equation}
    \end{itemize}
    \end{proposition}
The asymptotic behavior of $\mu_{\pm, j}$, $j=1,2$ could be described by following proposition.
\begin{proposition}
    Suppose that $q\mp q_{\pm}\in$ $L^{1,n+1}(\mathbb{R})$ and $q'\in$ $W^{1,1}(\mathbb{R})$. Then as $z\rightarrow\infty$, we have
    \begin{align}
    &\mu_{\pm,1}(z)=e_1-\frac{\mathrm{i}}{z}\left(\begin{array}{c}
                                    \int^{x}_{\pm\infty}[\sigma q(y,t)q(-y,-t)+1]dy\\
                                    \sigma q(-x,-t)
                                    \end{array}\right)+\mathcal{O}(z^{-2}),\\
    &\mu_{\pm,2}(z)=e_2+\frac{\mathrm{i}}{z}\left(\begin{array}{c}
                                    q(x,t)\\
                                    \int^{x}_{\pm\infty}[\sigma q(y,t)q(-y,-t)+1]dy
                                    \end{array}\right)+\mathcal{O}(z^{-2}),
    \end{align}
 and as $z\rightarrow0$, we have
     \begin{align}
    &\mu_{\pm,1}(z)=\mathrm{i}\frac{q_{\pm}}{z}e_2+\mathcal{O}(1),\\
    &\mu_{\pm,2}(z)=\mathrm{i}\frac{q_{\pm}}{z}e_1+\mathcal{O}(1),
    \end{align}
\end{proposition}
where $e_1=(1,0)^{\rm T}$, $e_2=(0,1)^{\rm T}$.

It follows  that $\Phi_{\pm}(x,t;z)$ are fundamental solutions of Lax pair \eqref{lax pair} as $z\in \Sigma_{0}$, thus there exists a constant scattering matrix $S(z)=(s_{ij}(z))_{2\times2}$ (independence of    $x$ and $t$) such that
\begin{equation}
   \Phi{+}(x,t,z)=\Phi_{-}(x,t,z)S(z), \quad z\in \Sigma_{0},
\end{equation}
where $s_{ij}(z)$ can be expressed as
\begin{equation}\label{sij}
\begin{split}
& s_{11}(z)=\frac{{\rm det}(\Phi_{+,1},\Phi_{-,2})}{1+z^{-2}}, \quad s_{12}(z)=\frac{{\rm det}(\Phi_{+,2},\Phi_{-,2})}{1+z^{-2}},\\
& s_{21}(z)=\frac{{\rm det}(\Phi_{-,1},\Phi_{+,1})}{1+z^{-2}}, \quad s_{22}(z)=\frac{{\rm det}(\Phi_{-,1},\Phi_{+,2})}{1+z^{-2}}.
\end{split}
\end{equation}

 By calculating, we get the symmetry relations of the Jost solutions and scattering matrix for variables and isospectral parameter.
 \begin{proposition}\label{symmetrypro}
The symmetries for Jost solutions $\Phi_{\pm}(x,t,z)$ and scattering matrix $S(z)$ in $z\in \Sigma$ are listed as follows:
    \begin{itemize}
    \item The first symmtry
      \begin{equation}
      \Phi_{\pm}(x,t,z)=\sigma_{4}\overline{\Phi_{\mp}(-x,-t,-\bar{z})}\sigma_{4}, \quad S(z)=\sigma_{4}[\overline{S(-\bar{z})}]^{-1}\sigma_{4},
      \end{equation}
      where  $\sigma_{4}$ is defined as
      \begin{align}
       \sigma_{4}=\left\{
        \begin{aligned}
      &\sigma_{1}, \quad \sigma=-1,\\
      &\sigma_{2}, \quad \sigma=1.
      \end{aligned}
      \right.
      \end{align}
    \item The second symmetry
      \begin{equation}
      \Phi_{\pm}(x,t,z)=\overline{\Phi_{\pm}(-x,-t,-\bar{z})}, \quad S(z)=\overline{S(-\bar{z})}.
      \end{equation}
    \item The third symmetry
      \begin{equation}
      \Phi_{\pm}(x,t,z)=\frac{\mathrm{i}}{z}\Phi_{\pm}(x,t,-z^{-1})\sigma_{3}Q_{\pm}, \quad S(z)=(\sigma_{3}Q_{-})^{-1}S(-z^{-1})(\sigma_{3}Q_{+}).
      \end{equation}
    \end{itemize}
\end{proposition}
We use the scattering coefficients $s_{ij}(z), i,j=1,2$ to define reflection coefficients
\begin{equation}
   \rho(z)=\frac{s_{21}(z)}{s_{11}(z)}, \quad \tilde{\rho}(z)=\frac{s_{12}(z)}{s_{22}(z)}, \quad  z\in \Sigma.
\end{equation}
The following proposition provides some essential properties for $s_{ij}(z), i,j=1,2$ and $\rho(z), \tilde{\rho}(z)$.
\begin{proposition} Let $q\mp q_{\pm}\in$ $L^{1,n+1}(\mathbb{R})$ and  $q'\in$ $W^{1,1}(\mathbb{R})$, then
    \begin{itemize}
        \item [{\rm 1)}] For $z\in\left\{z\in\mathbb{C}: |z|=1\right\}/\{\pm\mathrm{i}\}$,
          \begin{equation}
          |\rho(z)\tilde{\rho}(z)|\leq 1-|s_{11}(z)|^{-2}<1.
          \end{equation}
        \item [{\rm 2)}] $s_{ij}(z), i,j=1,2$ and the reflection coefficient $\rho(z), \tilde{\rho}(z)$ satisfy the symmetries
        \begin{equation}
         s_{11}(z)=-\sigma s_{22}(-z^{-1}), \quad s_{12}(z)=-\sigma s_{21}(-z^{-1}), \quad  \tilde{\rho}(z)=\rho(-z^{-1}).
        \end{equation}
        \item [{\rm 3)}]The scattering data have the asymptotics
        \begin{align}
        &\lim_{z\rightarrow\infty}(s_{11}(z)-1)z=i\int_{\mathbb{R}}\left[\sigma q(y,t)q(-y,-t)-1\right]dy, \label{32}\\
        &\lim_{z\rightarrow0}s_{11}(z)=-\sigma, \label{33}
        \end{align}
        \begin{align}
        &|s_{21}(z)|=\mathcal{O}(|z|^{-2}),\hspace{0.5cm} \text{as }|z|\rightarrow\infty,\label{34}\\
        &|s_{21}(z)|=\mathcal{O}(|z|^{2}),\hspace{0.5cm} \text{as }|z|\rightarrow0, \label{35}
        \end{align}
         So that
    \begin{align}\label{36}
    \rho(z), \tilde{\rho}(z)\sim z^{-2}, \hspace{0.2cm}|z|\rightarrow\infty;\hspace{0.5cm}\rho(z), \tilde{\rho}(z)\sim 0,\hspace{0.2cm}|z|\rightarrow0.
    \end{align}
    \item [{\rm 4)}] Although $s_{11}(z)$ and $s_{21}(z)$ have singularities at points $\pm \mathrm{i}$, we can claim that the reflection
coefficient $\rho(z), \tilde{\rho}(z)$ remain bounded at $z=\pm \mathrm{i}$ and $|\rho(\pm \mathrm{i})|=1, |\tilde{\rho}(\pm \mathrm{i})|=1$. In fact, by direct calculation, we obtain
    \begin{equation}
    s_{11}(z)=\frac{\mp s^{\pm}}{z\mp \mathrm{i}}+\mathcal{O}(1), \quad s_{21}(z)=\frac{-\sigma s^{\pm}}{z\mp \mathrm{i}}+\mathcal{O}(1),
    \end{equation}
    \begin{equation}\label{rhob2}
   \lim_{z\to \pm\mathrm{i}}\rho(z)=\lim_{z\to \pm\mathrm{i}}\tilde{\rho}(z)=\pm\sigma,
    \end{equation}
    where $s^{\pm}=\frac{1}{2\mathrm{i}}{\rm det}(\Phi_{+,1}(\pm\mathrm{i}),\Phi_{-2}(\pm\mathrm{i}))$.
    \end{itemize}
\end{proposition}

Since one cannot exclude the possibilities of zeros for $s_{11}(z)$ and $s_{22}(z)$ along $\Sigma$. To solve the Riemann-Hilbert problem in the
inverse process, we only consider the potentials without spectral singularities, i.e., $s_{11}(z)\neq0$, $s_{22}(z)\neq0$ for $z\in\Sigma$.

The next proposition shows that, given data $q_{0}(x)$ with sufficient smoothness and decay properties, the reflection coefficients will also be smooth and decaying.
\begin{proposition}
For given $q\mp q_{\pm}\in L^{1,2}(\mathbb{R})$ and $q'\in$ $W^{1,1}(\mathbb{R})$, we then have $\rho(z),\tilde{\rho}(z)\in H^{1}(\Gamma)$, where $\Gamma$ is defined in \eqref{Gamma}.
\end{proposition}
\begin{proof}
Proposition \ref{analydiff} and \eqref{sij} indicate that $s_{11}(z)$ and $s_{21}(z)$ are continuous to $\Sigma/\{0,\pm\mathrm{i}\}$. Noticing $s_{11}(z)\neq0$, $s_{22}(z)\neq0$ for $z\in\Sigma$, then $\rho(z)$ is continuous to $\Sigma/\{0,\pm\mathrm{i}\}$. From \eqref{rhob1} and \eqref{rhob2}, we know that $\rho(z)$ is bounded in the small neighborhood of $\{0,\pm\mathrm{i}\}$ and $\rho(z)\in L^{1}(\Gamma)\cap L^{2}(\Gamma)$. Next, we just need to  prove that $\rho'(z)\in L^{2}(\Gamma)$. For $\delta_{0}>0$ sufficiently small, from Proposition
\ref{analydiff}, the maps
\begin{equation}
q\rightarrow {\rm det}(\Phi_{+,1},\Phi_{-,2}), \quad q\rightarrow {\rm det}(\Phi_{-,1},\Phi_{+,1})
\end{equation}
are locally Lipschitz maps from
\begin{equation}\label{40}
\{q: q\in L^{1,n+1}(\mathbb{R}), q'\in W^{1,1}(\mathbb{R})\}\rightarrow W^{n,\infty}(\mathbb{R}/(-\delta_{0},\delta_{0})), n\geq0.
\end{equation}
Actually, $q\rightarrow \Phi_{+,1}(z,0)$ is, by Proposition \ref{analydiff}, a locally Lipschitz map with values in $W^{n,\infty}(\bar{D}_+/(0,\delta_{0}, \mathbb{C}^2))$. For $q\rightarrow \Phi_{-,2}(z,0)$ and $q\rightarrow \Phi_{-,1}(z,0)$ the same is true.
These and  \eqref{32} imply that $q\rightarrow\rho(z)$ is a locally Lipschitz map from the domain in \eqref{40} into
\begin{equation}
W^{n,\infty}(I_{\delta_{0}})\cap H^{n}(I_{\delta_{0}}),
\end{equation}
where $I_{\delta_{0}}=\Gamma/(-\delta_{0},\delta_{0})\cup dist(\pm\mathrm{i};\delta_{0})$. Now fix $\delta_{0}$ sufficiently small such that the $3$ intervals $dist(z;\pm \mathrm{i})\leq\delta_{0}$ and $|z-0|\leq\delta_{0}$ have no intersection. In the complement of their union
\begin{equation}
|\partial^{j}_{z}\rho(z)|\leq C_{\delta_{0}}\langle z\rangle^{-1}, \quad j=0,1.
\end{equation}
Let $|z-\mathrm{i}|<\delta_{0}$, then using the definition of $s_{+}$ we have
\begin{equation}
\rho(z)=\frac{s_{21}(z)}{s_{11}(z)}=\frac{{\rm det}[\Phi_{-,1},\Phi_{+,1}]}{{\rm det}[\Phi_{+,1},\Phi_{-,2}]}=\frac{\int_{\mathrm{i}}^{z}F(s)ds-\sigma s^{+}}{\int_{\mathrm{i}}^{z}G(s)ds- s^{+}},
\end{equation}
where
\begin{equation}
F(s)=\partial_{s}{\rm det}[\Phi_{-,1}(s)\Phi_{+,1}(s)], \quad G(s)=\partial_{s}{\rm det}[\Phi_{+,1},\Phi_{-,2}].
\end{equation}
If $s_{+}\neq0$ then it is clear from the above formula that $\rho'(z)$ exist and is bounded near $\mathrm{i}$.

If $s_{+}=0$, then $z=\mathrm{i}$ is not the pole of $s_{11}(z)$ and $s_{21}(z)$, so that $s_{11}(z)$ and $s_{21}(z)$ are continuous
at $z=\mathrm{i}$, then
\begin{equation}
\rho(z)=\frac{\int_{\mathrm{i}}^{z}F(s)ds}{\int_{\mathrm{i}}^{z}G(s)ds}.
\end{equation}
From \eqref{sij}, we have
\begin{equation}\label{s11}
(1+z^2)s_{11}(z)=z^2{\rm det}[\Phi_{+,1},\Phi_{-,2}].
\end{equation}
Since $s_{+}=0$ implies that ${\rm det}[\Phi_{+,1}(\mathrm{i}),\Phi_{-,2}(\mathrm{i})]=0$, differentiating \eqref{s11} at $z=\mathrm{i}$ we get
\begin{equation}
2\mathrm{i}s_{11}(\mathrm{i})=-\partial_{z}{\rm det}[\Phi_{+,1}(z),\Phi_{-,2}(z)]|_{z=\mathrm{i}}=-\mathrm{i}G(\mathrm{i}).
\end{equation}
With $|s_{11}(\mathrm{i})|^2=1+|s_{21}(\mathrm{i})|^2>1$, we have $G(\mathrm{i})\neq0$. It follows that the derivative $\rho'(z)$ is
bounded around $\mathrm{i}$. The same proof holds at $z=-\mathrm{i}$ and $z=0$. It follows that $\rho'(z)\in L^{2}(\Gamma)$.

According to the symmetry $\tilde{\rho}(z)=\rho(-z^{-1})$, we have $\tilde{\rho}(z)\in H^{1}(\Gamma)$.
\end{proof}
\begin{corollary}
    For  given $q\mp q_{\pm}\in L^{1,2}(\mathbb{R})$, $q'\in W^{1,1}(\mathbb{R})$,  we then have $r(z)\in H^{1,1}(\Gamma)$.
\end{corollary}
\begin{proof}
    Since $H^{1,1}(\mathbb{R})=L^{2,1}(\mathbb{R})\cap H^{1}(\mathbb{R})$, what we need to prove is $r\in L^{2,1}(\mathbb{R})$.
    With \eqref{36}, we can see that
    \begin{equation}
        \vert z\vert^2 r^{2}(z) \sim  \vert z\vert^{-2}, \quad |z|\rightarrow \infty
    \end{equation}
    Thus
    \begin{equation}
        \int_{\Gamma}\vert \langle z\rangle r(z)\vert^2 <\infty,
    \end{equation}
    which help us obtain the result.
\end{proof}

Suppose that $s_{11}(z)$ has $N_{1}$ and $N_{2}$ simple zeros, respectively, in $D_{+}\cap \{z\in \mathbb{C}: Rez>0\}$ denoted by $z_{k}$, $k=1,2,\cdot\cdot\cdot, N_{1}$ and in $D_{+}\cap \{z\in \mathbb{C}: Rez=0\}$ denoted by $\mathrm{i}\omega_{k}$, $k=1,2,\cdot\cdot\cdot, N_{2}$. It follows from the symmetry relations of the scattering coefficients that
\begin{equation}
\begin{split}
& s_{11}(z_{k})=s_{11}(-\bar{z}_{k})=s_{22}(-z_{k}^{-1})=s_{22}(\bar{z}_{k}^{-1}), \quad k=1,2,\cdot\cdot\cdot, N_{1},\\
& s_{11}({\mathrm{i}\omega_{k}})=s_{22}(\mathrm{i}\omega_{k}^{-1}), \quad k=1,2,\cdot\cdot\cdot, N_{2}.
\end{split}
\end{equation}
It is convenient to define that
\begin{align}
       \eta_{k}=\left\{
        \begin{aligned}
      &z_{k}, \quad k=1,2,\cdot\cdot\cdot, N_{1},\\
      &-\bar{z}_{k-N_{1}}, \quad k=N_{1}+1,N_{1}+2,\cdot\cdot\cdot, 2N_{1}.\\
      &\mathrm{i}\omega_{k-2N_{1}}, \quad k=2N_{1}+1,2N_{1}+2,\cdot\cdot\cdot, 2N_{1}+N_{2},
      \end{aligned}
      \right.
\end{align}
and
\begin{equation}
\hat{\eta}_{k}=-\eta_{k}^{-1}, \quad k=1,2,\cdot\cdot\cdot, 2N_{1}+N_{2}.
\end{equation}
Thus, the discrete spectrum is given by
\begin{equation}
\mathcal{Z}\cup\hat{\mathcal{Z}}=\{\eta_{k},\hat{\eta}_{k}\}_{k=1}^{2N_{1}+N_{2}},
\end{equation}
and the distribution of $\mathcal{Z}\cup\hat{\mathcal{Z}}$ on the $z$-plane is shown in Figure \ref{analregion&spectrumsdis}.

Given $z_{0}\in\mathcal{Z}$, it follows from the Wronskian representations and $s_{11}(z_{0})=0$ that $\Phi_{+,1}(x,t,z_{0})$ and $\Phi_{-,2}(x,t,z_{0})$ are linearly dependent; Given $z_{0}\in\hat{\mathcal{Z}}$, it follows from the Wronskian representations $s_{22}(z_{0})=0$ and that $\Phi_{+,2}(x,t,z_{0})$ and $\Phi_{-,1}(x,t,z_{0})$ are linearly dependent. For convenience, we denote the proportional coefficient in the following definition of $b[z_{0}]$ by
$\frac{\Phi_{+,1}(x,t,z_{0})}{\Phi_{-,2}(x,t,z_{0})}$ or $\frac{\Phi_{+,2}(x,t,z_{0})}{\Phi_{-,1}(x,t,z_{0})}$ according to the region $z_{0}$ belongs to. Let
\begin{align}
 b[z_{0}]=
 \left\{
    \begin{aligned}
    &\frac{\Phi_{+,1}(x,t,z_{0})}{\Phi_{-,2}(x,t,z_{0})},  \quad  z_{0}\in\mathcal{Z},\\
    & \frac{\Phi_{+,2}(x,t,z_{0})}{\Phi_{-,1}(x,t,z_{0})}, \quad  z_{0}\in\hat{\mathcal{Z}},
    \end{aligned}
        \right.
    \quad
    \left\{
    \begin{aligned}
    &A[z_{0}]=\frac{b[z_{0}]}{s^{'}_{11}(z_{0})}, \quad  z_{0}\in\mathcal{Z},\\
    &A[z_{0}]=\frac{b[z_{0}]}{s^{'}_{22}(z_{0})}, \quad  z_{0}\in\hat{\mathcal{Z}}.
    \end{aligned}
        \right.
\end{align}
\begin{proposition}
For the given $z_{0}\in\mathcal{Z}\cup\hat{\mathcal{Z}}$, there exist three relations for $b[z_{0}]$, $s'_{11}(z_{0})$ and $s'_{22}(z_{0})$:
    \begin{itemize}
    \item The first relation
      \begin{equation}
       b[z_{0}]=-\frac{\sigma}{\overline{b[-\bar{z}_{0}]}}, \quad s'_{11}(z_{0})=-\overline{s'_{11}(-\bar{z}_{0})}, \quad s'_{22}(z_{0})=-\overline{s'_{22}(-\bar{z}_{0})}.
      \end{equation}
    \item The second relation
      \begin{equation}
       b[z_{0}]=-\overline{b[-\bar{z}_{0}]}, \quad s'_{11}(z_{0})=-\overline{s'_{11}(-\bar{z}_{0})}, \quad s'_{22}(z_{0})=-\overline{s'_{22}(-\bar{z}_{0})}.
      \end{equation}
    \item The third relation
      \begin{equation}
       b[z_{0}]=-\sigma b[-z_{0}^{-1}], \quad s'_{11}(z_{0})=-\sigma z_{0}^{-2}s'_{22}(-z_{0}^{-1}).
      \end{equation}
    \end{itemize}
\end{proposition}
From the first relation, one concludes that imaginary discrete spectrum $\mathrm{i}\omega_{k}$ exists if and only if $\sigma=-1$. That is to say that as $\sigma=1$, one has $N_{2}=0$. We give the following relation among discrete spectrum $\mathcal{Z}\cup\hat{\mathcal{Z}}$.
\begin{proposition}
The relations for $b[\cdot]$ and $A[\cdot]$ in $\mathcal{Z}\cup\hat{\mathcal{Z}}$ are given by
\begin{equation}
\begin{split}
& b[\eta_{k}]=\overline{b[-\bar{\eta}_{k}]}=-\sigma b[\hat{\eta}_{k}]=-\sigma b[\bar{\eta}_{k}^{-1}], \quad b^{2}[\eta_{k}]=1\\
& s'_{11}(\eta_{k})=-\overline{s'_{11}(-\bar{\eta}_{k})}=-\sigma \hat{\eta}_{k}^{2}s'_{22}(\hat{\eta}_{k})=\sigma \hat{\eta}_{k}^{2}s'_{22}(\bar{\eta}_{k}^{-1}).
\end{split}
\end{equation}
As $\sigma=-1$, one has
\begin{equation}
b[\mathrm{i} \omega_{k}]=b[\mathrm{i} \omega_{k}^{-1}], \quad s'_{11}(\mathrm{i} \omega_{k})=-\omega_{k}^{2}s'_{22}(\mathrm{i} \omega_{k}^{-1}).
\end{equation}
Then,
\begin{equation}
A[\hat{\eta}_{k}]=\hat{\eta}_{k}^{2}A[\eta_{k}].
\end{equation}
\end{proposition}

\subsection{A RH problem}
Define a sectionally meromorphic matrix as follows
\begin{equation}
    m(z)=m(x,t;z):=\left\{
        \begin{aligned}
        \left(\frac{\mu_{+,1}(x,t;z)}{s_{11}(z)}, \mu_{-,2}(x,t;z)\right), \quad z\in D_{+}, \\
        \left(\mu_{-,1}(x,t;z), \frac{\mu_{+,2}(x,t;z)}{s_{22}(z)}\right), \quad z\in D_{-},
        \end{aligned}
        \right.
\end{equation}
and
\begin{equation}
   m_{\pm}(x,t,z)=\lim_{z'\to z \atop z\in D_{\pm}}m(x,t,z') \quad z\in\Sigma,
\end{equation}
the multiplicative matrix Riemann-Hilbert problem can be proposed as follows.
\begin{RHP}\label{rhp0}
Find a $2\times 2$ matrix-valued function $m(x,t;z)$ such that
\begin{itemize}
\item[*] Analyticity: $m(z)$ is analytical in $\mathbb{C}\backslash (\Sigma\cup\mathcal{Z}\cup\hat{\mathcal{Z}})$ and has simple poles in $\mathcal{Z}\cup\hat{\mathcal{Z}}=\{\eta_{k}, \bar{\eta}_{k}\}_{k=1}^{2N_{1}+N_{2}}$.
\item[*] Jump relation: $m_{+}(x,t,z)=m_{-}(x,t,z)v(z)$, where
\begin{equation}\label{rhp0jump}
    v(z)=\begin{bmatrix} 1-\rho(z)\tilde{\rho}(z) & -\tilde{\rho}(z)e^{2\mathrm{i}t\theta(x,t,z)} \\ \rho(z)e^{-2\mathrm{i}t\theta(x,t,z)} & 1 \end{bmatrix}, \quad z\in\Sigma.
\end{equation}
\item[*] Asymptotic behavior:
\begin{align}
    &m(x,t;z)=I+\mathcal{O}(z^{-1}), \quad  z\rightarrow\infty,\\
    &m(x,t;z)=\frac{\mathrm{i}}{z}\sigma_3Q_{-}+\mathcal{O}(1), \quad  z\rightarrow 0.
\end{align}
\item[*]Residue conditions
\begin{align}
&\underset{z=\eta_{k}}{\rm Res}m(z)=\lim_{z\rightarrow\eta_{k}}m(z)\begin{bmatrix} 0 & 0 \\ A[\eta_{k}] e^{-2\mathrm{i}t\theta(x,t,\eta_{k})} & 0 \end{bmatrix},\label{rhp0resa}\\
&\underset{z=\hat{\eta}_{k}}{\rm Res}m(z)=\lim_{z\rightarrow\hat{\eta}_{k}}m(z)\begin{bmatrix} 0 & A[\hat{\eta}_{k}]e^{2\mathrm{i}t\theta(,x,t,\hat{\eta}_{k})} \\ 0 & 0 \end{bmatrix},\label{rhp0resb}
\end{align}
where $\theta(x,t,z)=\lambda(z)[\frac{x}{t}+(4k^{2}(z)-2)]$.
\end{itemize}
\end{RHP}
The potential $q(x,t)$ is found by the reconstruction formula
\begin{equation}\label{resconstructm}
q(x,t)=-i(m_1)_{12}=-i\lim_{z\rightarrow\infty}(zm)_{12},
\end{equation}
where $m_1$ appears in the expansion of $m=I+z^{-1}m_1+O(z^{-2})$ as $z\rightarrow\infty$.
\hspace*{\parindent}

\section{Distribution of Saddle Points and Signature Table }\label{disphasepoint}
We notice that the long-time asymptotic behavior of RHP \ref{rhp0} is influenced by the growth and decay of the exponential
function
\begin{equation}\label{phasefunc}
e^{\pm 2it\theta}, \quad \theta(z)=\frac{1}{2}\left(z+\frac{1}{z}\right)\left[\frac{x}{t}-2+\left(z-\frac{1}{z}\right)^2\right],
\end{equation}
which not only appear in jump matrix $v(z)$ but also in the residue condition.  Based on this observation, we shall make analysis for the real part of
 $\pm 2it\theta$ to ensure the exponential decaying property. Let $\xi=\frac{x}{t}$, we consider the stationary phase points and the real part of
$2\mathrm{i}t\theta$:
\begin{equation}\label{theta'}
\theta'(z)=-\frac{(1-z^{2})(3z^{4}+\xi z^{2}+3)}{2z^{4}},
\end{equation}
\begin{equation}\label{Re 2itheta}
{\rm Re}[2\mathrm{i}t\theta]=-2t{\rm Im}\theta=-t{\rm Im}z(1-|z|^{-2})[\xi-3+(1+|z|^{-2}+|z|^{-4})(3{\rm Re}^{2}z-{\rm Im}^{2}z)].
\end{equation}
From Eq. \eqref{theta'}, we find six stationary phase points of $\theta(z)$:
\begin{equation}\label{Re 2itheta}
\pm 1, \quad \pm\frac{\sqrt{-\xi\pm\sqrt{-36+\xi^{2}}}}{\sqrt{6}}.
\end{equation}
Except to $z=\pm 1$, there are also four stationary phase points, whose distribution depends on different $\xi$ is as follows:
    \begin{itemize}
       \item[i.]For $\xi<-6$, the four phase points are located on real axis $\mathbb{R}$ corresponding to
         Figure \ref{theta'xiaoyu}. Among them, two are inside the unit circle and the other two are outside the unit circle;
        \item[ii.]For $-6<\xi<6$, the four phase points are all located on the unit circle and they are symmetrical to each other, which is
            corresponded to Figure \ref{theta'jieyu}. {\bf  We will mainly discuss this case in the present paper};
        \item[iii.] For $\xi>6$, the four phase points are located on $\mathbb{R}$ the imaginary axis $\mathrm{i}\mathbb{R}$ corresponding to \textmd{Figure \ref{theta'dayu}}. The two of them are inside the unit circle and the other two are outside the unit circle;.
    \end{itemize}
Moreover, the decaying regions of ${\rm Re}(2\mathrm{i}t\theta)$ are shown in Figure \ref{figtheta}.
\begin{figure}[htbp]
	\centering
	\subfigure[]{\includegraphics[width=0.3\linewidth]{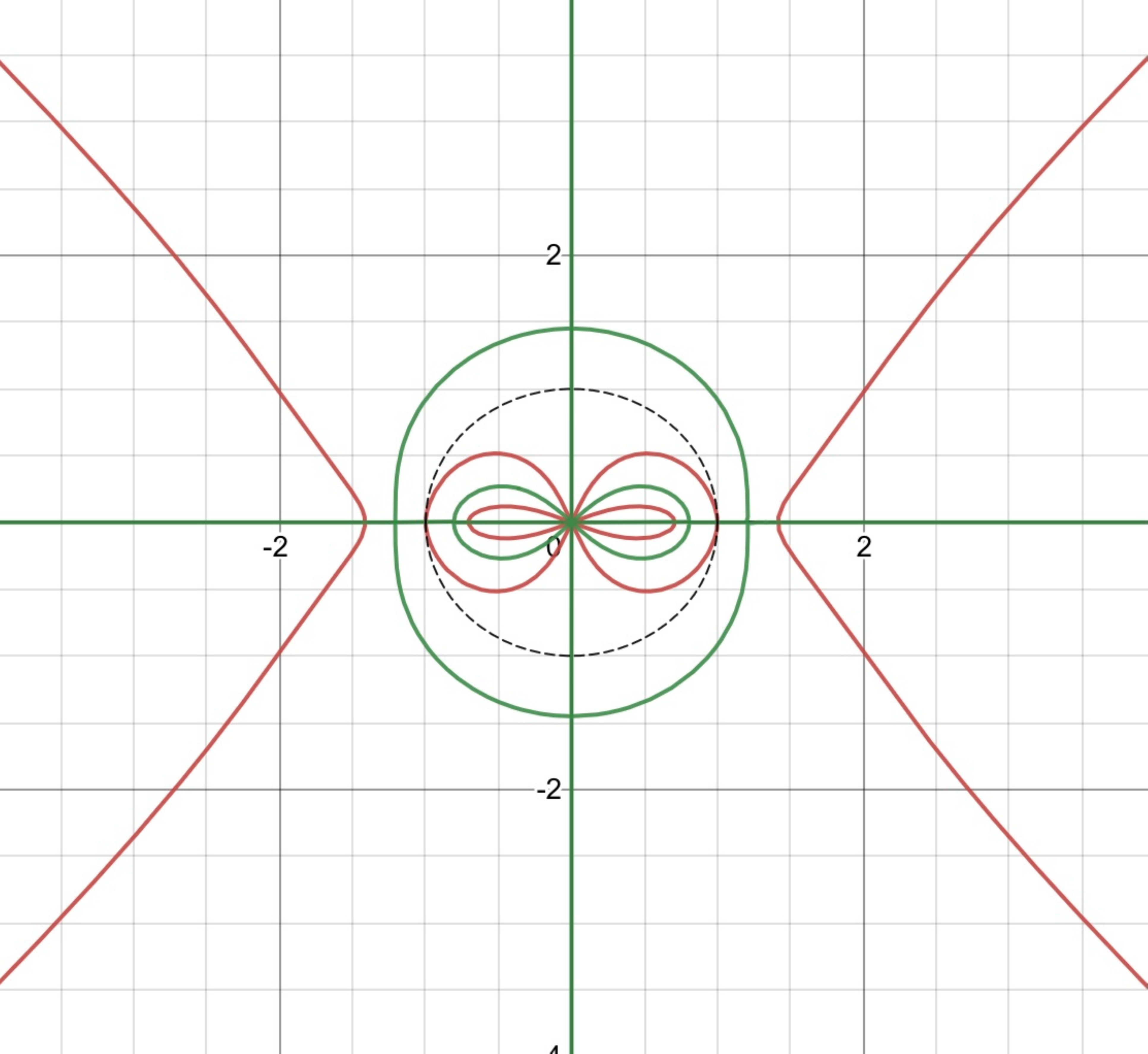}\hspace{0.5cm}\
	\label{theta'xiaoyu}}
	\subfigure[]{\includegraphics[width=0.3\linewidth]{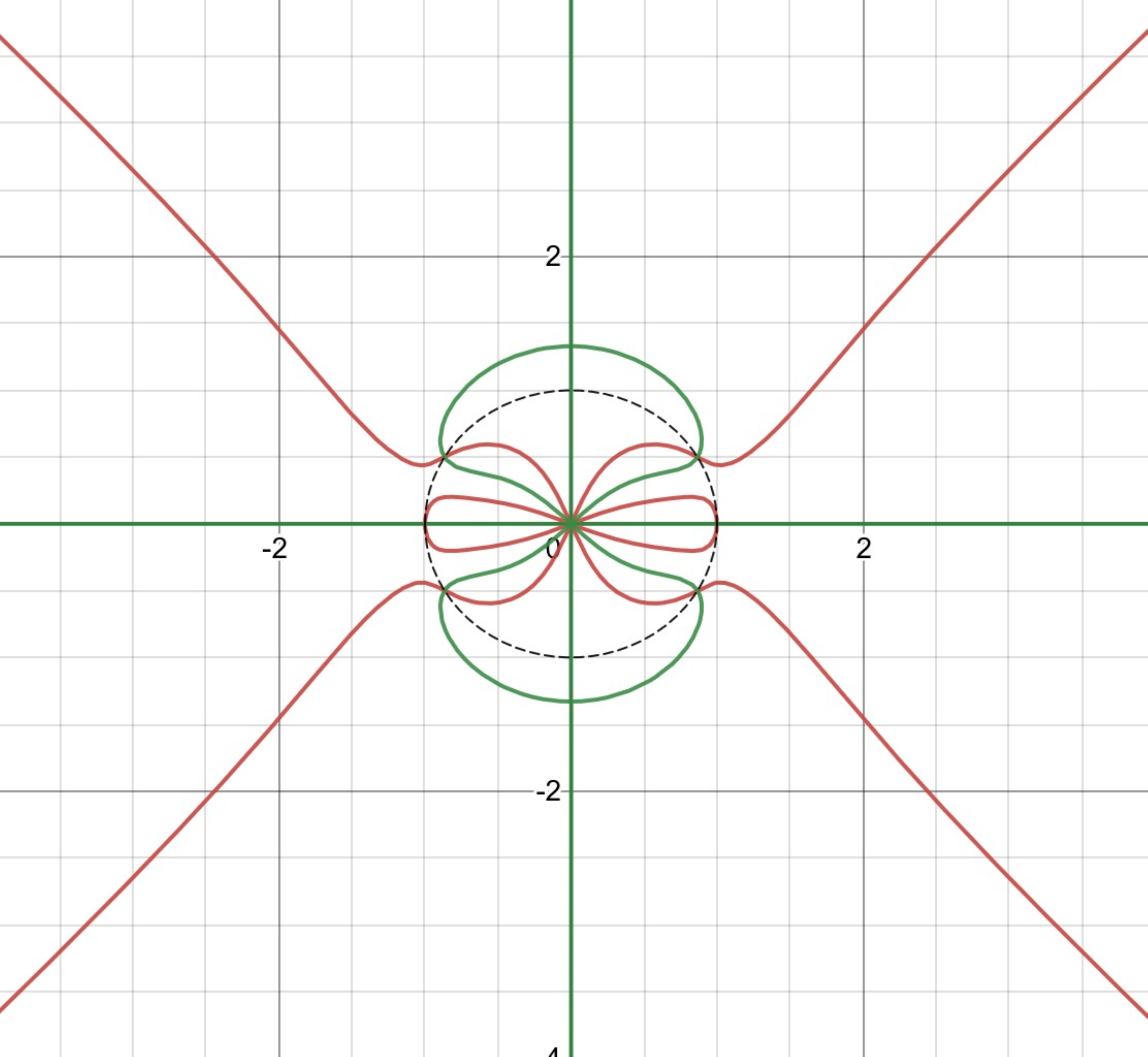}
	\label{theta'jieyu}}	
	\subfigure[]{\includegraphics[width=0.3\linewidth]{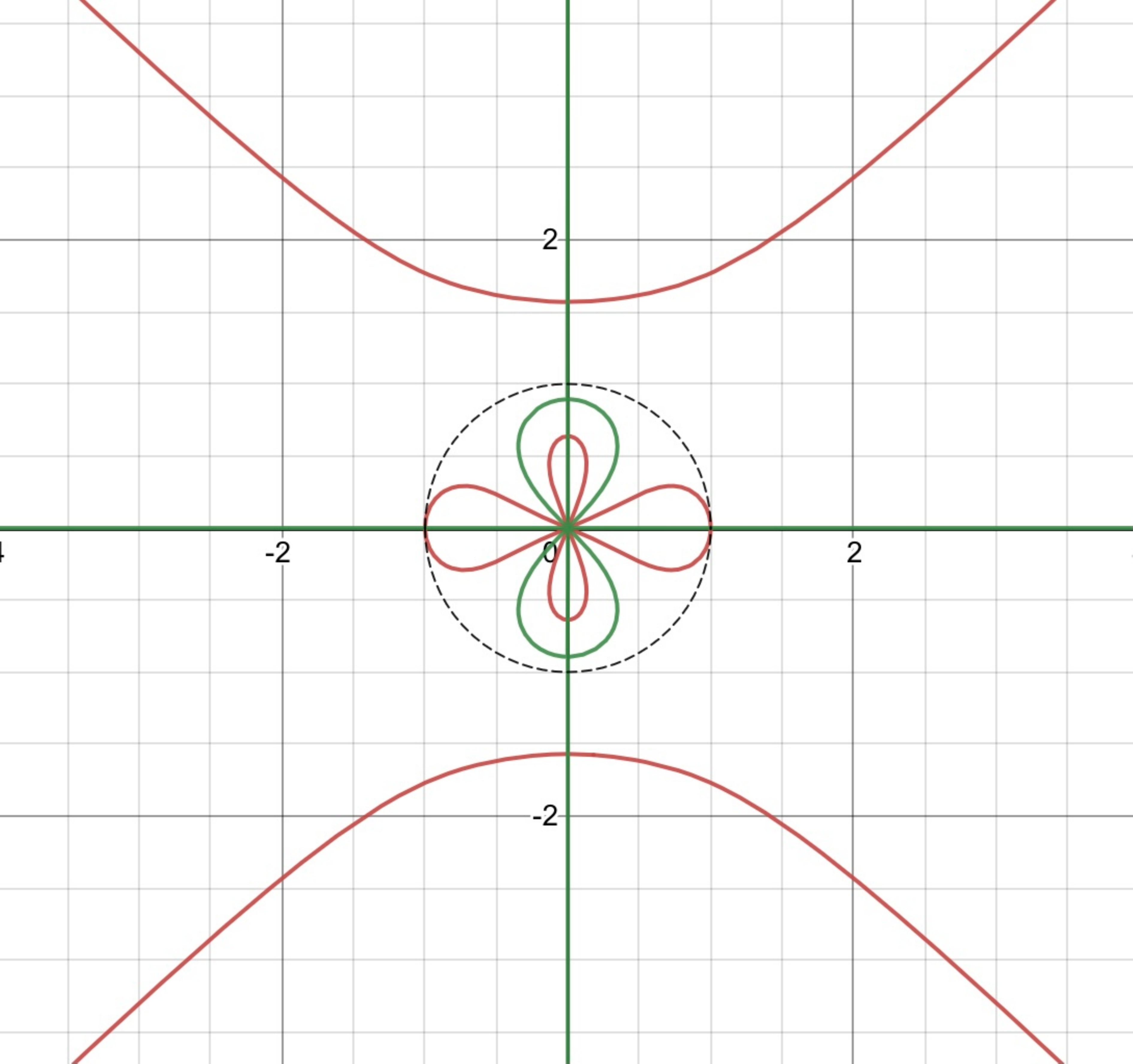}
	\label{theta'dayu}}
	\caption{\footnotesize Plots of the distributions for saddle points:
    $\textbf{(a)}$ $\xi<-6$,
    $\textbf{(b)}$ $-6<\xi<6$,
    $\textbf{(c)}$ $\xi>6$. The black dotted curve is an unit circle. The red curve shows the ${\rm Re} \theta'(z)=0$, and the green curve shows the ${\rm Im} \theta'(z)=0$.
     The intersection points are the saddle points which express $\theta'(z)=0$.}
	\label{figsaddle}
\end{figure}
\begin{figure}[htbp]
	\centering
	\subfigure[]{\includegraphics[width=0.3\linewidth]{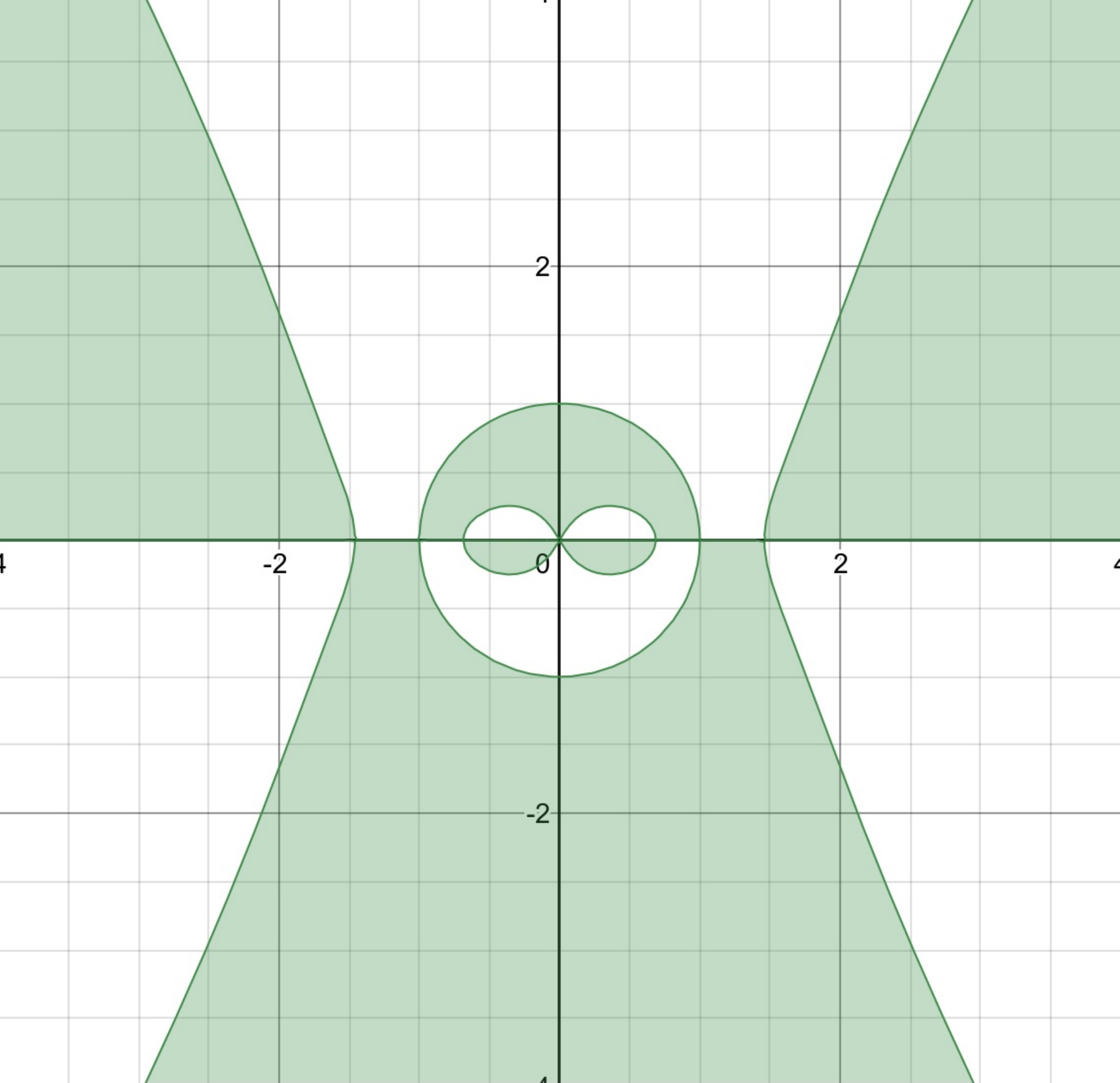}\hspace{0.5cm}\
	\label{Re2ithetaxiaoyu}}
	\subfigure[]{\includegraphics[width=0.3\linewidth]{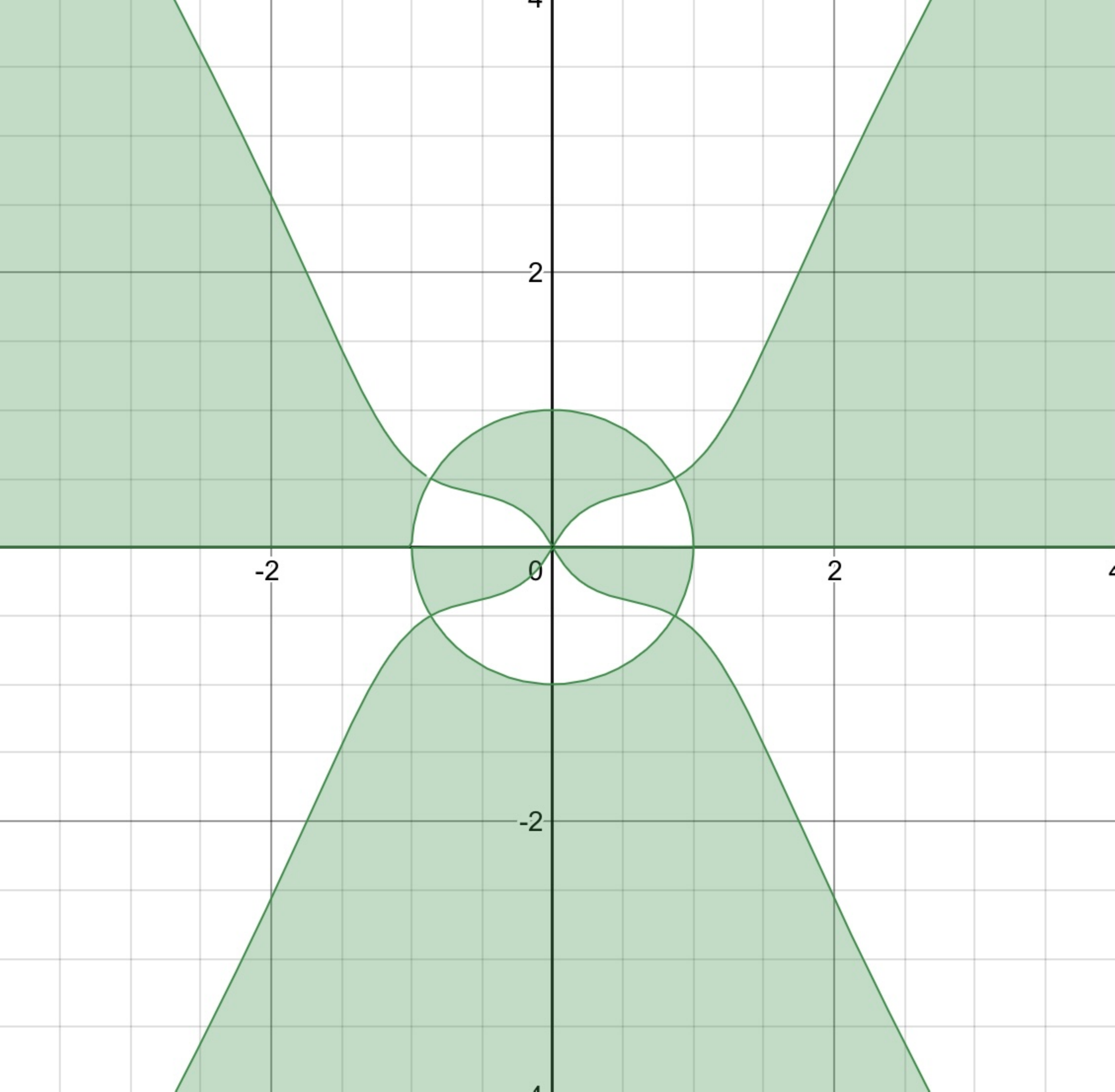}
	\label{Re2ithetajieyu}}	
	\subfigure[]{\includegraphics[width=0.3\linewidth]{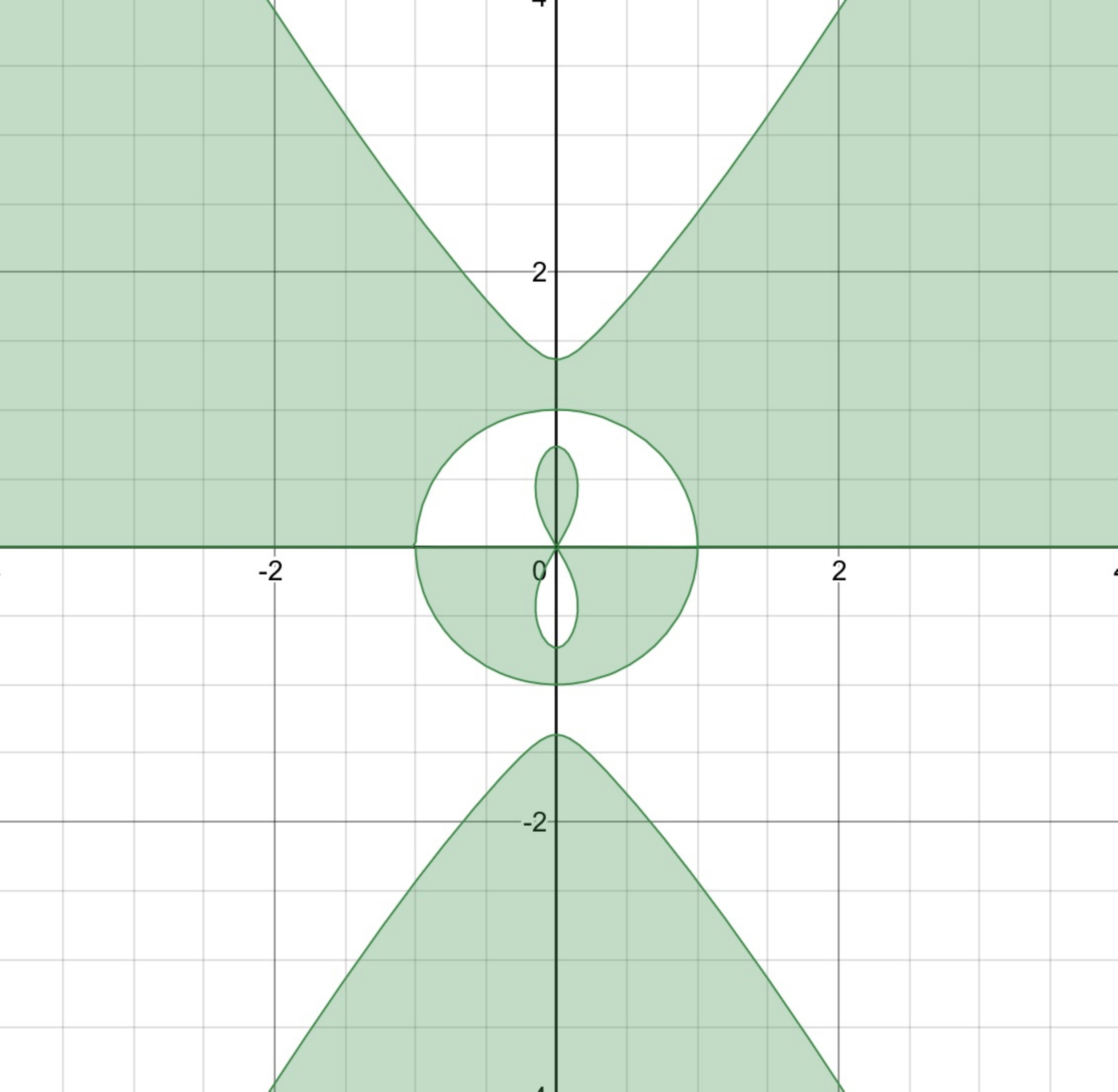}
	\label{Re2ithetadayu}}
	\caption{\footnotesize Signature table of ${\rm Re}(2it\theta)$ with different $\xi$:
    $\textbf{(a)}$ $\xi<-6$,
    $\textbf{(b)}$ $-6<\xi<6$,
    $\textbf{(c)}$   $\xi>6$. ${\rm Re}(2it\theta)<0$ in the green region and  ${\rm Re}(2it\theta)>0$ in
     the white region. In other words, $|e^{2it\theta}|\rightarrow 0$  as $t\rightarrow\infty$ in the green region and  $|e^{-2it\theta} |\rightarrow 0$ as $t\rightarrow\infty$  in the white region. Moreover, ${\rm Re}(2it\theta)=0$ on the green curve.}
	\label{figtheta}
\end{figure}

\begin{remark}
According to \eqref{phasefunc}, $\theta(z)$ allows the following symmetry:
\begin{equation}
\theta(-z^{-1})=-\theta(z), \quad \theta(-\bar{z})=-\overline{\theta(z)}, \quad \theta(-\bar{z}^{-1})=-\overline{\theta(z)}.
\end{equation}
\end{remark}
\hspace*{\parindent}

\section{Deformation of the RH Problem}\label{openlens}

\subsection{Jump matrix  factorizations}\label{1stdeform}
Now we use factorizations of the jump matrix along $\Sigma$ to deform the contours onto those on which the oscillatory jump  is traded for exponential decay. Record the six stationary phase points as $\zeta_{j}, j=1,2,\cdot\cdot\cdot,6$, and $\alpha_{1,2}=\pm\mathrm{i}$, see in Figure \ref{alljunpcontours}. Then by the well known factorizations of $v(z)$:
\begin{align}
 v(x,t;z)=\left\{
        \begin{aligned}
    &\begin{bmatrix} 1 & -\tilde{\rho}e^{2\mathrm{i}t\theta} \\  0 & 1 \end{bmatrix}\begin{bmatrix} 1 & 0 \\  \rho e^{-2\mathrm{i}t\theta} & 0 \end{bmatrix}, \quad z\in \tilde{\Gamma}\\
    &\begin{bmatrix} 1 & 0 \\  \frac{\rho}{1-\rho\tilde{\rho}}e^{-2\mathrm{i}t\theta} & 1 \end{bmatrix}
    \begin{bmatrix} 1-\rho\tilde{\rho} & 0 \\ 0 & \frac{1}{1-\rho\tilde{\rho}} \end{bmatrix}
    \begin{bmatrix} 1 & -\frac{\tilde{\rho}}{1-\rho\tilde{\rho}}e^{2\mathrm{i}t\theta} \\ 0  & 1 \end{bmatrix}, \quad z\in \Gamma,
    \end{aligned}
        \right.
\end{align}
where
\begin{equation}\label{Gamma}
\tilde{\Gamma}=\widehat{\zeta_{2}\zeta_{3}}\cup\widehat{\zeta_{5}\zeta_{6}}, \quad \Gamma=\Sigma/\tilde{\Gamma}.
\end{equation}
We will utilize these factorizations to deform the jump contours so that the oscillating factor $e^{\pm2\mathrm{i}\theta}$ are decaying in corresponding region respectively. For this purpose, we introduce the following scalar RH problem
\begin{RHP}\label{scalarrhp}
Find a scalar function $\delta(z)\triangleq \delta(z;\xi)$, which is defined by the following properties:
\begin{itemize}
\item[*] Analyticity: $\delta(z)$ is analytical in $\mathbb{C}\backslash \Gamma$.
\item[*] Jump relation:
\begin{equation}
\begin{split}
& \delta{+}(z)=\delta{-}(z)(1-\rho(z)\tilde{\rho}(z)), \quad z\in {\Gamma};\\
& \delta{+}(z)=\delta{-}(z), \quad z\in\tilde{\Gamma}.
\end{split}
\end{equation}
\item[*] Asymptotic behavior:
\begin{equation}
\delta(z)\rightarrow 1, z\rightarrow\infty
\end{equation}
\end{itemize}
\end{RHP}
Utilizing the Plemelj's formula, we are arriving
\begin{equation}
  \delta(z)={\rm exp}\left[-\frac{1}{2\pi i}\int_{\Gamma}{\rm log}\left(1-\rho(s)\tilde{\rho}(s)\right)\frac{1}{s-z}ds\right].
\end{equation}
Taking $\nu(z)=-\frac{1}{2\pi}{\rm log}(1-\rho(z)\tilde{\rho}(z))$, then we can express
\begin{equation}\label{delta}
    \delta(z)={\rm exp}\left(i\int_{\Gamma}\frac{\nu(s)}{s-z}ds\right).
\end{equation}
\begin{remark}
    From the symmetries of $S(z)$ in Proposition \ref{symmetrypro}, we can derive the symmetry relation between $\rho(z)$ and $\tilde{\rho}(z)$, i.e. $\tilde{\rho}(z)=\overline{\rho(z)}$ when $|z|=1$, then for $z\in\{z\in\mathbb{C}: |z|=1\}\cap\Gamma$:
    \begin{equation}
     \nu(z)=-\frac{1}{2\pi}{\rm log}(1-|\rho(z)|^{2})
    \end{equation}
is a positive real function.
\end{remark}

For brevity, we denote $\mathcal{N}=\{1,2,\cdot\cdot\cdot,2N_{1}+N_{2}\}$. Moreover, we introduce a small positive constant $\varrho$:
\begin{equation}
\varrho=\frac{1}{2}{\rm min}\left\{\underset{k\in\mathcal{N}}{\rm min}\left\{|{\rm Im}(\eta_{k})|, |{\rm Im}(\hat{\eta}_{k})|\right\},{\underset{\lambda,\mu\in\mathcal{Z}\cup\hat{\mathcal{Z}},\lambda\neq\mu}{\rm min}|\lambda-\mu|},{\underset{\lambda\in\mathcal{Z}\cup\hat{\mathcal{Z}},i=1,2,3,4}{\rm min}|\lambda-\zeta_{i}|}\right\}.
\end{equation}
Taking $\delta_{0}<\varrho$, we define $\triangle, \nabla$ and $\Lambda$ of $\mathcal{N}$ as follows:
\begin{equation}
\begin{split}
&\triangle=\{k\in\mathcal{N}:{\rm Re}(2\mathrm{i}\theta(\eta_{k}))<0\}, \quad \nabla=\{k\in\mathcal{N}:Re(2\mathrm{i}\theta(\eta_{k}))>0\}, \\ &\Lambda=\{k\in\mathcal{N}:|{\rm Re}(2\mathrm{i}\theta(\eta_{k}))|<\delta_{0}\}.
\end{split}
\end{equation}
To distinguish different type of zeros, we further give
\begin{equation}
\begin{split}
&\triangle_{1}=\{k\in\{1,2,\cdot\cdot\cdot,N_{1}\}:{\rm Re}(2\mathrm{i}\theta(z_{k}))<0\}, \quad \nabla_{1}=\{k\in\{1,2,\cdot\cdot\cdot,N_{1}\}:{\rm Re}(2\mathrm{i}\theta(z_{k}))>0\}, \\ &\triangle_{2}=\{k\in\{1,2,\cdot\cdot\cdot,N_{1}\}:{\rm Re}(2\mathrm{i}\theta(\mathrm{i}\omega_{k}))<0\}, \quad \nabla_{2}=\{k\in\{1,2,\cdot\cdot\cdot,N_{1}\}:{\rm Re}(2\mathrm{i}\theta(\mathrm{i}\omega_{k})))>0\}, \\
&\Lambda_{1}=\{k\in\{1,2,\cdot\cdot\cdot,N_{1}\}:|{\rm Re}(2\mathrm{i}\theta(z_{k}))|<\delta_{0}\}, \quad \Lambda_{2}=\{k\in\{1,2,\cdot\cdot\cdot,N_{1}\}:|{\rm Re}(2\mathrm{i}\theta(\mathrm{i}\omega_{k}))|<\delta_{0}\}.
\end{split}
\end{equation}
Define the function
\begin{equation}\label{T}
T(z):=T(z;\xi)=\prod_{k\in\triangle_{1}}\prod_{l\in\triangle_{2}}\frac{(z+z_{k}^{-1})(z-\bar{z}_{k}^{-1})(z-\mathrm{i}\omega_{l}^{-1})}{(zz_{k}^{-1}-1)(z\bar{z}_{k}^{-1}+1)(\mathrm{i}\omega_{l}^{-1}z+1)}{\rm exp}[\mathrm{i}\int_{\Gamma}\nu(s)(\frac{1}{s-z}-\frac{1}{2s})ds].
\end{equation}
\begin{proposition}
The function defined by Eq. \eqref{T} has following properties:
    \begin{itemize}
        \item[(a)] $T(z)$ is meromorphic in $\mathbb{C}\backslash \Gamma$, and for each $k\in\triangle_{1}$, $l\in\triangle_{2}$, $z_{k}$, $-\bar{z}_{k}$, $\mathrm{i}\omega_{l}$ are simple poles and $-z_{k}^{-1}$, $\bar{z}_{k}^{-1}$, $\mathrm{i}\omega_{l}^{-1}$ are c simple zeros of $T(z)$.
        \item[(b)]$T(z)=-[T(-z^{-1})]^{-1}$.
        \item[(c)] For $z\in\Gamma$,
        \begin{equation}
        \frac{T_{+}(z)}{T_{-}(z)}=1-\rho(z)\tilde{\rho}(z).
        \end{equation}
        Particularly, for $z\in\{z:|z|=1\}\cap \Gamma$,
        \begin{equation}
        \frac{T_{+}(z)}{T_{-}(z)}=1-|\rho(z)|^{2}.
        \end{equation}
        \item[(d)] Define $\mathbb{R}_{i+}=[{\rm Re}\zeta_{i},+\infty)$, \ $i=1,4$, \ $\mathbb{R}_{i-}=(-\infty,{\rm Re}\zeta_{i}], \ i=2,3$. Then, as $z\rightarrow \zeta_{i}$ along any ray $\zeta_{i}+e^{\mathrm{i}\phi}\mathbb{R}_{i\pm}$ with $|\phi|<\pi$,
        \begin{equation}
        |T(z)-T_{i}(\zeta_{i})(z-\zeta_{i})^{-\mathrm{i}\nu(\zeta_{i})}|\leq c\left\|{\rm log}(1-\rho(z)\tilde{\rho}(z))\right\|_{H^{1}(\Gamma)}|z-\zeta_{i}|^{1/2},
        \end{equation}
         where
         \begin{equation}
         \begin{split}
         &T_{i}(\xi)=\prod_{k\in\triangle_{1}}\prod_{l\in\triangle_{2}}\frac{(\zeta_{i}+z_{k}^{-1})(\zeta_{i}-\bar{z}_{k}^{-1})(\zeta_{i}-\mathrm{i}\omega_{l}^{-1})}{(\zeta_{i}z_{k}^{-1}-1)(\zeta_{i}\bar{z}_{k}^{-1}+1)(\mathrm{i}\omega_{l}^{-1}\zeta_{i}+1)} {\rm exp}[\mathrm{i}\beta_{i}(\zeta_{i},\xi)],\\
         &\beta_{i}(z,\xi)=\nu(\zeta_{i}){\rm log}(z-\alpha_{l})+\int_{\widehat{\zeta_{i}\zeta_{k}}}\frac{\nu(s)-\nu(\zeta_{i})}{s-z}ds+\int_{\Gamma\backslash\widehat{\zeta_{i}\zeta_{k}}}\frac{\nu(s)}{s-z}ds-\int_{\Gamma}\frac{\nu(s)}{2s}ds,
        \end{split}
         \end{equation}
       $i=1,2,3,4$, where $k=5$ when $i=1,4$ and $k=6$ when $i=2,3$, $l=1$ when $i=1,2$ and $l=2$ when $i=3,4$.
    \end{itemize}
\end{proposition}
\begin{proof}
Properties $(a)$, $(b)$ and $(c)$ can be obtain by simple calculation from the definition of $T(z)$ in \eqref{T}. And for $(d)$, analogously to , rewrite
\begin{equation}
{\rm exp}[\mathrm{i}\int_{\Gamma}\nu(s)(\frac{1}{s-z}-\frac{1}{2s})ds]=(z-\zeta_{i})^{-\mathrm{i}\nu(\zeta_{i})}{\rm exp}[\mathrm{i}\beta_{i}(z,\xi)],
\end{equation}
and note the fact that
\begin{equation}
|(z-\zeta_{i})^{-\mathrm{i}\nu(\zeta_{i})}| \leq  e^{\pi \nu(\zeta_{i})}=(1-|\rho(\zeta_{i})|^{2})^{-1/2},
\end{equation}
and
\begin{equation}
|\beta_{i}(z,\xi)-\beta_{i}(\zeta_{i},\xi)|\leq c\left\|{\rm log}(1-\rho(z)\tilde{\rho}(z))\right\|_{H^{1}(\Gamma)}|z-\zeta_{i}|^{1/2}.
\end{equation}
The result then follows promptly.
\end{proof}
By using $T(z)$, the new matrix-valued function $m^{(1)}(z)$ is defined as
\begin{equation}\label{defm1}
m^{(1)}(z)=T(\infty)^{\sigma_{3}}m(z)T(z)^{-\sigma_{3}},
\end{equation}
which satisfies the following RH problem.
\begin{RHP}\label{rhp1}
Find a $2\times 2$ matrix-valued function $m^{(1)}(x,t;z)$ such that
\begin{itemize}
    \item[*] $m^{(1)}(z)$ is meromorphic in $\mathbb{C}\backslash \Sigma$.
    \item[*] The non-tangential limits $m^{(1)}_{\pm}(z)$ exist for any $z\in\Sigma=\tilde{\Gamma}\cup\Gamma$ and
    satisfy the jump relation $m_+^{(1)}(z)=m_-^{(1)}(z)v^{(1)}(z)$, where
    \begin{align}\label{rhp1jump}
    v(x,t;z)=\left\{
        \begin{aligned}
    &\begin{bmatrix} 1 & -\tilde{\rho}T^{2}e^{2\mathrm{i}t\theta} \\  0 & 1 \end{bmatrix}\begin{bmatrix} 1 & 0 \\  \rho T^{-2} e^{-2\mathrm{i}t\theta} & 0 \end{bmatrix}, \quad z\in \tilde{\Gamma}\\
    &\begin{bmatrix} 1 & 0 \\  \frac{\rho}{1-\rho\tilde{\rho}}T_{-}^{-2}e^{-2\mathrm{i}t\theta} & 1 \end{bmatrix}
    \begin{bmatrix} 1 & -\frac{\tilde{\rho}}{1-\rho\tilde{\rho}}T_{+}^{2}e^{2\mathrm{i}t\theta} \\ 0  & 1 \end{bmatrix}, \quad z\in \Gamma,
    \end{aligned}
        \right.
    \end{align}
    \item[*] Asymptotic behavior
    \begin{equation}\label{rhp1aym}
    m^{(1)}(x,t;z)=I+\mathcal{O}(z^{-1}), \quad  z\rightarrow\infty.
    \end{equation}
    \item[*] Residue conditions
    \begin{align}
    \underset{z=\eta_k}{\rm Res}m^{(1)}(z)=\left\{
    \begin{aligned}\label{rhp1resa}
    &\lim_{z\rightarrow\eta_{k}}m^{(1)}(z)\begin{bmatrix} 0 & 0 \\ A[\eta_{k}]T^{-2}(\eta_{k}) e^{-2it\theta(\eta_{k})} & 0 \end{bmatrix}, \quad k\in\triangle,\\
    &\lim_{z\rightarrow\eta_{k}}m^{(1)}(z)\begin{bmatrix} 0 & \frac{1}{A[\eta_{k}]}[(\frac{1}{T})'(\eta_{k})]^{-2}e^{2it\theta(\eta_{k})} \\ 0 & 0 \end{bmatrix}, \quad k\in\nabla,
    \end{aligned}
        \right.
    \end{align}
    \begin{align}\label{rhp1resb}
    \underset{z=\hat{\eta}_k}{\rm Res}m^{(1)}(z)=\left\{
    \begin{aligned}
    &\lim_{z\rightarrow\hat{\eta}_k}m^{(1)}(z)\begin{bmatrix} 0 & A[\hat{\eta}_k]T^{2}(\hat{\eta}_{k}) e^{2it\theta(\hat{\eta}_{k})} \\ 0 & 0 \end{bmatrix}, \quad k\in\triangle,\\
    &\lim_{z\rightarrow\hat{\eta}_k}m^{(1)}(z)\begin{bmatrix} 0 & 0 \\ \frac{1}{A[\hat{\eta}_k]}\frac{1}{[T'(\hat{\eta}_k)]^{2}} e^{-2it\theta(\hat{\eta}_{k})} & 0 \end{bmatrix}, \quad k\in\nabla.
    \end{aligned}
        \right.
    \end{align}
\end{itemize}
\end{RHP}
\begin{proof}
    The analyticity of $m^{(1)}(z)$ directly follows from its definition \eqref{rhp1}. By simple computation, we can obtain the jump relation and residue condition from \eqref{rhp1}, \eqref{rhp1jump}, \eqref{rhp1resa} and \eqref{rhp1resb}as well as the jump relation of RHP \ref{rhp0}. As for asymptotic behaviors, we notice that $T(\infty)^{\sigma_{3}}m(z)T(z)^{-\sigma_{3}}\rightarrow I,  z\rightarrow \infty$, thus the asymptotic behaviors of $m^{(1)}(z)$ is obtained.
\end{proof}

\subsection{Characteristic lines and estimates for ${\rm Im}\theta(z)$}

In this section, we construct a new matrix function $m^{(2)}$ for deforming the contour $\Sigma$ into a contour $\Sigma_{2}$ such that: First, $m^{(2)}$ has no jump on $\Sigma$. For this purpose, we choose the boundary values of $R^{(2)}$ through the factorization of $v^{(1)}$ where the new jumps on $\Sigma_{2}$ match a well known model RH problem; Second, we need to control the norm of $\mathcal{R}^{(2)}$, so that the $\partial$-contribution to the long-time asymptotics of $q(x,t)$ can be ignored; Third, the residues are unaffected by the transformation.

For this purpose, we first introduce some new contours, see in Figure \ref{alljunpcontours}:
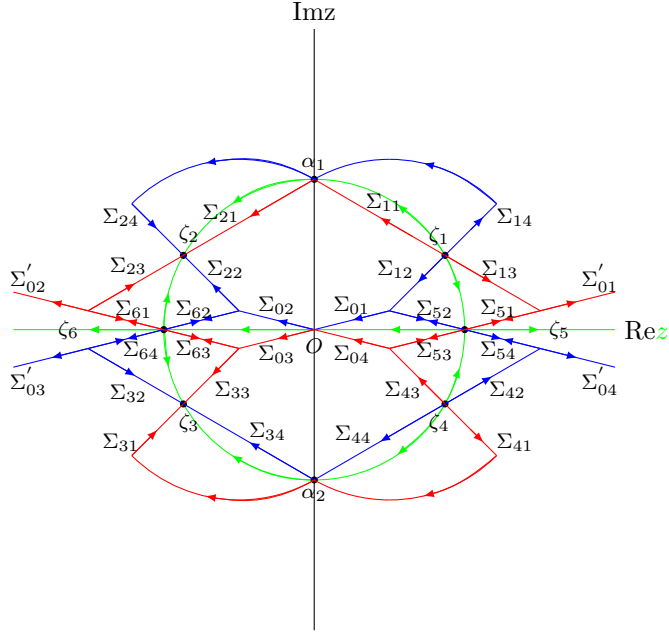
\begin{figure}[H]
        \begin{center}
        \begin{tikzpicture}[node distance=2cm]
          \draw [color=green](-4,0)--(4,0)  node[right, scale=1] {$\textcolor{black}{\rm Re}z$};
          \draw [](0,-4)--(0,4)  node[above, scale=1] {{\rm Im}z};
          \draw [color=green](0,0)circle(2cm);
          \draw [color=green] [-latex]  (1.7386,0.9886) to  [out=120, in=320]  (1.1,1.67);
          \draw [color=green] [-latex]  (1.7386,0.9886) to  [out=300, in=110]  (1.936,0.502);
          \draw [color=green] [-latex]  (1.7386,-0.9886) to  [out=60, in=250]  (1.936,-0.502);
          \draw [color=green] [-latex]  (1.7386,-0.9886) to  [out=240, in=40]  (1.1,-1.67);
          \draw [color=green] [-latex]  (0,-2) to  [out=180, in=330]  (-1.1,-1.67);
          \draw [color=green] [-latex]  (-2,0) to  [out=90, in=100]  (-1.936,-0.502);
          \draw [color=green] [-latex]  (-2,0) to  [out=90, in=260]  (-1.936,0.502);
          \draw [color=green][-latex]  (0,2) to  [out=180, in=33]  (-1.1,1.67);

          \node  [above]  at (0,2) {\footnotesize $\alpha_{1}$};
          \node  [below]  at (0,-2) {\footnotesize $\alpha_{2}$};
          \node  [right]  at (3,0) {\footnotesize $\zeta_{5}$};
          \node  [right]  at (1.4,1.245) {\footnotesize $\zeta_{1}$};
          \node  [left]  at (-1.4,1.28) {\footnotesize $\zeta_{2}$};
          \node  [left]  at (-3,0) {\footnotesize $\zeta_6$};
          \node  [left]  at (-1.4,-1.245) {\footnotesize $\zeta_{3}$};
          \node  [right]  at (1.4,-1.245) {\footnotesize $\zeta_{4}$};
          \node  [below]  at (0,0) {\footnotesize $O$};

          \draw[fill] (2,0) circle [radius=0.04];
          \draw[fill] (1.7386,0.9886) circle [radius=0.04];
          \draw[fill] (-1.7386,0.9886) circle [radius=0.04];
          \draw[fill] (-2,0) circle [radius=0.04];
          \draw[fill] (-1.7386,-0.9886) circle [radius=0.04];
          \draw[fill] (1.7386,-0.9886) circle [radius=0.04];
          \draw[fill] (0,2) circle [radius=0.04];
          \draw[fill] (0,-2) circle [radius=0.04];
          \draw [color=green] [-latex] (2,0) -- (1,0);
          \draw [color=green] [-latex] (2,0) -- (3,0);
          \draw [color=green] [-latex] (0,0) -- (-1,0);
          \draw [color=green] [-latex] (-2,0) -- (-3,0);
          \draw [color=green] [-latex] (-2,0) -- (-3,0);
          \draw [color=blue](0,0)--(1,0.25)  node[right, scale=1] {};
          \draw [color=blue](1,0.25)--(4,-0.5)  node[right, scale=1] {};
          \draw [color=blue](0,0)--(-1,0.25)  node[right, scale=1] {};
          \draw [color=blue](-1,0.25)--(-4,-0.5)  node[right, scale=1] {};
          \draw [color=red](-1,-0.25)--(-2.425,-1.675)  node[right, scale=1] {};
          \draw [color=red](0,0)--(1,-0.25)  node[right, scale=1] {};
          \draw [color=red](1,-0.25)--(4,0.5)  node[right, scale=1] {};
          \draw [color=red](1,-0.25)--(2.425,-1.675)  node[right, scale=1] {};
          \draw [color=red](0,0)--(-1,-0.25)  node[right, scale=1] {};
          \draw [color=red](-1,-0.25)--(-4,0.5)  node[right, scale=1] {};
          \draw [color=blue](1,0.25)--(2.425,1.675)  node[right, scale=1] {};
          \draw [color=blue](-1,0.25)--(-2.425,1.675)  node[right, scale=1] {};
          \draw [color=red](0,2)--(3.006,0.251)  node[right, scale=1] {};
          \draw [color=red](0,2)--(-3.006,0.251)  node[right, scale=1] {};
          \draw [color=blue](0,-2)--(3.006,-0.251)  node[right, scale=1] {};
          \draw [color=blue](0,-2)--(-3.006,-0.251)  node[right, scale=1] {};
          \draw[color=blue] (1,2.266) arc (90:45:2.016);
          \draw[color=blue] (1,2.266) arc (90:120:2.016);
          \draw[color=blue] (-1,2.266) arc (90:60:2.016);
          \draw[color=blue] (-1,2.266) arc (90:135:2.016);
          \draw[color=red] (-1,-2.266) arc (270:300:2.016);
          \draw[color=red] (-1,-2.266) arc (270:225:2.016);
          \draw[color=red] (1,-2.266) arc (270:315:2.016);
          \draw[color=red] (1,-2.266) arc (270:240:2.016);

          \draw [color=blue] [-latex] (1,0.25) -- (0.5,0.125);
          \node  [above]  at (0.5,0.09) {\footnotesize $\Sigma_{01} $};
          \draw [color=blue] [-latex] (0,0) -- (-0.5,0.125);
          \node  [above]  at (-0.5,0.09) {\footnotesize $\Sigma_{02} $};
          \draw [color=red] [-latex] (0,0) -- (-0.5,-0.125);
          \node  [below]  at (-0.5,-0.09) {\footnotesize $\Sigma_{03} $};
          \draw [color=red] [-latex]  (1,-0.25)--(0.5,-0.125) ;
          \node  [below]  at (0.5,-0.09) {\footnotesize $\Sigma_{04} $};

          \draw [color=blue] [-latex] (1,0.25) -- (1.5,0.125);
          \draw [color=blue] [-latex] (2,0) -- (1.5,0.125);
          \node  [above]  at (1.6,0.03) {\footnotesize $\Sigma_{52} $};
          \draw [color=blue] [-latex] (3.006,-0.251) -- (2.5,-0.125);
          \draw [color=blue] [-latex] (2,0) -- (2.5,-0.125);
          \node [below]  at (2.45,-0.03) {\footnotesize $\Sigma_{54} $};
          \draw [color=red] [-latex] (3.006,0.251) -- (2.5,0.125);
          \draw [color=red] [-latex] (2,0) -- (2.5,0.125);
          \node [above]  at (2.45,0.04) {\footnotesize $\Sigma_{51} $};
          \draw [color=red] [-latex] (1,-0.25) -- (1.5,-0.125);
          \draw [color=red] [-latex] (2,0) -- (1.5,-0.125);
          \node [below]  at (1.6,-0.05) {\footnotesize $\Sigma_{53} $};

          \draw [color=blue] [-latex] (-1,0.25) -- (-1.5,0.125);
          \draw [color=blue] [-latex] (-2,0) -- (-1.5,0.125);
          \node  [above]  at (-1.6,0.03) {\footnotesize $\Sigma_{62} $};
          \draw [color=blue] [-latex] (-3.006,-0.251) -- (-2.5,-0.125);
          \draw [color=blue] [-latex] (-2,0) -- (-2.5,-0.125);
          \node [below]  at (-2.3,-0.04) {\footnotesize $\Sigma_{64} $};
          \draw [color=red] [-latex] (-3.006,0.251) -- (-2.5,0.125);
          \draw [color=red] [-latex] (-2,0) -- (-2.5,0.125);
          \node [above]  at (-2.4,0.025) {\footnotesize $\Sigma_{61} $};
          \draw [color=red] [-latex] (-1,-0.25) -- (-1.5,-0.125);
          \draw [color=red] [-latex] (-2,0) -- (-1.5,-0.125);
          \node [below]  at (-1.6,-0.025) {\footnotesize $\Sigma_{63} $};

          \draw [color=red] [-latex] (3.006,0.251) -- (3.5,0.375);
          \node [above]  at (3.8,0.375) {\footnotesize $\Sigma_{01}^{'} $};
          \draw [color=red] [-latex] (-3.006,0.251) -- (-3.5,0.375);
          \node [above]  at (-3.8,0.375) {\footnotesize $\Sigma_{02}^{'} $};
          \draw [color=blue] [-latex] (-3.006,-0.251) -- (-3.5,-0.375);
          \node [below]  at (-3.8,-0.375) {\footnotesize $\Sigma_{03}^{'} $};
          \draw [color=blue] [-latex] (3.006,-0.251) -- (3.5,-0.375);
          \node [below]  at (3.8,-0.375) {\footnotesize $\Sigma_{04}^{'} $};

          \draw [color=blue] [-latex] (2.425,1.675) to  [out=140, in=345]  (1.45,2.215);
          \draw [color=blue] [-latex] (2.0818,1.3318) -- (2.2534,1.5034);
          \node [right]  at (2.3,1.5) {\footnotesize $\Sigma_{14}$};
          \draw [color=blue] [-latex] (1.7386,0.9886) -- (1.3693,0.6193);
          \node [left]  at (1.45,0.8) {\footnotesize $\Sigma_{12}$};
          \draw [color=red] [-latex] (1.30395,1.24145)-- (0.8693,1.4943);
          \node [right]  at (0.56,1.6943) {\footnotesize $\Sigma_{11}$};
          \draw [color=red] [-latex] (1.7386,0.9886) -- (2.3723,0.6198);
          \node [right]  at (2.1,0.8) {\footnotesize $\Sigma_{13}$};

          \draw [color=blue] [-latex]  (0,2) to  [out=150, in=10]  (-1.45,2.215);
          \draw [color=blue] [-latex]  (-2.425,1.675)-- (-2.0818,1.3318);
          \node [left]  at (-2.2,1.5) {\footnotesize $\Sigma_{24}$};
          \draw [color=blue] [-latex] (-1,0.25) -- (-1.3693,0.6193);
          \node [right]  at (-1.55,0.8) {\footnotesize $\Sigma_{22}$};
          \draw [color=red] [-latex]  (0,2) -- (-0.8693,1.4943);
          \node [left]  at (-0.8693,1.55) {\footnotesize $\Sigma_{21}$};
          \draw [color=red] [-latex] (-3.006,0.251) -- (-2.3723,0.6198);
          \node [left]  at (-2.1,0.85) {\footnotesize $\Sigma_{23}$};

          \draw [color=red] [-latex]  (0,-2) to  [out=210, in=350]  (-1.45,-2.215);
          \draw [color=red] [-latex]  (-2.425,-1.675)-- (-2.0818,-1.3318);
          \node [left]  at (-2.2,-1.5) {\footnotesize $\Sigma_{31}$};
          \draw [color=red] [-latex] (-1,-0.25) -- (-1.3693,-0.6193);
          \node [right]  at (-1.45,-0.8) {\footnotesize $\Sigma_{33}$};
          \draw [color=blue] [-latex]  (0,-2) -- (-0.8693,-1.4943);
          \node [right]  at (-0.99,-1.35) {\footnotesize $\Sigma_{34}$};
          \draw [color=blue] [-latex] (-3.006,-0.251) -- (-2.3723,-0.6198);
          \node [left]  at (-2.1,-0.85) {\footnotesize $\Sigma_{32}$};

          \draw [color=red] [-latex] (2.425,-1.675) to [out=220, in=15]  (1.45,-2.215);
          \draw [color=red] [-latex] (2.0818,-1.3318) -- (2.2534,-1.5034);
          \node [right]  at (2.3,-1.5) {\footnotesize $\Sigma_{41}$};
          \draw [color=red] [-latex] (1.7386,-0.9886) -- (1.3693,-0.6193);
          \node [left]  at (1.55,-0.8) {\footnotesize $\Sigma_{43}$};
          \draw [color=blue] [-latex] (1.7386,-0.9886) -- (0.8693,-1.4943);
          \node [left]  at (0.8893,-1.3943) {\footnotesize $\Sigma_{44}$};
          \draw [color=blue] [-latex] (1.7386,-0.9886) -- (2.3723,-0.6198);
          \node [right]  at (2.2,-0.8) {\footnotesize $\Sigma_{42}$};
          \end {tikzpicture}
          \caption{\footnotesize The green curves are the origin jump contours $\Sigma=\mathbb{R}\cup\{Z\in\mathbb{C}: |z|=1\}$, the blue curves are the opening contours in region $\{z\in\mathbb{C}: |e^{-2\mathrm{i}t\theta(z)}|\rightarrow0\}$ while the red curves are the opening contours in region $\{z\in\mathbb{C}: |e^{2\mathrm{i}t\theta(z)}|\rightarrow0\}$. These arrows represent directions of jump contours.}.
        \label{alljunpcontours}
        \end{center}
        \end{figure}
\begin{itemize}
    \item[1.] To open the lens at $(0,0)$, we fix an angle $\theta_{0}>0$ sufficiently small such that the set $\{z\in \mathbb{C}: \left| \frac{{\rm Re}z}{z}\right|\}$ does not intersect any of the disks $|z-\eta_{k}|<\varrho$ or $|z-\hat{\eta}_{k}|<\varrho$.
        Let
        \begin{equation}
        \phi(\xi)={\rm min}\left\{\theta_{0}, {\rm arccos}\sqrt{\frac{|\xi|+6}{12}}\right\}
        \end{equation}
        and define
        \begin{equation}
        \begin{split}
         &\Sigma_{01}=l_{01}e^{\mathrm{i}\phi(\xi)},  \quad \Sigma_{02}=l_{02}e^{\mathrm{i}(\pi-\phi(\xi)}), \\
         & \Sigma_{03}=l_{03}e^{\mathrm{i}(\pi+\phi(\xi))}, \quad \Sigma_{04}=l_{04}e^{\mathrm{i}(2\pi-\phi(\xi)}),
        \end{split}
        \end{equation}
        where
        \begin{equation}
        l_{01}=l_{04}=(0,\frac{1}{2}), \quad l_{02}=l_{03}=(-\frac{1}{2},0).
        \end{equation}
        In the same way, we can define $\Sigma_{5j}$, $\Sigma_{6j}$ and $\Sigma_{0j}^{'}$, \ $j=1,2,3,4.$
    \item[2.] Let
        \begin{equation}
        \begin{split}
        & l_{11}=(0,{\rm Re}\zeta_{1}), \quad l_{12}=(1,{\rm Re}\zeta_{1}),\\
        & l_{13}=({\rm Re}\zeta_{1},3), \quad l_{14}=({\rm Re}\zeta_{1},2.4).
        \end{split}
        \end{equation}
        Then define
        \begin{equation}
        \begin{split}
        & \Sigma_{11}=l_{11}e^{\frac{\pi}{4}}, \quad \Sigma_{12}=l_{12}e^{-\frac{\pi}{4}},\\
        & \Sigma_{13}=l_{13}e^{-\frac{\pi}{4}},
        \end{split}
        \end{equation}
        and the straight part of $\Sigma_{14}$ as $l_{14}e^{\frac{\pi}{4}}$. Moreover, $\Sigma_{ij}$ can be defined similarly, where  $i=2,3,4, \ j=1,2,3,4$.
        \end{itemize}

Inspired by  the idea of Cuccagna and Burgers in \cite{BJ,CJ}, we open the jump contour $\Sigma$ at $(0,0)$ by a small angle and at six stationary phase points by $\frac{\pi}{4}$, see the blue and red curves in Figure \ref{alljunpcontours}. It is worth noting that there are two opened jump contours of opposite directions on each $\Sigma_{ij}$, \ $i=5,6, \ j=1,2,3,4$. Thus, they can cancel each other out and the contours $\Sigma_{ij}$ are disappeared, where $i=5,6, \ j=1,2,3,4$. In a word, the boundary of $\mathcal{R}^{(2)}$ can be defined as:
\begin{equation}
\begin{split}
&\widetilde{\Sigma}^{(2)}=L_{0}\cup L^{'}_{0}\cup(\underset{i=1,2,3,4}\cup L_{i}),\\
& L_{0}=\underset{j=1,2,3,4}\cup\Sigma_{0j}, \quad L^{'}_{0}=\underset{j=1,2,3,4}\cup\Sigma^{'}_{0j},\\
& L_{i}=\underset{j=1,2,3,4}\cup\Sigma_{ij}, \quad i=1,2,\cdot\cdot\cdot,6.
\end{split}
\end{equation}
See Figure \ref{boumdaryofR2}. Additionally, domains of $\mathcal{R}^{(2)}$ are presented in Figure \ref{domainsofR2}.

\begin{figure}[H]
\begin{center}
\begin{tikzpicture}[node distance=2cm]
\draw[->](-5,0)--(5,0)node[right]{ \textcolor{black}{${\rm Re} z$}};
\draw(0,0)circle(2cm);

\node  [right]  at (3,0) {\footnotesize $\zeta_{5}$};
\node  [right]  at (1.4,1.245) {\footnotesize $\zeta_{1}$};
\node  [left]  at (-1.4,1.28) {\footnotesize $\zeta_{2}$};
\node  [left]  at (-3,0) {\footnotesize $\zeta_6$};
\node  [left]  at (-1.4,-1.245) {\footnotesize $\zeta_{3}$};
\node  [right]  at (1.4,-1.245) {\footnotesize $\zeta_{4}$};

\draw[fill] (1.7386,0.9886) circle [radius=0.04];
\draw[fill] (-1.7386,0.9886) circle [radius=0.04];
\draw[fill] (-1.7386,-0.9886) circle [radius=0.04];
\draw[fill] (1.7386,-0.9886) circle [radius=0.04];

\draw [color=blue](0,0)--(1,0.25)  node[right, scale=1] {};
\draw [color=blue](3.006,-0.251)--(5,-0.75)  node[right, scale=1] {};
\draw [color=blue](0,0)--(-1,0.25)  node[right, scale=1] {};
\draw [color=blue](-3.006,-0.251)--(-5,-0.75)  node[right, scale=1] {};
\draw [color=red](-1,-0.25)--(-2.425,-1.675)  node[right, scale=1] {};
\draw [color=red](0,0)--(1,-0.25)  node[right, scale=1] {};
\draw [color=red](3.006,0.251)--(5,0.75)  node[right, scale=1] {};
\draw [color=red](1,-0.25)--(2.425,-1.675)  node[right, scale=1] {};
\draw [color=red](0,0)--(-1,-0.25)  node[right, scale=1] {};
\draw [color=red](-3.006,0.251)--(-5,0.75)  node[right, scale=1] {};
\draw [color=blue](1,0.25)--(2.425,1.675)  node[right, scale=1] {};
\draw [color=blue](-1,0.25)--(-2.425,1.675)  node[right, scale=1] {};
\draw [color=red](0,2)--(3.006,0.251)  node[right, scale=1] {};
\draw [color=red](0,2)--(-3.006,0.251)  node[right, scale=1] {};
\draw [color=blue](0,-2)--(3.006,-0.251)  node[right, scale=1] {};
\draw [color=blue](0,-2)--(-3.006,-0.251)  node[right, scale=1] {};
\draw[color=blue] (1,2.266) arc (90:45:2.016);
\draw[color=blue] (1,2.266) arc (90:120:2.016);
\draw[color=blue] (-1,2.266) arc (90:60:2.016);
\draw[color=blue] (-1,2.266) arc (90:135:2.016);
\draw[color=red] (-1,-2.266) arc (270:300:2.016);
\draw[color=red] (-1,-2.266) arc (270:225:2.016);
\draw[color=red] (1,-2.266) arc (270:315:2.016);
\draw[color=red] (1,-2.266) arc (270:240:2.016);

\draw [color=blue] [-latex] (1,0.25) -- (0.5,0.125);
\node  [above]  at (0.8,-0.15) {\tiny $\Sigma_{01} $};
\draw [color=blue] [-latex] (0,0) -- (-0.5,0.125);
\node  [above]  at (-0.7,-0.09) {\tiny $\Sigma_{02} $};
\draw [color=red] [-latex] (0,0) -- (-0.5,-0.125);
\node  [below]  at (-0.7,0.09) {\tiny $\Sigma_{03} $};
\draw [color=red] [-latex]  (1,-0.25)--(0.5,-0.125) ;
\node  [below]  at (0.8,0.08) {\tiny $\Sigma_{04} $};

          \draw [color=red] [-latex] (3.006,0.251) -- (3.5,0.375);
          \node [above]  at (3.8,0.375) {\footnotesize $\Sigma_{01}^{'} $};
          \draw [color=red] [-latex] (-3.006,0.251) -- (-3.5,0.375);
          \node [above]  at (-3.8,0.375) {\footnotesize $\Sigma_{02}^{'} $};
          \draw [color=blue] [-latex] (-3.006,-0.251) -- (-3.5,-0.375);
          \node [below]  at (-3.8,-0.375) {\footnotesize $\Sigma_{03}^{'} $};
          \draw [color=blue] [-latex] (3.006,-0.251) -- (3.5,-0.375);
          \node [below]  at (3.8,-0.375) {\footnotesize $\Sigma_{04}^{'} $};

          \draw [color=blue] [-latex] (2.425,1.675) to  [out=140, in=345]  (1.45,2.215);
          \draw [color=blue] [-latex] (2.0818,1.3318) -- (2.2534,1.5034);
          \node [right]  at (2.3,1.5) {\footnotesize $\Sigma_{14}$};
          \draw [color=blue] [-latex] (1.7386,0.9886) -- (1.3693,0.6193);
          \node [left]  at (1.45,0.8) {\footnotesize $\Sigma_{12}$};
          \draw [color=red] [-latex] (1.30395,1.24145)-- (0.8693,1.4943);
          \node [right]  at (0.7793,1.5943) {\footnotesize $\Sigma_{11}$};
          \draw [color=red] [-latex] (1.7386,0.9886) -- (2.3723,0.6198);
          \node [right]  at (2.1,0.8) {\footnotesize $\Sigma_{13}$};

          \draw [color=blue] [-latex]  (0,2) to  [out=150, in=10]  (-1.45,2.215);
          \draw [color=blue] [-latex]  (-2.425,1.675)-- (-2.0818,1.3318);
          \node [left]  at (-2.2,1.5) {\footnotesize $\Sigma_{24}$};
          \draw [color=blue] [-latex] (-1,0.25) -- (-1.3693,0.6193);
          \node [right]  at (-1.55,0.8) {\footnotesize $\Sigma_{22}$};
          \draw [color=red] [-latex]  (0,2) -- (-0.8693,1.4943);
          \node [left]  at (-0.8693,1.55) {\footnotesize $\Sigma_{21}$};
          \draw [color=red] [-latex] (-3.006,0.251) -- (-2.3723,0.6198);
          \node [left]  at (-2.1,0.85) {\footnotesize $\Sigma_{23}$};

          \draw [color=red] [-latex]  (0,-2) to  [out=210, in=350]  (-1.45,-2.215);
          \draw [color=red] [-latex]  (-2.425,-1.675)-- (-2.0818,-1.3318);
          \node [left]  at (-2.2,-1.5) {\footnotesize $\Sigma_{31}$};
          \draw [color=red] [-latex] (-1,-0.25) -- (-1.3693,-0.6193);
          \node [right]  at (-1.45,-0.8) {\footnotesize $\Sigma_{33}$};
          \draw [color=blue] [-latex]  (0,-2) -- (-0.8693,-1.4943);
          \node [right]  at (-0.99,-1.35) {\footnotesize $\Sigma_{34}$};
          \draw [color=blue] [-latex] (-3.006,-0.251) -- (-2.3723,-0.6198);
          \node [left]  at (-2.1,-0.85) {\footnotesize $\Sigma_{32}$};

          \draw [color=red] [-latex] (2.425,-1.675) to [out=220, in=15]  (1.45,-2.215);
          \draw [color=red] [-latex] (2.0818,-1.3318) -- (2.2534,-1.5034);
          \node [right]  at (2.3,-1.5) {\footnotesize $\Sigma_{41}$};
          \draw [color=red] [-latex] (1.7386,-0.9886) -- (1.3693,-0.6193);
          \node [left]  at (1.55,-0.8) {\footnotesize $\Sigma_{43}$};
          \draw [color=blue] [-latex] (1.7386,-0.9886) -- (0.8693,-1.4943);
          \node [left]  at (0.8893,-1.3943) {\footnotesize $\Sigma_{44}$};
          \draw [color=blue] [-latex] (1.7386,-0.9886) -- (2.3723,-0.6198);
          \node [right]  at (2.2,-0.8) {\footnotesize $\Sigma_{42}$};
\end {tikzpicture}
\caption{\footnotesize Boundary of $\mathcal{R}^{(2)}$.}
\label{boumdaryofR2}
\end{center}
\end{figure}
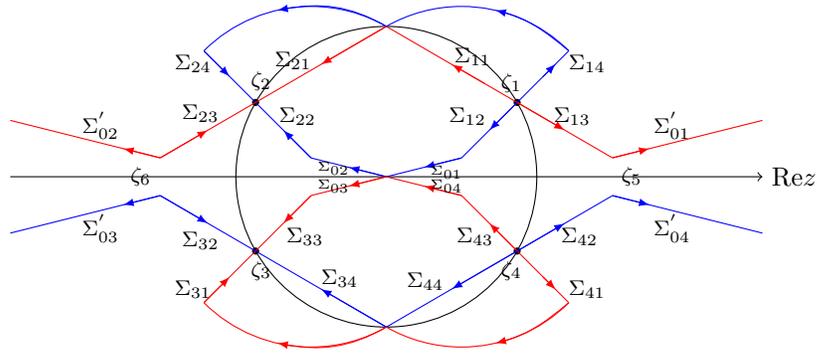

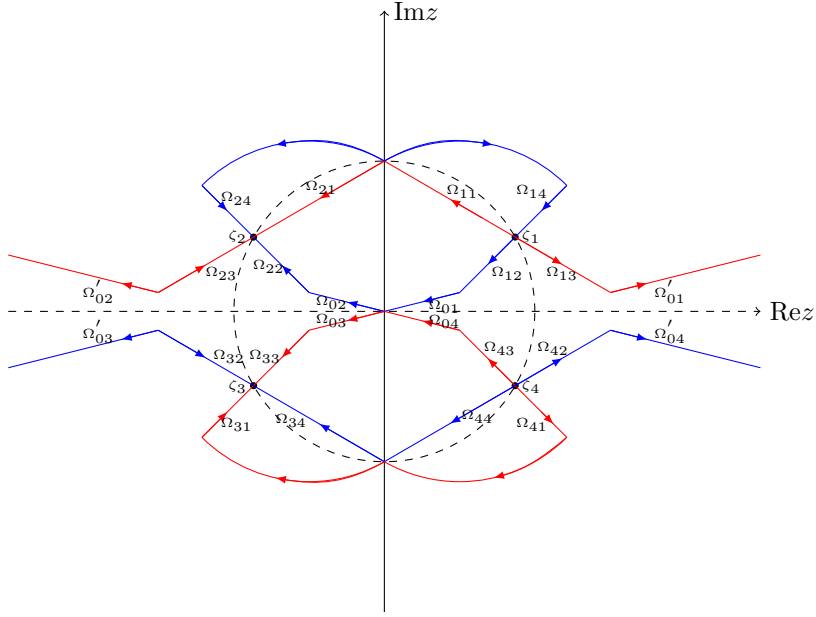
\begin{figure}[H]
\begin{center}
\begin{tikzpicture}[node distance=2cm]
\draw[->,dashed](-5,0)--(5,0)node[right]{ \textcolor{black}{${\rm Re} z$}};
\draw[->](0,-4)--(0,4)node[right]{\textcolor{black}{${\rm Im} z$}};
\draw[dashed](0,0)circle(2cm);
\node  [right]  at (1.7,1) {\tiny $\zeta_{1}$};
\node  [left]  at (-1.7,1) {\tiny $\zeta_{2}$};
\node  [left]  at (-1.7,-1) {\tiny $\zeta_{3}$};
\node  [right]  at (1.7,-1) {\tiny $\zeta_{4}$};

\draw[fill] (1.7386,0.9886) circle [radius=0.04];
\draw[fill] (-1.7386,0.9886) circle [radius=0.04];
\draw[fill] (-1.7386,-0.9886) circle [radius=0.04];
\draw[fill] (1.7386,-0.9886) circle [radius=0.04];

\draw [color=blue](0,0)--(1,0.25)  node[right, scale=1] {};
\draw [color=blue](3.006,-0.251)--(5,-0.75)  node[right, scale=1] {};
\draw [color=blue](0,0)--(-1,0.25)  node[right, scale=1] {};
\draw [color=blue](-3.006,-0.251)--(-5,-0.75)  node[right, scale=1] {};
\draw [color=red](-1,-0.25)--(-2.425,-1.675)  node[right, scale=1] {};
\draw [color=red](0,0)--(1,-0.25)  node[right, scale=1] {};
\draw [color=red](3.006,0.251)--(5,0.75)  node[right, scale=1] {};
\draw [color=red](1,-0.25)--(2.425,-1.675)  node[right, scale=1] {};
\draw [color=red](0,0)--(-1,-0.25)  node[right, scale=1] {};
\draw [color=red](-3.006,0.251)--(-5,0.75)  node[right, scale=1] {};
\draw [color=blue](1,0.25)--(2.425,1.675)  node[right, scale=1] {};
\draw [color=blue](-1,0.25)--(-2.425,1.675)  node[right, scale=1] {};
\draw [color=red](0,2)--(3.006,0.251)  node[right, scale=1] {};
\draw [color=red](0,2)--(-3.006,0.251)  node[right, scale=1] {};
\draw [color=blue](0,-2)--(3.006,-0.251)  node[right, scale=1] {};
\draw [color=blue](0,-2)--(-3.006,-0.251)  node[right, scale=1] {};
\draw[color=blue] (1,2.266) arc (90:45:2.016);
\draw[color=blue] (1,2.266) arc (90:120:2.016);
\draw[color=blue] (-1,2.266) arc (90:60:2.016);
\draw[color=blue] (-1,2.266) arc (90:135:2.016);
\draw[color=red] (-1,-2.266) arc (270:300:2.016);
\draw[color=red] (-1,-2.266) arc (270:225:2.016);
\draw[color=red] (1,-2.266) arc (270:315:2.016);
\draw[color=red] (1,-2.266) arc (270:240:2.016);

\draw [color=blue] [-latex] (1,0.25) -- (0.5,0.125);
\node  [above]  at (0.8,-0.15) {\tiny $\Omega_{01} $};
\draw [color=blue] [-latex] (0,0) -- (-0.5,0.125);
\node  [above]  at (-0.7,-0.09) {\tiny $\Omega_{02} $};
\draw [color=red] [-latex] (0,0) -- (-0.5,-0.125);
\node  [below]  at (-0.7,0.09) {\tiny $\Omega_{03} $};
\draw [color=red] [-latex]  (1,-0.25)--(0.5,-0.125) ;
\node  [below]  at (0.8,0.08) {\tiny $\Omega_{04} $};

\draw [color=red] [-latex] (3.006,0.251) -- (3.5,0.375);
\node [below]  at (3.8,0.55) {\tiny  $\Omega_{01}^{'} $};
\draw [color=red] [-latex] (-3.006,0.251) -- (-3.5,0.375);
\node [below]  at (-3.8,0.55) {\tiny $\Omega_{02}^{'} $};
\draw [color=blue] [-latex] (-3.006,-0.251) -- (-3.5,-0.375);
\node [above]  at (-3.8,-0.55) {\tiny  $\Omega_{03}^{'} $};
\draw [color=blue] [-latex] (3.006,-0.251) -- (3.5,-0.375);
\node [above]  at (3.8,-0.55) {\tiny  $\Omega_{04}^{'} $};

\draw [color=blue] [-latex]  (0,2) to  [out=30, in=170]  (1.45,2.215);
\draw [color=blue] [-latex] (2.425,1.675) -- (2.0818,1.3318);
\node [left]  at (2.3,1.5943) {\tiny  $\Omega_{14}$};
\draw [color=blue] [-latex] (1.7386,0.9886) -- (1.3693,0.6193);
\node [right]  at (1.3,0.5) {\tiny  $\Omega_{12}$};
\draw [color=red] [-latex] (1.30395,1.24145)-- (0.8693,1.4943);
\node [right]  at (0.7,1.5943) {\tiny  $\Omega_{11}$};
\draw [color=red] [-latex] (1.7386,0.9886) -- (2.3723,0.6198);
\node [left]  at (2.7,0.5) {\tiny  $\Omega_{13}$};

\draw [color=blue] [-latex]  (0,2) to  [out=150, in=10]  (-1.45,2.215);
\draw [color=blue] [-latex]  (-2.425,1.675)-- (-2.0818,1.3318);
\node [right]  at (-2.3,1.5) {\tiny  $\Omega_{24}$};
\draw [color=blue] [-latex] (-1,0.25) -- (-1.3693,0.6193);
\node [left]  at (-1.2,0.6) {\tiny  $\Omega_{22}$};
\draw [color=red] [-latex]  (0,2) -- (-0.8693,1.4943);
\node [left]  at (-0.5,1.65) {\tiny  $\Omega_{21}$};
\draw [color=red] [-latex] (-3.006,0.251) -- (-2.3723,0.6198);
\node [right]  at (-2.5,0.5) {\tiny  $\Omega_{23}$};

\draw [color=red] [-latex]  (0,-2) to  [out=210, in=350]  (-1.45,-2.215);
\draw [color=red] [-latex]  (-2.425,-1.675)-- (-2.0818,-1.3318);
\node [right]  at (-2.3,-1.5) {\tiny  $\Omega_{31}$};
\draw [color=red] [-latex] (-1,-0.25) -- (-1.3693,-0.6193);
\node [left]  at (-1.25,-0.6) {\tiny $\Omega_{33}$};
\draw [color=blue] [-latex]  (0,-2) -- (-0.8693,-1.4943);
\node [left]  at (-0.9,-1.45) {\tiny  $\Omega_{34}$};
\draw [color=blue] [-latex] (-3.006,-0.251) -- (-2.3723,-0.6198);
\node [right]  at (-2.4,-0.6) {\tiny  $\Omega_{32}$};

\draw [color=red] [-latex] (2.425,-1.675) to [out=220, in=15]  (1.45,-2.215);
\draw [color=red] [-latex] (2.0818,-1.3318) -- (2.2534,-1.5034);
\node [left]  at (2.3,-1.5) {\tiny  $\Omega_{41}$};
\draw [color=red] [-latex] (1.7386,-0.9886) -- (1.3693,-0.6193);
\node [right]  at (1.2,-0.5) {\tiny  $\Omega_{43}$};
\draw [color=blue] [-latex] (1.7386,-0.9886) -- (0.8693,-1.4943);
\node [right]  at (0.9,-1.4) {\tiny  $\Omega_{44}$};
\draw [color=blue] [-latex] (1.7386,-0.9886) -- (2.3723,-0.6198);
\node [left]  at (2.58,-0.5) {\tiny  $\Omega_{42}$};
\end {tikzpicture}
\caption{\footnotesize Domains of $\mathcal{R}^{(2)}$.}.
\label{domainsofR2}
\end{center}
\end{figure}

\begin{proposition}
For $z=|z|e^{\mathrm{i}w}$,
\begin{equation}
\begin{split}
& {\rm Re}[2\mathrm{i}\theta(z)]>-(1-|z|^{-2})|z|{\rm sin}wg(z)>0, \quad z\in\Omega_{01}\cup\Omega_{02}\cup\Omega^{'}_{03}\cup\Omega^{'}_{04},\\
& {\rm Re}[2\mathrm{i}\theta(z)]<-(1-|z|^{-2})|z|{\rm sin}wg(z)<0, \quad z\in\Omega_{02}\cup\Omega_{03}\cup\Omega^{'}_{01}\cup\Omega^{'}_{02},
\end{split}
\end{equation}
where
\begin{equation}
g(z)=3(|z|+|z|^{-1})^{2}{\rm cos}^{2}w+\xi-4>0.
\end{equation}
\end{proposition}
\begin{proof}
We give a proof for $z\in\Omega_{01}$, the others are similar. For $z=|z|e^{\mathrm{i}w}\in\Omega_{01}$,
\begin{equation}
{\rm Re}[2\mathrm{i}\theta(z)]=-(1-|z|^{-2})|z|{\rm sin}w[3(|z|+|z|^{-1})^{2}{\rm cos}^{2}w+\xi-4],
\end{equation}
let
\begin{equation}
g(z)=3(|z|+|z|^{-1})^{2}{\rm cos}^{2}w+\xi-4.
\end{equation}
Noticing $w\in(0,\phi(\xi))$, thus $\sqrt{\frac{|\xi|+6}{12}}<{\rm cos}w<1$. Then observing that $|z|+|z|^{-1}\geq2$, we can easy obtain $g(z)>0$. As a result,
\begin{equation}
{\rm Re}[2\mathrm{i}\theta(z)]>-(1-|z|^{-2})|z|{\rm sin}wg(z)>0, z\in\Omega_{01}.
\end{equation}
\end{proof}

\subsection{Mixed  $\bar{\partial}$-RH problem}

We choose $\mathcal{R}^{(2)}(z):=\mathcal{R}^{(2)}(z;\xi)$ as:
\begin{align}
\mathcal{R}^{(2)}(z)=\left\{
        \begin{aligned}
    &\begin{bmatrix} 1 & 0 \\  R_{0j}e^{-2\mathrm{i}t\theta} & 1 \end{bmatrix}, \quad z\in\Omega_{0j}, \quad j=1,2,\\
    &\begin{bmatrix} 1 & -R_{0j}e^{2\mathrm{i}t\theta} \\ 0 & 1 \end{bmatrix}^{-1}, \quad z\in\Omega_{0j}, \quad j=3,4,\\
    &\begin{bmatrix} 1 & -R_{0j}^{'}e^{2\mathrm{i}t\theta} \\ 0 & 1 \end{bmatrix}^{-1}, \quad z\in\Omega_{0j}^{'}, \quad j=1,2,\\
    &\begin{bmatrix} 1 & 0 \\ R_{0j}^{'}e^{-2\mathrm{i}t\theta} & 1 \end{bmatrix}, \quad z\in\Omega_{0j}^{'}, \quad j=3,4,\\
    &\begin{bmatrix} 1 & -R_{i1}e^{2\mathrm{i}t\theta} \\ 0 & 1 \end{bmatrix}^{-1}, \quad z\in\Omega_{i1}, \quad   i=1,2,3,4.\\
    &\begin{bmatrix} 1 & 0 \\ R_{i2}e^{-2\mathrm{i}t\theta} & 1 \end{bmatrix}, \quad z\in\Omega_{i2}, \quad   i=1,2,3,4.\\
    &\begin{bmatrix} 1 & -R_{i3}e^{2\mathrm{i}t\theta} \\ 0 & 1 \end{bmatrix}^{-1}, \quad z\in\Omega_{i3}, \quad   i=1,2,3,4.\\
    &\begin{bmatrix} 1 & 0 \\ R_{i4}e^{-2\mathrm{i}t\theta} & 1 \end{bmatrix}, \quad z\in\Omega_{i4}, \quad   i=1,2,3,4,\\
    &I, \quad elsewhere,
    \end{aligned}
        \right.
\label{mu}
\end{align}
where the functions $R_{0j}$, $R_{0j}^{'}$ and $R_{ij}$ are defined as the following two propositions.
\begin{proposition}[Opening lens at a small tangle]
$R_{0j}: \overline{\Omega}_{0j}\rightarrow \mathbb{C}$ and $R_{0j}^{'}: \overline{\Omega}_{0j}^{'}\rightarrow \mathbb{C}$, $j=1,2,3,4$ are continuous on $\overline{\Omega}_{0j}, \overline{\Omega}_{0j}^{'}$ respectively, $j=1,2,3,4$. Their boundary values are as follows:
\begin{align}
   R_{0j}(z)=
   &\left\{
        \begin{aligned}
    &\frac{\rho(z)}{1-\rho(z)\tilde{\rho}(z)}T_{-}^{-2}(z), \quad z\in l_{0j}, \quad j=1,2,\\
    &0, \quad \quad z\in\Sigma_{0j}, \quad j=1,2,
    \end{aligned}
        \right.\\
   R_{0j}(z)=
    &\left\{
        \begin{aligned}
    &\frac{\tilde{\rho}(z)}{1-\rho(z)\tilde{\rho}(z)}T_{+}^{2}(z), \quad z\in l_{0j}, \quad j=3,4,\\
    &0, \quad \quad z\in\Sigma_{0j}, \quad j=3,4.
    \end{aligned}
        \right.
\end{align}
\begin{align}
   R_{0j}^{'}(z)=
   &\left\{
     \begin{aligned}
      &\frac{\tilde{\rho}(z)}{1-\rho(z)\tilde{\rho}(z)}T_{+}^{2}(z), \quad z\in l_{0j}^{'}, \quad j=1,2,\\
      &0, \quad \quad z\in\Sigma_{0j}^{'}, \quad j=1,2,
     \end{aligned}
        \right.\\
   R_{0j}^{'}(z)=
    &\left\{
        \begin{aligned}
        &\frac{\rho(z)}{1-\rho(z)\tilde{\rho}(z)}T_{-}^{-2}(z), \quad z\in l_{0j}{'}, \quad j=3,4,\\
      &0, \quad \quad z\in\Sigma_{0j}{'}, \quad j=3,4.
    \end{aligned}
        \right.
\end{align}
Moreover, $R_{0j}$ and $R_{0j}^{'}$ have following property: for j = 1; 2; 3; 4;
\begin{equation}
\begin{split}
& |\bar{\partial}R_{0j}(z)|\lesssim |z|^{-\frac{1}{2}}+\left|\left(\frac{\rho}{1-\rho\tilde{\rho}}\right)^{'}({\rm Re}z)\right|, \quad j=1,2,\\
& |\bar{\partial}R_{0j}^{'}(z)|\lesssim |z|^{-\frac{1}{2}}+\left|\left(\frac{\rho}{1-\rho\tilde{\rho}}\right)^{'}({\rm Re} z)\right| \quad j=3,4,\\
& |\bar{\partial}R_{0j}(z)|\lesssim |z|^{-\frac{1}{2}}+\left|\left(\frac{\tilde{\rho}}{1-\rho\tilde{\rho}}\right)^{'}({\rm Re} z)\right|, \quad j=3,4,\\
& |\bar{\partial}R_{0j}^{'}(z)|\lesssim|z|^{-\frac{1}{2}}+\left|\left(\frac{\tilde{\rho}}{1-\rho\tilde{\rho}}\right)^{'}({\rm Re}z)\right| \quad j=1,2.\\
\end{split}
\end{equation}
\end{proposition}

\begin{proof}
Taking $R_{01}(z)$ as an example, its extensions can be constructed by:
\begin{equation}
  R_{01}(z)=\frac{\rho({\rm Re}z)}{1-\rho({\rm Re}z)\tilde{\rho}({\rm Re}z)}T_{-}^{-2}(z){\rm cos}(\frac{\pi}{2\phi(\xi)}\varphi),
\end{equation}
where $z=re^{\mathrm{i}\varphi}$, ${\rm Re}z=r{\rm cos}\varphi$.
Utilizing
\begin{equation}
  \bar{\partial}=\frac{1}{2}e^{\mathrm{i}\varphi}(\partial r+\mathrm{i}r^{-1}\partial{\varphi}),
\end{equation}
we have
\begin{equation}
  \bar{\partial}R_{01}(z)=-\frac{\pi}{4\phi(\xi)}\mathrm{i}r^{-1}e^{\mathrm{i}\varphi}{\rm sin}(\frac{\pi}{2\phi(\xi)}\varphi)\left(\frac{\rho}{1-\rho\tilde{\rho}}\right)({\rm Re}z)+\frac{1}{2}T_{-}^{-2}(z){\rm cos}(\frac{\pi}{2\phi(\xi)}\varphi)\left(\frac{\rho}{1-\rho\tilde{\rho}}\right)^{'}({\rm Re}z).
\end{equation}
Consequently,
\begin{equation}
\begin{split}
  &|\bar{\partial}R_{01}(z)|\lesssim \frac{1}{|z|}\left|\left(\frac{\rho}{1-\rho\tilde{\rho}}\right)({\rm Re}z)\right|+\left|\left(\frac{\rho}{1-\rho\tilde{\rho}}\right)^{'}({\rm Re}z)\right|,\\
  &\quad \quad \quad \quad \lesssim c_{1}|z|^{-\frac{1}{2}}+c_{2}\left|\left(\frac{\rho}{1-\rho\tilde{\rho}}\right)^{'}({\rm Re}z)\right|,
\end{split}
\end{equation}
where
\begin{equation}
\begin{split}
&\left|\left(\frac{\rho}{1-\rho\tilde{\rho}}\right)({\rm Re}z)\right|=\left|\int_{0}^{{\rm Re}z}\left(\frac{\rho}{1-\rho\tilde{\rho}}\right)^{'}(s)ds\right|,\\
&\leq \left\|\left(\frac{\rho}{1-\rho\tilde{\rho}}\right)^{'}(s)\right\|_{L^{2}(0,{\rm Re}z)}|z|^{\frac{1}{2}}
\end{split}
\end{equation}
\end{proof}

\begin{proposition}[Opening lens at stationary phase points]
$R_{ij}: \overline{\Omega}_{ij}\rightarrow \mathbb{C}$, $i, j=1,2,3,4$ are continuous on $\overline{\Omega}_{ij}$ with boundary values:
\begin{align}
 R_{i1}(z)=\left\{
        \begin{aligned}
    &\tilde{\rho}(z)T^{2}(z), \quad z\in\widehat{\zeta_{i}\alpha_{k}}\\
    &\tilde{\rho}(\zeta_{i})T^{2}(\zeta_{i})(z-\zeta_{i})^{-2\mathrm{i}\nu(\zeta_{i})}(1-\chi_{\mathcal{Z}}(z)), \quad z\in\Sigma_{i1}
    \end{aligned}
        \right.
\end{align}

\begin{align}
 R_{i2}(z)=\left\{
        \begin{aligned}
    &\frac{\rho(z)}{1-\rho(z)\tilde{\rho}(z)}T_{-}^{-2}(z), \quad z\in\widehat{\zeta_{i}\zeta_{l}}\\
    &\tilde{\rho}(\zeta_{i})T_{-}^{-2}(\zeta_{i})(z-\zeta_{i})^{2\mathrm{i}\nu(\zeta_{i})}(1-\chi_{\mathcal{Z}}(z)), \quad z\in\Sigma_{i2}
    \end{aligned}
        \right.
\end{align}

\begin{align}
 R_{i3}(z)=\left\{
        \begin{aligned}
    &\frac{\tilde{\rho}(z)}{1-\rho(z)\tilde{\rho}(z)}T_{+}^{2}(z), \quad z\in\widehat{\zeta_{i}\zeta_{l}}\\
    &\tilde{\rho}(\zeta_{i})T_{+}^{2}(\zeta_{i})(z-\zeta_{i})^{-2\mathrm{i}\nu(\zeta_{i})}(1-\chi_{\mathcal{Z}}(z)), \quad z\in\Sigma_{i3}
    \end{aligned}
        \right.
\end{align}

\begin{align}
 R_{i4}(z)=\left\{
        \begin{aligned}
    &\rho(z)T^{-2}(z), \quad z\in\widehat{\zeta_{i}\alpha_{k}}\\
    &\rho(\zeta_{i})T^{-2}(\zeta_{i})(z-\zeta_{i})^{2\mathrm{i}\nu(\zeta_{i})}(1-\chi_{\mathcal{Z}}(z)), \quad z\in\Sigma_{i4}
    \end{aligned}
        \right.
\end{align}
$i=1,2,3,4$, $k=1$ when $i=1,2$ and $k=2$ when $i=3,4$, $l=5$ when $i=1,2$ and $l=6$ when $i=3,4$, and $\chi_{\mathcal{Z}}(z)\in C_0^\infty$ is defined as
\begin{align}
\chi_{\mathcal{Z}}(z)=\left\{
        \begin{aligned}
    &1, \quad {\rm dist(z,\mathcal{Z}\cup\mathcal{\hat{Z}})}<\varrho/3\\
    &0, \quad {\rm dist(z,\mathcal{Z}\cup\mathcal{\hat{Z}})}>2\varrho/3.
    \end{aligned}
        \right.
\end{align}
Moreover, $R_{ij}(z), i,j=1,2,3,4$ have following properties:
\begin{equation}
\begin{split}
&|R_{ij}(z)|\lesssim (1+|z|^{2})^{-\frac{1}{4}}+c,\\
& |\bar{\partial}R_{ij}(z)|=\mathcal{O}(1), \quad i,j=1,2,3,4.
\end{split}
\end{equation}
\end{proposition}
\begin{proof}
We give the details for $R_{11}(z)$ only. The other cases are easily inferred. The continuous extension of $R_{11}(z)$ on $\Omega_{11}$ can be constructed by
\begin{equation}
   R_{11}(z)=[\gamma\tilde{\rho}(z)+(1-\gamma)g_{11}(z)]T^{2}(z)(1-\chi_{\mathcal{Z}}(z)), \quad \gamma\in [0,1],
\end{equation}
where $g_{11}(z)$ is defined as
\begin{equation}
   g_{11}(z)=\tilde{\rho}(\zeta_{1})T^{2}(\zeta_{1})T^{-2}(z)(z-\zeta_{1})^{-2\mathrm{i}\nu(\zeta_{1})}.
\end{equation}
Firstly, we have
\begin{equation}
\begin{split}
|\tilde{\rho}(\zeta_{1})(z-\zeta_{1})^{-2\mathrm{i}\nu(\zeta_{1})}|& =|\tilde{\rho}(\zeta_{1})|e^{2\nu(\zeta_{1}){\rm arg}(z-\zeta_{1})}\\
&  \leq |\tilde{\rho}(\zeta_{1})|(1- |\tilde{\rho}(\zeta_{1})|^{2})^{-1}.
\end{split}
\end{equation}
Recall $\tilde{\rho}(z)\in H^1 (\Gamma)$, we obtain
\begin{equation}
   |\tilde{\rho}(z)|\lesssim\left|(1+z^{2})^{-\frac{1}{4}}\right|\lesssim (1+|z|^{2})^{-\frac{1}{4}}.
\end{equation}
Thus,
\begin{equation}
\begin{split}
&|R_{11}(z)|\lesssim (1+|z|^{2})^{-\frac{1}{4}}+c,\\
&|\bar{\partial}R_{11}(z)|\leq|[\gamma\tilde{\rho}(z)+(1-\gamma)g_{11}(z)]T^{2}(z)\bar{\partial}\chi_{\mathcal{Z}}(z)|=\mathcal{O}(1).
\end{split}
\end{equation}
\end{proof}

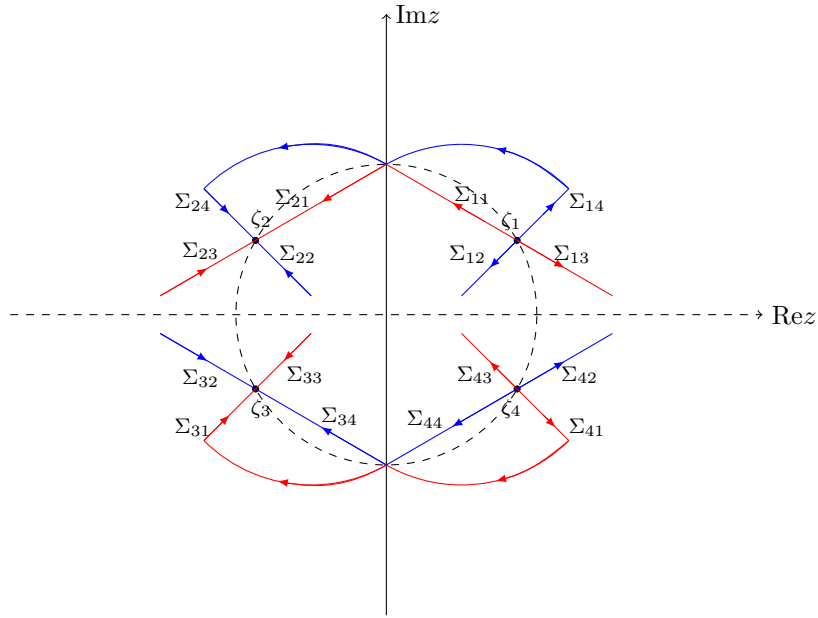
\begin{figure}[H]
\begin{center}
\begin{tikzpicture}[node distance=2cm]
\draw[->,dashed](-5,0)--(5,0)node[right]{ \textcolor{black}{${\rm Re} z$}};
\draw[->](0,-4)--(0,4)node[right]{\textcolor{black}{${\rm Im} z$}};
\draw[dashed](0,0)circle(2cm);

\node  [right]  at (1.4,1.245) {\footnotesize $\zeta_{1}$};
\node  [left]  at (-1.4,1.28) {\footnotesize $\zeta_{2}$};
\node  [left]  at (-1.4,-1.245) {\footnotesize $\zeta_{3}$};
\node  [right]  at (1.4,-1.245) {\footnotesize $\zeta_{4}$};

\draw[fill] (1.7386,0.9886) circle [radius=0.04];
\draw[fill] (-1.7386,0.9886) circle [radius=0.04];
\draw[fill] (-1.7386,-0.9886) circle [radius=0.04];
\draw[fill] (1.7386,-0.9886) circle [radius=0.04];

\draw [color=red](-1,-0.25)--(-2.425,-1.675)  node[right, scale=1] {};
\draw [color=red](1,-0.25)--(2.425,-1.675)  node[right, scale=1] {};
\draw [color=blue](1,0.25)--(2.425,1.675)  node[right, scale=1] {};
\draw [color=blue](-1,0.25)--(-2.425,1.675)  node[right, scale=1] {};
\draw [color=red](0,2)--(3.006,0.251)  node[right, scale=1] {};
\draw [color=red](0,2)--(-3.006,0.251)  node[right, scale=1] {};
\draw [color=blue](0,-2)--(3.006,-0.251)  node[right, scale=1] {};
\draw [color=blue](0,-2)--(-3.006,-0.251)  node[right, scale=1] {};
\draw[color=blue] (1,2.266) arc (90:45:2.016);
\draw[color=blue] (1,2.266) arc (90:120:2.016);
\draw[color=blue] (-1,2.266) arc (90:60:2.016);
\draw[color=blue] (-1,2.266) arc (90:135:2.016);
\draw[color=red] (-1,-2.266) arc (270:300:2.016);
\draw[color=red] (-1,-2.266) arc (270:225:2.016);
\draw[color=red] (1,-2.266) arc (270:315:2.016);
\draw[color=red] (1,-2.266) arc (270:240:2.016);

          \draw [color=blue] [-latex] (2.425,1.675) to  [out=140, in=345]  (1.45,2.215);
          \draw [color=blue] [-latex] (2.0818,1.3318) -- (2.2534,1.5034);
          \node [right]  at (2.3,1.5) {\footnotesize $\Sigma_{14}$};
          \draw [color=blue] [-latex] (1.7386,0.9886) -- (1.3693,0.6193);
          \node [left]  at (1.45,0.8) {\footnotesize $\Sigma_{12}$};
          \draw [color=red] [-latex] (1.30395,1.24145)-- (0.8693,1.4943);
          \node [right]  at (0.7793,1.5943) {\footnotesize $\Sigma_{11}$};
          \draw [color=red] [-latex] (1.7386,0.9886) -- (2.3723,0.6198);
          \node [right]  at (2.1,0.8) {\footnotesize $\Sigma_{13}$};

          \draw [color=blue] [-latex]  (0,2) to  [out=150, in=10]  (-1.45,2.215);
          \draw [color=blue] [-latex]  (-2.425,1.675)-- (-2.0818,1.3318);
          \node [left]  at (-2.2,1.5) {\footnotesize $\Sigma_{24}$};
          \draw [color=blue] [-latex] (-1,0.25) -- (-1.3693,0.6193);
          \node [right]  at (-1.55,0.8) {\footnotesize $\Sigma_{22}$};
          \draw [color=red] [-latex]  (0,2) -- (-0.8693,1.4943);
          \node [left]  at (-0.8693,1.55) {\footnotesize $\Sigma_{21}$};
          \draw [color=red] [-latex] (-3.006,0.251) -- (-2.3723,0.6198);
          \node [left]  at (-2.1,0.85) {\footnotesize $\Sigma_{23}$};

          \draw [color=red] [-latex]  (0,-2) to  [out=210, in=350]  (-1.45,-2.215);
          \draw [color=red] [-latex]  (-2.425,-1.675)-- (-2.0818,-1.3318);
          \node [left]  at (-2.2,-1.5) {\footnotesize $\Sigma_{31}$};
          \draw [color=red] [-latex] (-1,-0.25) -- (-1.3693,-0.6193);
          \node [right]  at (-1.45,-0.8) {\footnotesize $\Sigma_{33}$};
          \draw [color=blue] [-latex]  (0,-2) -- (-0.8693,-1.4943);
          \node [right]  at (-0.99,-1.35) {\footnotesize $\Sigma_{34}$};
          \draw [color=blue] [-latex] (-3.006,-0.251) -- (-2.3723,-0.6198);
          \node [left]  at (-2.1,-0.85) {\footnotesize $\Sigma_{32}$};

          \draw [color=red] [-latex] (2.425,-1.675) to [out=220, in=15]  (1.45,-2.215);
          \draw [color=red] [-latex] (2.0818,-1.3318) -- (2.2534,-1.5034);
          \node [right]  at (2.3,-1.5) {\footnotesize $\Sigma_{41}$};
          \draw [color=red] [-latex] (1.7386,-0.9886) -- (1.3693,-0.6193);
          \node [left]  at (1.55,-0.8) {\footnotesize $\Sigma_{43}$};
          \draw [color=blue] [-latex] (1.7386,-0.9886) -- (0.8693,-1.4943);
          \node [left]  at (0.8893,-1.3943) {\footnotesize $\Sigma_{44}$};
          \draw [color=blue] [-latex] (1.7386,-0.9886) -- (2.3723,-0.6198);
          \node [right]  at (2.2,-0.8) {\footnotesize $\Sigma_{42}$};
\end {tikzpicture}
\caption{\footnotesize The jump contours $\Sigma^{(2)}$ of $m^{(2)}$.}
\label{Sigma2}
\end{center}
\end{figure}

Define $\Sigma^{(2)}=\underset{i=1,2,3,4}\cup L_{i}$,  which can be referred in the following Figure \ref{Sigma2}.
We now use $\mathcal{R}^{(2)}(z)$ to define a new transformation
\begin{equation}
m^{(2)}(z)=m^{(1)}(z)\mathcal{R}^{(2)}(z),
\end{equation}
which satisfies the following mixed $\bar{\partial}$-RH problem.
\begin{RHP}\label{rhp2}
Find a $2\times 2$ matrix-valued function $m^{(2)}(x,t;z)$ such that
\begin{itemize}
    \item[*] $m^{(2)}(z)$ is continuous in $\mathbb{C}\backslash (\Sigma^{(2)}\cup\mathcal{Z}\cup\hat{\mathcal{Z}})$.

    \item[*] Jump relation: $m^{(2)}_{+}(z)=m^{(2)}_{-}(z)v^{(2)}(z)$, $z\in\Sigma_{(2)}$, where
    \begin{equation}
       v^{(2)}(z)=[\mathcal{R}^{(2)}_{-}]^{-1}v^{(1)}\mathcal{R}^{(2)}_{+}=I+(1-\chi_{\mathcal{Z}}(z))\delta v^{(2)}(z),
    \end{equation}
    \begin{align}
     \delta v^{(2)}(z)=\left\{
        \begin{aligned}
    & \begin{bmatrix} 0 & -\tilde{\rho}(\zeta_{i})T^{2}(\zeta_{i})(z-\zeta_{i})^{2\mathrm{i}\nu(\zeta_{i})}e^{2\mathrm{i}t\theta(z)} \\  0 & 0 \end{bmatrix}, \quad z\in\Sigma_{i1}, i=1,2,3,4,\\
    &\begin{bmatrix} 0 & 0 \\ \frac{\rho(\zeta_{i})}{1-\rho(\zeta_{i})\tilde{\rho}(\zeta_{i})}T_{-}^{-2}(\zeta_{i})(z-\zeta_{i})^{-2\mathrm{i}\nu(\zeta_{i})}e^{-2\mathrm{i}t\theta(z)}  & 0 \end{bmatrix}, \quad z\in\Sigma_{i2}, i=1,2,3,4,\\
    &\begin{bmatrix} 0 &  -\frac{\tilde{\rho}(\zeta_{i})}{1-\rho(\zeta_{i})\tilde{\rho}(\zeta_{i})}T_{+}^{-2}(\zeta_{i})(z-\zeta_{i})^{2\mathrm{i}\nu(\zeta_{i})}e^{2\mathrm{i}t\theta(z)} \\ 0  & 0 \end{bmatrix}, \quad z\in\Sigma_{i3}, i=1,2,3,4,\\
    & \begin{bmatrix} 0 & 0 \\  -\rho(\zeta_{i})T^{-2}(\zeta_{i})(z-\zeta_{i})^{-2\mathrm{i}\nu(\zeta_{i})}e^{-2\mathrm{i}t\theta(z)} & 0 \end{bmatrix}, \quad z\in\Sigma_{i4}, i=1,2,3,4,\\
    & \begin{bmatrix} 0 & 0\\ 0 & 0 \end{bmatrix}, \quad elsewhere.
    \end{aligned}
        \right.
    \end{align}

    \item[*]For $z\in\mathbb{C}\backslash (\Sigma^{(2)}\cup\mathcal{Z}\cup\hat{\mathcal{Z}})$,
      \begin{equation}
      \bar{\partial}m^{(2)}=m^{(2)}\bar{\partial}\mathcal{R}^{(2)},
      \end{equation}
      where
      \begin{align}
         \bar{\partial}\mathcal{R}^{(2)}=\left\{
        \begin{aligned}
        &\begin{bmatrix} 0 & -\bar{\partial}R_{i1}e^{2\mathrm{i}t\theta} \\ 0 & 0 \end{bmatrix}, \quad z\in\Omega_{i1}, i=1,2,3,4,\\
        &\begin{bmatrix} 0 & 0 \\ \bar{\partial}R_{i2}e^{-2\mathrm{i}t\theta} & 0 \end{bmatrix}, \quad z\in\Omega_{i2}, i=1,2,3,4,\\
        &\begin{bmatrix} 0 & \bar{\partial}R_{i3}e^{2\mathrm{i}t\theta} \\ 0 & 0 \end{bmatrix}, \quad z\in\Omega_{i3}, i=1,2,3,4,\\
        &\begin{bmatrix} 0 & 0 \\ \bar{\partial}R_{i4}e^{-2\mathrm{i}t\theta} & 0 \end{bmatrix}, \quad z\in\Omega_{i4}, i=1,2,3,4,\\
        &\begin{bmatrix} 0 & 0 \\ 0 & 0 \end{bmatrix}, \quad elsewhere.\\
      \end{aligned}
        \right.
      \end{align}

    \item[*] Asymptotic behavior
    \begin{equation}\label{rhp1aym}
    m^{(2)}(x,t;z)=I+\mathcal{O}(z^{-1}), \quad  z\rightarrow\infty.
    \end{equation}

    \item[*] Residue conditions
    \begin{align}
    \underset{z=\eta_k}{\rm Res}m^{(2)}(z)=\left\{
    \begin{aligned}
    &\lim_{z\rightarrow\eta_{k}}m^{(2)}(z)\begin{bmatrix} 0 & 0 \\ A[\eta_{k}]T^{-2}(\eta_{k}) e^{-2it\theta(\eta_{k})} & 0 \end{bmatrix}, \quad k\in\triangle,\\
    &\lim_{z\rightarrow\eta_{k}}m^{(2)}(z)\begin{bmatrix} 0 & \frac{1}{A[\eta_{k}]}[(\frac{1}{T})'(\eta_{k})]^{-2}e^{2it\theta(\eta_{k})} \\ 0 & 0 \end{bmatrix}, \quad k\in\nabla,
    \end{aligned}
        \right.
    \end{align}
    \begin{align}\label{rhp1resb}
    \underset{z=\hat{\eta}_k}{\rm Res}m^{(2)}(z)=\left\{
    \begin{aligned}
    &\lim_{z\rightarrow\hat{\eta}_k}m^{(2)}(z)\begin{bmatrix} 0 & A[\hat{\eta}_k]T^{2}(\hat{\eta}_{k}) e^{2it\theta(\hat{\eta}_{k})} \\ 0 & 0 \end{bmatrix}, \quad k\in\triangle,\\
    &\lim_{z\rightarrow\hat{\eta}_k}m^{(2)}(z)\begin{bmatrix} 0 & 0 \\ \frac{1}{A[\hat{\eta}_k]}\frac{1}{[T'(\hat{\eta}_k)]^{2}} e^{-2it\theta(\hat{\eta}_{k})} & 0 \end{bmatrix}, \quad k\in\nabla.
    \end{aligned}
        \right.
    \end{align}
\end{itemize}
\end{RHP}

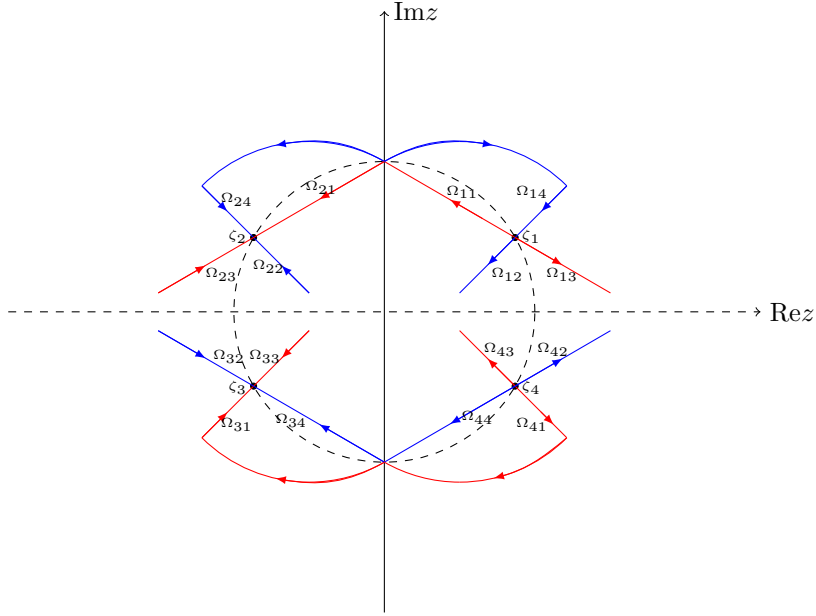
\begin{figure}[H]
\begin{center}
\begin{tikzpicture}[node distance=2cm]
\draw[->,dashed](-5,0)--(5,0)node[right]{ \textcolor{black}{${\rm Re} z$}};
\draw[->](0,-4)--(0,4)node[right]{\textcolor{black}{${\rm Im} z$}};
\draw[dashed](0,0)circle(2cm);
\node  [right]  at (1.7,1) {\tiny $\zeta_{1}$};
\node  [left]  at (-1.7,1) {\tiny $\zeta_{2}$};
\node  [left]  at (-1.7,-1) {\tiny $\zeta_{3}$};
\node  [right]  at (1.7,-1) {\tiny $\zeta_{4}$};

\draw[fill] (1.7386,0.9886) circle [radius=0.04];
\draw[fill] (-1.7386,0.9886) circle [radius=0.04];
\draw[fill] (-1.7386,-0.9886) circle [radius=0.04];
\draw[fill] (1.7386,-0.9886) circle [radius=0.04];

\draw [color=red](-1,-0.25)--(-2.425,-1.675)  node[right, scale=1] {};
\draw [color=red](1,-0.25)--(2.425,-1.675)  node[right, scale=1] {};
\draw [color=blue](1,0.25)--(2.425,1.675)  node[right, scale=1] {};
\draw [color=blue](-1,0.25)--(-2.425,1.675)  node[right, scale=1] {};
\draw [color=red](0,2)--(3.006,0.251)  node[right, scale=1] {};
\draw [color=red](0,2)--(-3.006,0.251)  node[right, scale=1] {};
\draw [color=blue](0,-2)--(3.006,-0.251)  node[right, scale=1] {};
\draw [color=blue](0,-2)--(-3.006,-0.251)  node[right, scale=1] {};
\draw[color=blue] (1,2.266) arc (90:45:2.016);
\draw[color=blue] (1,2.266) arc (90:120:2.016);
\draw[color=blue] (-1,2.266) arc (90:60:2.016);
\draw[color=blue] (-1,2.266) arc (90:135:2.016);
\draw[color=red] (-1,-2.266) arc (270:300:2.016);
\draw[color=red] (-1,-2.266) arc (270:225:2.016);
\draw[color=red] (1,-2.266) arc (270:315:2.016);
\draw[color=red] (1,-2.266) arc (270:240:2.016);

\draw [color=blue] [-latex]  (0,2) to  [out=30, in=170]  (1.45,2.215);
\draw [color=blue] [-latex] (2.425,1.675) -- (2.0818,1.3318);
\node [left]  at (2.3,1.5943) {\tiny  $\Omega_{14}$};
\draw [color=blue] [-latex] (1.7386,0.9886) -- (1.3693,0.6193);
\node [right]  at (1.3,0.5) {\tiny  $\Omega_{12}$};
\draw [color=red] [-latex] (1.30395,1.24145)-- (0.8693,1.4943);
\node [right]  at (0.7,1.5943) {\tiny  $\Omega_{11}$};
\draw [color=red] [-latex] (1.7386,0.9886) -- (2.3723,0.6198);
\node [left]  at (2.7,0.5) {\tiny  $\Omega_{13}$};

\draw [color=blue] [-latex]  (0,2) to  [out=150, in=10]  (-1.45,2.215);
\draw [color=blue] [-latex]  (-2.425,1.675)-- (-2.0818,1.3318);
\node [right]  at (-2.3,1.5) {\tiny  $\Omega_{24}$};
\draw [color=blue] [-latex] (-1,0.25) -- (-1.3693,0.6193);
\node [left]  at (-1.2,0.6) {\tiny  $\Omega_{22}$};
\draw [color=red] [-latex]  (0,2) -- (-0.8693,1.4943);
\node [left]  at (-0.5,1.65) {\tiny  $\Omega_{21}$};
\draw [color=red] [-latex] (-3.006,0.251) -- (-2.3723,0.6198);
\node [right]  at (-2.5,0.5) {\tiny  $\Omega_{23}$};

\draw [color=red] [-latex]  (0,-2) to  [out=210, in=350]  (-1.45,-2.215);
\draw [color=red] [-latex]  (-2.425,-1.675)-- (-2.0818,-1.3318);
\node [right]  at (-2.3,-1.5) {\tiny  $\Omega_{31}$};
\draw [color=red] [-latex] (-1,-0.25) -- (-1.3693,-0.6193);
\node [left]  at (-1.25,-0.6) {\tiny $\Omega_{33}$};
\draw [color=blue] [-latex]  (0,-2) -- (-0.8693,-1.4943);
\node [left]  at (-0.9,-1.45) {\tiny  $\Omega_{34}$};
\draw [color=blue] [-latex] (-3.006,-0.251) -- (-2.3723,-0.6198);
\node [right]  at (-2.4,-0.6) {\tiny  $\Omega_{32}$};

\draw [color=red] [-latex] (2.425,-1.675) to [out=220, in=15]  (1.45,-2.215);
\draw [color=red] [-latex] (2.0818,-1.3318) -- (2.2534,-1.5034);
\node [left]  at (2.3,-1.5) {\tiny  $\Omega_{41}$};
\draw [color=red] [-latex] (1.7386,-0.9886) -- (1.3693,-0.6193);
\node [right]  at (1.2,-0.5) {\tiny  $\Omega_{43}$};
\draw [color=blue] [-latex] (1.7386,-0.9886) -- (0.8693,-1.4943);
\node [right]  at (0.9,-1.4) {\tiny  $\Omega_{44}$};
\draw [color=blue] [-latex] (1.7386,-0.9886) -- (2.3723,-0.6198);
\node [left]  at (2.58,-0.5) {\tiny  $\Omega_{42}$};
\end {tikzpicture}
\caption{\footnotesize Domains of $m^{(2)}$.}.
\label{domainsofm2}
\end{center}
\end{figure}

\hspace*{\parindent}

\section{Decomposition of the mixed $\bar{\partial}$-RH problem}

To solve RHP \ref{rhp2}, we decompose $m^{(2)}$ into a pure RH problem for $m^{(2)}_{RHP}$  with  $\bar{\partial}\mathcal{R}^{(2)}\equiv0$ and a pure $\bar{\partial}$-problem with nonzero $\bar{\partial}$-derivatives, which can be shown as the following structure
\begin{align}
m^{(2)}=m^{(3)}m^{(2)}_{RHP}\left\{
        \begin{aligned}
    &\bar{\partial}\mathcal{R}^{(2)}\equiv0\longrightarrow m^{(2)}_{RHP},\\
    &\bar{\partial}\mathcal{R}^{(2)}\neq0\longrightarrow m^{(3)}=m^{(2)}[m^{(2)}_{RHP}]^{-1}.
    \end{aligned}
        \right.
\end{align}
$m^{(2)}_{RHP}$ satisfies the following RH problem.
\begin{RHP}\label{m2php}
Find a $2\times 2$ matrix-valued function $m^{(2)}_{RHP}(x,t;z)$ such that
\begin{itemize}
    \item[*] $m^{(2)}_{RHP}(z)$ is analytic in $\mathbb{C}\backslash (\Sigma^{(2)}\cup\mathcal{Z}\cup\hat{\mathcal{Z}})$.

    \item[*] Jump relation: \ $m^{(2)}_{PHP+}(z)=m^{(2)}_{RHP-}(z)v^{(2)}(z)$, $z\in\Sigma^{(2)}$, where

    \item[*]For $z\in\mathbb{C}\backslash (\Sigma^{(2)}\cup\mathcal{Z}\cup\hat{\mathcal{Z}})$,
      \begin{equation}
      \bar{\partial}m^{(2)}\equiv 0.
      \end{equation}

    \item[*] Asymptotic behavior
    \begin{equation}\label{rhp1aym}
    m^{(2)}_{RHP}(x,t;z)=I+\mathcal{O}(z^{-1}), \quad  z\rightarrow\infty.
    \end{equation}

    \item[*] $m^{(2)}_{RHP}$ has same jump matrix and residue conditions as $m^{(2)}$.
\end{itemize}
\end{RHP}
Define $U_{\zeta_{i}}$ as
\begin{equation}
U_{\zeta_{i}}=\left\{z: |z-\zeta_{i}|<\varrho\right\}, i=1,2,3,4.
\end{equation}
Then the jump matrix in RHP \ref{m2php} admits the following estimates.
\begin{proposition}\label{v2guji}
As $t\rightarrow\infty$, there exist positive constants $K_{i}, i=1,2,3,4$ that the jump matrix $v^{(2)}(z)$ admits the following estimate
 \begin{equation}
\begin{split}
& \|v^{(2)}(z)-I\|_{L^{\infty}(\Sigma_{ij}\backslash U_{\zeta_{i}})}=\mathcal{O}(e^{-K_{i}|z-\zeta_{i}|t}),\\
& \|v^{(2)}(z)-I\|_{L^{\infty}(\Sigma_{ij}\cap U_{\zeta_{i}})}=\mathcal{O}(|z-\zeta_{i}|^{-1/2}t^{-1/2}), \quad i,j=1,2,3,4.
\end{split}
\end{equation}
\end{proposition}
\begin{proof}
We prove the case of $z\in\Sigma_{11}\backslash U_{\zeta_{1}}$ and $z\in\Sigma_{11}\cap U_{\zeta_{1}}$, then the another case can be
proved in similar way. Denote $\zeta_{1}=a_{1}+\mathrm{i}b_{1}$, we can obtain
\begin{equation}
\xi=-3(\zeta_{1}^{2}+\zeta_{1}^{-2}).
\end{equation}
For $z\in\Sigma_{11}$, $z=\zeta_{1}+|z-\zeta_{1}|e^{\mathrm{i}\frac{3}{4}\pi}$, we have
\begin{equation}
{\rm Re}[2\mathrm{i}t\theta(z)]=-t(1-|z|^{-2})\left(b_{1}+\frac{\sqrt{2}}{2}|z-\zeta_{1}|\right)f(z),
\end{equation}
where
\begin{equation}
f(z)=\left[12b_{1}^{2}-9+(1+|z|^{-2}+|z|^{-4})(3-4b_{1}^{2}-\sqrt{2}(3\sqrt{1-b_{1}^{2}}+b_{1})|z-\zeta_{1}|+|z-\zeta_{1}|^{2})\right].
\end{equation}
Observe that
\begin{equation}
\begin{split}
f(z)&\lesssim|z-\zeta_{1}|(1+|z|^{-2}+|z|^{-4})\left(|z-\zeta_{1}|-\sqrt{2}(3\sqrt{1-b_{1}^{2}}+b_{1})\right)\\
&\leq-\sqrt{2}|z-\zeta_{1}|\left(1+\frac{1}{2}(1+b_{1})^{2}+\frac{1}{4}(1+b_{1})^{4}\right)\left(2\sqrt{1-b_{1}^{2}}+b_{1}\right)\\
&\lesssim -|z-\zeta_{1}|,
\end{split}
\end{equation}
we obtain
\begin{equation}
\|v^{(2)}(z)-I\|_{L^{\infty}(\Sigma_{11})}=\left|-\tilde{\rho}(\zeta_{1})T^{2}(\zeta_{1})(z-\zeta_{1})^{2\mathrm{i}\nu(\zeta_{1})}e^{2\mathrm{i}t\theta(z)} \right|\lesssim e^{-K_{1}|z-\zeta_{1}|t}.
\end{equation}
Thus, for $z\in\Sigma_{11}\backslash U_{\zeta_{1}}$,
\begin{equation}
 \|v^{(2)}(z)-I\|_{L^{\infty}(\Sigma_{11}\backslash U_{\zeta_{1}})}\lesssim e^{-K_{1}|z-\zeta_{1}|t},
\end{equation}
while for $z\in\Sigma_{11}\cap U_{\zeta_{1}}$,
\begin{equation}
 \|v^{(2)}(z)-I\|_{L^{\infty}(\Sigma_{11}\cap U_{\zeta_{1}})}\lesssim e^{-K_{1}|z-\zeta_{1}|t}\lesssim |z-\zeta_{1}|^{-1/2}t^{-1/2}.
\end{equation}
\end{proof}
This proposition implies that the jump matrix $v^{(2)}$ uniformly goes to $I$ in terms of exponentially small error outside $U_{\zeta_{i}}, i=1,2,3,4$. So we can ignore the jump relation of $m^{(2)}_{RHP}$ outside the $U_{\zeta_{i}}, i=1,2,3,4$ and decompose $m^{(2)}_{RHP}$ in the following form
\begin{align}
m^{(2)}_{RHP}(z)=\left\{
        \begin{aligned}
    &E(z)m^{(out)}(z), \quad z\in D\backslash (\underset{i=1,2,3,4}\cup U_{\zeta_{i}}),\\
    &E_(z)m^{(out)}(z)m^{(lo)}(z), \quad z\in U_{\zeta_{i}}, \quad i=1,2,3,4.
    \end{aligned}
        \right.
\end{align}

In this decomposition,  $m^{(out)}$ solves the pure RHP obtained by ignoring the jump conditions of RHP \ref{m2php}, which will be solved in next Section \ref{out}; $m^{(lo)}$ uses parabolic cylinder functions to build a matrix to match jumps of $m^{(2)}_{RHP}$ in a neighborhood $\zeta_{i}, i=1,2,3,4$ which is shown in Section \ref{in}.  $E(z)$ is an error function and a solution of a small norm Riemann-Hilbert problem which is shown in Section \ref{erro}.

\subsection{Analysis on the pure RH Problem}\label{purerhp}

\subsubsection{Outer model RH problem}\label{out}

In this subsection, we build a reflectionless case of RHP \ref{rhp0} to show that its solution can approximated with $m^{(out)}$. As the main contribution to $m^{(out)}$ comes from its scattering data
 \begin{equation}\label{sigmad}
   \sigma_{d}=\left\{\left(\eta_{k},{c}_{k}\right)\right\}_{k=1}^{2N_{1}+N_{2}},
 \end{equation}
 \begin{align}
 \tilde{c}_{k}=\left\{
        \begin{aligned}
    &A[\eta_{k}]T^{-2}(\eta_{k}), \quad z\in\triangle\\
    &\frac{1}{A[\hat{\eta}_k]}\frac{1}{[T'(\hat{\eta}_k)]^{2}}, \quad z\in\nabla,
    \end{aligned}
        \right.
\end{align}
thus $m^{(out)}(z)=m^{(out)}(z|\sigma_{d}^{out})$, and $m^{(out)}(z|\sigma_{d}^{out})$ can be construct as
\begin{RHP}\label{rhpmout}
Find a $2\times 2$ matrix-valued function $m^{(out)}(z|\sigma_{d}^{out})$ such that
\begin{itemize}
  \item[*]$m^{(out)}(z|\sigma_{d}^{out})$ is analytic in $\mathbb{C}\backslash\left(\mathcal{Z}\cup\hat{\mathcal{Z}}\right)$.
  \item[*]$ m^{(out)}(z|\sigma_{d}^{out})=I+\mathcal{O}(z^{-1}), \quad  z\rightarrow\infty$.
  \item[*]Residue conditions: for $\eta_{k}\in \mathcal{Z}$ and $\hat{\eta}_{k}\in \hat{\mathcal{Z}}$,
    \begin{align}
    \underset{z=\eta_k}{\rm Res}m^{(out)}(z|\sigma_{d}^{out})=\left\{
    \begin{aligned}
    &\lim_{z\rightarrow\eta_{k}}m^{(out)}(z|\sigma_{d}^{out})\begin{bmatrix} 0 & 0 \\ A[\eta_{k}]T^{-2}(\eta_{k}) e^{-2it\theta(\eta_{k})} & 0 \end{bmatrix}, \quad k\in\triangle,\\
    &\lim_{z\rightarrow\eta_{k}}m^{(out)}(z|\sigma_{d}^{out})\begin{bmatrix} 0 & \frac{1}{A[\eta_{k}]}[(\frac{1}{T})'(\eta_{k})]^{-2}e^{2it\theta(\eta_{k})} \\ 0 & 0 \end{bmatrix}, \quad k\in\nabla,
    \end{aligned}
        \right.
    \end{align}
    \begin{align}
    \underset{z=\hat{\eta}_k}{\rm Res}m^{(out)}(z|\sigma_{d}^{out})=\left\{
    \begin{aligned}
    &\lim_{z\rightarrow\hat{\eta}_k}m^{(out)}(z|\sigma_{d}^{out})\begin{bmatrix} 0 & A[\hat{\eta}_k]T^{2}(\hat{\eta}_{k}) e^{2it\theta(\hat{\eta}_{k})} \\ 0 & 0 \end{bmatrix}, \quad k\in\triangle,\\
    &\lim_{z\rightarrow\hat{\eta}_k}m^{(out)}(z|\sigma_{d}^{out})\begin{bmatrix} 0 & 0 \\ \frac{1}{A[\hat{\eta}_k]}\frac{1}{[T'(\hat{\eta}_k)]^{2}} e^{-2it\theta(\hat{\eta}_{k})} & 0 \end{bmatrix}, \quad k\in\nabla.
    \end{aligned}
        \right.
    \end{align}
\end{itemize}
\end{RHP}
Furthermore, we set a new RH problem $m(z|\sigma_{d})$, which is a reflectionless case $\rho(z)=\tilde{\rho}(z)=0$ of original RHP \ref{rhp0}, and satisfies
\begin{RHP}\label{rhpsigmad}
Find a $2\times 2$ matrix-valued function $m(z|\sigma_{d})$ such that
\begin{itemize}
\item[*]$m(z|\sigma_{d})$ is analytic in $\mathbb{C}\backslash\left(\mathcal{Z}\cup\hat{\mathcal{Z}}\right)$.
\item[*]$ m(z|\sigma_{d})=I+\mathcal{O}(z^{-1}), \quad  z\rightarrow\infty$.
\item[*]Residue conditions: for $\eta_{k}\in \mathcal{Z}$ and $\hat{\eta}_{k}\in \hat{\mathcal{Z}}$,
           \begin{equation}
            \begin{split}
             & \underset{z=\eta_k}{\rm Res}m(z|\sigma_{d})=\lim_{z\rightarrow\eta_{k}}m(z|\sigma_{d})N_{k},\\
             & \underset{z=\hat{\eta}_k}{\rm Res}m(z|\sigma_{d})=-\hat{\eta}_{k}^{2}\lim_{z\rightarrow\hat{\eta}_{k}}m(z|\sigma_{d})\sigma_{2}N_{k}\sigma_{2},
            \end{split}
           \end{equation}
where
  \begin{equation}
   \sigma_{d}=\left\{\left(\eta_{k},A[\eta_{k}]\right)\right\}_{k=1}^{2N_{1}+N_{2}},
  \end{equation}
is the corresponding scattering data.
  \begin{equation}
  N_{k}=\begin{bmatrix} 0 & 0 \\ \gamma_{k}(x,t) & 0 \end{bmatrix},\quad \gamma_{k}(x,t)=A[\eta_{k}]e^{-2\mathrm{i}\theta(x,t,\eta_{k})}
  \end{equation}
\end{itemize}
\end{RHP}
\begin{proposition}\label{unique}
For giving scattering data
\begin{equation}
\sigma_{d}=\left\{\left(\eta_{k},A[\eta_{k}]\right)\right\}_{k=1}^{2N_{1}+N_{2}}.
\end{equation}
RHP \ref{rhpsigmad} has an unique solution.
\end{proposition}
\begin{proof}
The uniqueness of the solution can be deduced from the Liouville's theorem directly. By Plemelj's formulae, we can construct $m(z|\sigma_{d})$ as the following form
\begin{equation}\label{plemelj}
m(z|\sigma_{d})=I+\frac{\mathrm{i}}{z}\sigma_{3}Q_{-}+\sum_{k=1}^{2N_{1}+N_{2}}\frac{\underset{z=\eta_k}{\rm Res}m(z|\sigma_{d})}{z-\eta_{k}}+\sum_{k=1}^{2N_{1}+N_{2}}\frac{\underset{z=\hat{\eta}_k}{\rm Res}m(z|\sigma_{d})}{z-\hat{\eta}_{k}},
\end{equation}
where
\begin{equation}\label{resetak1}
\underset{z=\eta_k}{\rm Res}m(z|\sigma_{d})
=a(\eta_{k})N_{k}=\begin{bmatrix} a_{12}(\eta_{k})\gamma_{k} & 0 \\ a_{22}(\eta_{k})\gamma_{k} & 0 \end{bmatrix}
\triangleq\begin{bmatrix} \alpha_{k} & 0 \\ \beta_{k} & 0 \end{bmatrix}.
\end{equation}
Take advantage of the symmetry: $m(-z^{-1}|\sigma_{d})=-\mathrm{i} z m(z|\sigma_{d})\begin{bmatrix} 0 & \sigma \\ -1 & 0 \end{bmatrix}, \sigma=\pm1$, we have
\begin{equation}\label{resetak2}
\underset{z=\hat{\eta}_k}{\rm Res}m(z|\sigma_{d})
=\begin{bmatrix} 0 & \mathrm{i}\hat{\eta}_{k}\alpha_{k} \\ 0 & \mathrm{i}\hat{\eta}_{k}\beta_{k} \end{bmatrix}.
\end{equation}
Thus \eqref{plemelj} can be written as
\begin{equation}
m(z|\sigma_{d})=I+\frac{\mathrm{i}}{z}\sigma_{3}Q_{-}
+\sum_{k=1}^{2N_{1}+N_{2}}\frac{1}{z-\eta_{k}}\begin{bmatrix} \alpha_{k} & 0 \\ \beta_{k} & 0 \end{bmatrix}
+\sum_{k=1}^{2N_{1}+N_{2}}\frac{1}{z-\hat{\eta}_{k}} \begin{bmatrix} 0 & \mathrm{i}\hat{\eta}_{k}\alpha_{k} \\ 0 & \mathrm{i}\hat{\eta}_{k}\beta_{k} \end{bmatrix}.
\end{equation}
Substituting this formulae into
\begin{equation}
\underset{z=\eta_j}{\rm Res}(z|\sigma_{d})=\lim_{z\rightarrow\eta_{j}}m(z|\sigma_{d})N_{j},
\end{equation}
we can derive following equations
\begin{equation}\label{alphak}
    \alpha_{j}-\mathrm{i}\sum_{k=1}^{2N_{1}+N_{2}}\frac{\hat{\eta}_{k}\gamma_{j}}{\eta_{j}-\hat{\eta}_{k}}
     =-\sigma\mathrm{i}\frac{\gamma_{j}}{\eta_{j}},
\end{equation}
\begin{equation}\label{betak}
    \beta_{j}-\mathrm{i}\sum_{k=1}^{2N_{1}+N_{2}}\frac{\hat{\eta}_{k}\gamma_{j}}{\eta_{j}-\hat{\eta}_{k}}=\gamma_{j}.
\end{equation}
Denote the coefficient matrix of \eqref{alphak} and \eqref{betak} as $A$ and $B$, respectively. It is easy to prove that $|A|\neq0$ and $|B|\neq0$, so by Cramer's Rule, both \eqref{alphak} and \eqref{betak} have an unique solution.
\end{proof}
\begin{corollary}
Denote $q_{sol}(x,t,\sigma_{d})$ the soliton solution with scattering data
\begin{equation}
\sigma_{d}=\left\{\left(\eta_{k},A[\eta_{k}]\right)\right\}_{k=1}^{2N_{1}+N_{2}}.
\end{equation}
 By reconstruction formula formulae \eqref{resconstructm}, the soliton solution is given by
\begin{equation}
q_{sol}(x,t,\sigma_{d})=-\mathrm{i}\lim_{z\rightarrow\infty}\left[zm(z|\sigma_{d})\right]_{12}
=-\sigma+\sum_{k=1}^{2N_{1}+N_{2}}\hat{\eta_{k}}\alpha_{k}.
\end{equation}
\end{corollary}
In reflectionless case, the transmission coefficient admits following trace formula
\begin{equation}
a(z)=\prod_{k=1}^{2N_{1}+N_{2}}\frac{z-\eta_{k}}{z-\hat{\eta}_{k}},
\end{equation}
whose poles can be split into two parts. Let $\diamondsuit\subseteq\left\{1,2,\cdot\cdot\cdot,2N_{1}+N_{2}\right\}$ and define
\begin{equation}\label{az}
a_{\diamondsuit}(z)=\prod_{k\in\diamondsuit}\frac{z-\eta_{k}}{z-\hat{\eta}_{k}}.
\end{equation}
We make a renormalization transformation
\begin{equation}\label{transforma}
m^{\diamondsuit}(z|\sigma_{d}^{\lozenge})=m(z|\sigma_{d})a_{\lozenge}(z)^{\sigma_{3}},
\end{equation}
which satisfies the following RH problem.
\begin{RHP}\label{rhpsigmaddiamond}
Find a $2\times 2$ matrix-valued function $m^{\lozenge}(z|\sigma_{d}^{\lozenge})$ such that
\begin{itemize}
    \item[*]$m^{\diamondsuit}(z|\sigma_{d}^{\diamondsuit})$ is analytic in $\mathbb{C}\backslash\left(\mathcal{Z}\cup\hat{\mathcal{Z}}\right)$.
    \item[*]$ m^{\diamondsuit}(z|\sigma_{d}^{\diamondsuit})=I+\mathcal{O}(z^{-1}), \quad  z\rightarrow\infty$.
    \item[*]Residue conditions: for $\eta_{k}\in \mathcal{Z}$ and $\hat{\eta}_{k}\in \hat{\mathcal{Z}}$,
           \begin{equation}
            \begin{split}
             & \underset{z=\eta_k}{\rm Res}m^{\diamondsuit}(z|\sigma_{d}^{\diamondsuit})=\lim_{z\rightarrow\eta_{k}}m^{\diamondsuit}(z|\sigma_{d}^{\diamondsuit})N^{\diamondsuit}_{k},\\
             & \underset{z=\hat{\eta}_k}{\rm Res}m^{\diamondsuit}(z|\sigma_{d}^{\diamondsuit})=-\hat{\eta}_{k}^{2}\lim_{z\rightarrow\hat{\eta}_{k}}m^{\diamondsuit}(z|\sigma_{d}^{\diamondsuit})\sigma_{2}N^{\diamondsuit}_{k}\sigma_{2},
            \end{split}
           \end{equation}
where
  \begin{equation}
   \sigma^{\diamondsuit}_{d}=\left\{\left(\eta_{k},\hat{c}_{k}\right)\right\}_{k=1}^{2N_{1}+N_{2}},
  \end{equation}
  \begin{align}
    \hat{c}_{k}=\left\{
        \begin{aligned}
    &\frac{1}{A[\eta_{k}]}[a_{\diamondsuit}^{'}(\eta_{k})]^{-2}, \quad k\in\diamondsuit,\\
    &A[\eta_{k}][a_{\diamondsuit}(\eta_{k})]^{2}, \quad k\in\mathcal{N}\backslash\diamondsuit
    \end{aligned}
        \right.
  \end{align}
is the corresponding scattering data, and
  \begin{align}
   N_{k}^{\diamondsuit}=\left\{
        \begin{aligned}
    &\begin{bmatrix} 0 & \gamma_{k}^{\diamondsuit} \\ 0 & 0 \end{bmatrix}, \quad k\in\diamondsuit\\
    &\begin{bmatrix} 0 & 0 \\ \gamma_{k}^{\diamondsuit} & 0 \end{bmatrix}, \quad k\in\mathcal{N}\backslash\diamondsuit,
    \end{aligned}
        \right.
  \end{align}
  \begin{align}
    \gamma^{\diamondsuit}_{k}=\left\{
        \begin{aligned}
    &\frac{1}{A[\eta_{k}]}[a_{\diamondsuit}^{'}(\eta_{k})]^{-2}e^{2\mathrm{i}t\theta(\eta_{k})}, \quad k\in\diamondsuit,\\
    &A[\eta_{k}][a_{\diamondsuit}(\eta_{k})]^{2}e^{-2\mathrm{i}t\theta(\eta_{k})}, \quad k\in\mathcal{N}\backslash\diamondsuit.
    \end{aligned}
        \right.
  \end{align}
\end{itemize}
\end{RHP}
\begin{proposition}\label{mtan1}
For giving scattering data $\sigma^{\diamondsuit}_{d}=\left\{\left(\eta_{k},\hat{c}_{k}\right)\right\}_{k=1}^{2N_{1}+N_{2}}$ without reflection, RHP \ref{rhpsigmaddiamond} has an unique solution and
\begin{equation}\label{transformq}
q_{sol}(x,t;\sigma_{d}^{\diamondsuit})=-\mathrm{i}\lim_{z\to\infty}\left[zm^{\diamondsuit}(z|\sigma_{d}^{\diamondsuit})\right]_{12}
=-\mathrm{i}\lim_{z\to\infty}\left[zm(z|\sigma_{d})\right]_{12}=q_{sol}(x,t;\sigma_{d}).
\end{equation}
\end{proposition}
\begin{proof}
Since the transform \eqref{transforma} is explicit, we obtain the existence and uniqueness of the solution of the RHP \ref{rhpsigmaddiamond} by Proposition \ref{unique}. And from reconstruction formula, \eqref{transformq} can be derived.
\end{proof}
We observe that $m^{(out)}(z|\sigma_{d}^{out})$ has reflection, and its reflection mainly from $T(\eta_{k})$. A nature idea is to connect $m^{(out)}(z|\sigma_{d}^{out})$ with reflectionless scattering data $\sigma^{\triangle}_{d}=\left\{\left(\eta_{k},\hat{c}_{k}\right)\right\}_{k=1}^{2N_{1}+N_{2}}$. For this purpose, we take $a_{\diamondsuit}(z)$  as
\begin{equation}
a_{\triangle}(z)=\prod_{k=1}^{2N_{1}}\prod_{l=1}^{N_{2}}
\frac{(zz_{k}-1)(z\bar{z}_{k}^{-1}+1)(\mathrm{i}\omega_{l}^{-1}z+1)}{(z+z_{k}^{-1})(z-\bar{z}_{k}^{-1})(z-\mathrm{i}\omega_{l}^{-1})}
\end{equation}
in \eqref{az}. Then $T(z)=a_{\triangle}^{-1}(z)\tilde{\delta}(z)$, where $\tilde{\delta}(z)=\delta(z){\rm exp}\left[-\mathrm{i}\int_{\Gamma}\frac{\nu(s)}{2s}ds\right]$. Therefore, the scattering data \eqref{sigmad} in RHP \ref{rhpmout} can be rewritten as
\begin{equation}
\sigma_{d}^{out}=\left\{\eta_{k},\tilde{c}_{k}\right\}_{k=1}^{2N_{1}+N_{2}},
\end{equation}
\begin{align}
 \tilde{c}_{k}=\left\{
        \begin{aligned}
    &A[\eta_{k}]a_{\triangle}^{2}(\eta_{k})\tilde{\delta}^{-2}(\eta_{k}), \quad k\in\triangle\\
    &\frac{1}{A[\eta_{k}]}[a_{\triangle}^{'}(\eta_{k})]^{-2}\tilde{\delta}^{2}(\eta_{k}), \quad k\in\nabla.
    \end{aligned}
        \right.
\end{align}
Under this new scattering data, RHP \ref{rhpmout} becomes
\begin{RHP}\label{rhpmout2}
Find a $2\times 2$ matrix-valued function $m^{(out)}(z|\sigma^{out}_{d})$ such that
\begin{itemize}
\item[*]$m^{(out)}(z|\sigma^{out}_{d})$ is analytic in $\mathbb{C}\backslash\left(\mathcal{Z}\cup\hat{\mathcal{Z}}\right)$.
\item[*]$m^{(out)}(z|\sigma^{out}_{d})=I+\mathcal{O}(z^{-1}), \quad  z\rightarrow\infty$.
\item[*]Residue conditions: for $\eta_{k}\in \mathcal{Z}$ and $\hat{\eta}_{k}\in \hat{\mathcal{Z}}$,
           \begin{equation}
            \begin{split}
             & \underset{z=\eta_k}{\rm Res}m^{(out)}(z|\sigma^{out}_{d})=\lim_{z\rightarrow\eta_{k}}m^{(out)}(z|\sigma^{out}_{d})N^{out}_{k},\\
             & \underset{z=\hat{\eta}_k}{\rm Res}m^{(out)}(z|\sigma^{out}_{d})=-\hat{\eta}_{k}^{2}\lim_{z\rightarrow\hat{\eta}_{k}}m^{(out)}(z|\sigma^{out}_{d})\sigma_{2}N^{out}_{k}\sigma_{2},
            \end{split}
           \end{equation}
where
\begin{align}
N_{k}^{out}=\left\{
        \begin{aligned}
    &\begin{bmatrix} 0 & 0 \\ A[\eta_{k}]a_{\triangle}^{2}(\eta_{k})\tilde{\delta}^{-2}(\eta_{k}) & 0 \end{bmatrix}, \quad k\in\triangle\\
    &\begin{bmatrix} 0 & \frac{1}{A[\eta_{k}]}[a_{\triangle}^{'}(\eta_{k})]^{-2}\tilde{\delta}^{2}(\eta_{k}) \\ 0 & 0 \end{bmatrix}, \quad k\in\nabla.
    \end{aligned}
        \right.
\end{align}
\end{itemize}
\end{RHP}
It is easy to verify that the solution of RHP \ref{rhpmout2} is given by
\begin{equation}\label{moutmtan}
m^{(out)}(z|\sigma^{out}_{d})=m^{\triangle}(z|\sigma_{d}^{\triangle})\tilde{\delta}(z)^{\sigma_{3}}.
\end{equation}
\begin{proposition}
RHP \ref{rhpmout2} has an unique solution. Moreover, the relation between $N$-soliton solution with reflection data and  $N$-soliton solution without reflection data as follows
\begin{equation}\label{qoutqtan}
q_{sol}(x,t;\sigma_{d}^{out})=cq_{sol}(x,t;\sigma_{d}^{\triangle}),
\end{equation}
where $c={\rm exp}\left[-\mathrm{i}\int_{\Gamma}\frac{\nu(s)}{2s}ds\right]$.
\end{proposition}
\begin{proof}
Since the transform \eqref{moutmtan} is explicit, we obtain the existence and uniqueness of the solution of the RHP \ref{rhpmout2} by Proposition \ref{mtan1}. And
\begin{equation}
\begin{split}
q_{sol}(x,t;\sigma_{d}^{out})&=-\mathrm{i}\lim_{z\to \infty}\left[zm^{(out)}(z|\sigma_{d}^{out})\right]_{12}
=-\mathrm{i}\lim_{z\to \infty}\left[zm^{\triangle}(z|\sigma_{d}^{\triangle})\tilde{\delta}(z)^{\sigma_{3}}\right]_{12}\\
&=-\mathrm{i}c\lim_{z\to \infty}\left[zm^{\triangle}(z|\sigma_{d}^{\triangle})\right]_{12}=cq_{sol}(x,t;\sigma_{d}^{\triangle}).
\end{split}
\end{equation}
\end{proof}
Next, we consider the asymptotic behavior of $m^{\triangle}(z|\sigma_{d}^{\triangle})$. Recall
\begin{equation}
\Lambda=\{k\in\mathcal{N}:|Re(2\mathrm{i}\theta(\eta_{k}))|<\delta_{0}\},
\end{equation}
and further define
\begin{equation}
\begin{split}
& \Lambda_{+}=\{k\in\mathcal{N}:0\leq Re(2\mathrm{i}\theta(\eta_{k}))<\delta_{0}\},\\
& \Lambda_{-}=\{k\in\mathcal{N}:-\delta_{0}<Re(2\mathrm{i}\theta(\eta_{k}))\leq 0\}.
\end{split}
\end{equation}
\begin{proposition}
For giving scattering data $\sigma^{\triangle}_{d}=\left\{\left(\eta_{k},\hat{c}_{k}\right)\right\}_{k=1}^{2N_{1}+N_{2}}$, we have
\begin{equation}
m^{\triangle}(z|\sigma_{d}^{\triangle})=\left(I+\mathcal{O}(e^{-\delta_{0}t})\right)m^{\Lambda}(z|\sigma_{d}^{\Lambda}),
\end{equation}
where $m^{\Lambda}(z|\sigma_{d}^{\Lambda})$ is a solution for RH problem defined by scattering data $\sigma^{\Lambda}_{d}$,
$\sigma^{\Lambda}_{d}=\{(\eta_{k},\hat{c}_{k}),k\in\Lambda\}$.
\end{proposition}
\begin{proof}
For $k\in\triangle$,
\begin{equation}
N_{k}^{\triangle}=\begin{bmatrix} 0 & \gamma_{k}^{\triangle} \\ 0 & 0 \end{bmatrix}, \quad \gamma_{k}^{\triangle}=\frac{1}{A[\eta_{k}]}[a_{\triangle}^{'}(\eta_{k})]^{-2}e^{2\mathrm{i}t\theta(\eta_{k})}.
\end{equation}
Therefore, for $k\in\triangle\cap\Lambda_{-}$, $|\gamma_{k}^{\triangle}|=\mathcal{O}(1)$ while for  $k\in\triangle\backslash\Lambda_{-}$,
\begin{equation}
|\gamma_{k}^{\triangle}|\lesssim e^{-\delta_{0}t}=\mathcal{O}(e^{-\delta_{0}t});
\end{equation}
As for $k\in\nabla$, it has same estimate as above. In a word, as $t\rightarrow\infty$,
\begin{align}
\|N_{k}^{\triangle}\|=\left\{
        \begin{aligned}
    &\mathcal{O}(1), \quad k\in\Lambda\\
    &\mathcal{O}(e^{-\delta_{0}t}), \quad k\in\mathcal{N}\backslash\Lambda,
    \end{aligned}
        \right.
\end{align}
Make discs with sufficiently small radius for each discrete spectrum $\eta_{k}, \hat{\eta}_{k}, k\in\mathcal{N}\backslash\Lambda$ so that they do not intersect each other. Define
\begin{align}\label{Phi}
\Phi=\left\{
        \begin{aligned}
    &I-\frac{1}{z-\eta_{k}}N_{k}^{\triangle}, \quad z\in D_{k},\\
    &I+\frac{\hat{\eta_{k}}^{2}}{z-\hat{\eta}_{k}}\sigma_{2}N_{k}^{\triangle}, \quad z\in \hat{D}_{k}\sigma_{2},\\
    &I, \quad elsewhere.
    \end{aligned}
        \right.
\end{align}
Make a transformation
\begin{equation}\label{transmphi}
\widehat{m}^{\triangle}(z|\sigma_{d}^{\triangle})=m^{\triangle}(z|\sigma_{d})\Phi(z).
\end{equation}
By \eqref{Phi}, we have
\begin{equation}
\widehat{m}^{\triangle}_{+}(z|\sigma_{d}^{\triangle})=\widehat{m}^{\triangle}_{-}(z|\sigma_{d}^{\triangle})\widehat{v}(z), \quad z\in\widehat{\Sigma}=\cup_{k\in\mathcal{N}\backslash\Lambda}\left(\partial D_{k}\cup\partial \hat{D}_{k}\right),
\end{equation}
where $\widehat{v}(z)=\Phi(z)$, which satisfies the following estimate
\begin{equation}
\|\widehat{v}-I\|_{L^{\infty}(\widehat{\Sigma})}=\mathcal(e^{-\delta_{0}t}).
\end{equation}
Take $\diamondsuit=\Lambda$, then $m^{\Lambda}(z|\sigma_{d}^{\Lambda})$ and $\widehat{m}^{\triangle}_{+}(z|\sigma_{d}^{\triangle})$ have the same poles and residue conditions in $\Lambda$. Therefore,
\begin{equation}\label{varepsilon}
\varepsilon(z)=\widehat{m}^{\triangle}_{+}(z|\sigma_{d}^{\triangle})[m^{\Lambda}(z|\sigma_{d}^{\Lambda})]^{-1}
\end{equation}
have no poles, and allows the following jump relation
\begin{equation}
\varepsilon_{+}(z)=\varepsilon_{-}(z)v_{\epsilon}(z),
\end{equation}
where $v_{\epsilon}=m^{\Lambda}\widehat{v}[m^{\Lambda}]^{-1}\sim\widehat{v}, z\rightarrow\infty$, and
\begin{equation}
\|v_{\epsilon}-I\|_{L^{\infty}(\widehat{\Sigma})}=\mathcal(e^{-\delta_{0}t}).
\end{equation}
According to the properties of Small norm RH problem, we know that $\varepsilon(z)$ exists and
\begin{equation}
\varepsilon(z)=I+\mathcal(-\delta_{0}t).
\end{equation}
Finally, combine \eqref{transmphi} and \eqref{varepsilon}, we obtain
\begin{equation}
m^{\triangle}(z|\sigma_{d}^{\triangle})=\left(I+\mathcal{O}(e^{-\delta_{0}t})\right)m^{\Lambda}(z|\sigma_{d}^{\Lambda}).
\end{equation}
\end{proof}
\begin{corollary}
Denote $q_{sol}(x,t;\sigma_{d}^{\triangle})$ as the corresponding $N$-soliton solution under scattering data $\sigma^{\triangle}_{d}=\left\{\left(\eta_{k},\hat{c}_{k}\right)\right\}_{k=1}^{2N_{1}+N_{2}}$, then
\begin{equation}
q_{sol}(x,t;\sigma_{d}^{out})=cq_{sol}(x,t;\sigma_{d}^{\triangle})=cq_{sol}(x,t;\sigma_{d}^{\Lambda})+\mathcal{O}(e^{-\delta_{0}t}) \quad t\rightarrow\infty.
\end{equation}
\end{corollary}

\subsubsection{A local solvable RH model near phase points}\label{in}

From the Proposition \ref{v2guji},  we find that $v^{(2)}-I$ does not have a uniformly small jump for $t\rightarrow\infty$ in the neighborhood $U_{\xi_{i}}$ of $\xi_{i}, i=1,2,3,4$ , therefore we establish a local model $m^{(lo)}$ which exactly matches the jumps of $m^{(2)}_{RHP}$ on $\Sigma^{(lo)}$ for function $E(z)$ and then it has a uniform estimate on the decay of the jump, where $\Sigma^{(lo)}$ is defined as:
\begin{equation}
\begin{split}
&\Sigma^{(lo)}=\underset{i=1,2,3,4}\cup\left( L_{i}\cap U_{\zeta_{i}}\right), \\
&\Sigma_{i}^{(lo)}=L_{i}\cap U_{\zeta_{i}}, i=1,2,3,4,
\end{split}
\end{equation}
see in Figure \ref{Sigmalo}.
\begin{RHP}\label{rhplo}
Find a $2\times 2$ matrix-valued function $m^{(lo)}(z)$ such that
\begin{itemize}
\item[*] Analyticity: $m^{(lo)}(z)$ is analytic in $\mathbb{C}\backslash\Sigma^{(lo)}$.
\item[*] Jump relation: $m_{+}^{(lo)}(z)=m_{-}^{(lo)}(z)v^{(2)}(z), \quad z\in\Sigma^{(lo)}$.
\item[*] Asymptotic behaviors: $m^{(lo)}(z)=I+\mathcal{O}(z^{-1}), \quad  z\rightarrow\infty$.
\end{itemize}
\end{RHP}
\begin{figure}[H]
\begin{center}
\begin{tikzpicture}[node distance=2cm]
\draw[->,dashed](-5,0)--(5,0);
\draw[->,dashed](0,-4)--(0,4);
\draw[dashed](0,0)circle(2cm);

\node  [right]  at (1.44,1.245) {\footnotesize $\zeta_{1}$};
\node  [left]  at (-1.44,1.245) {\footnotesize $\zeta_{2}$};
\node  [left]  at (-1.44,-1.245) {\footnotesize $\zeta_{3}$};
\node  [right]  at (1.44,-1.245) {\footnotesize $\zeta_{4}$};

\draw[fill] (1.7386,0.9886) circle [radius=0.04];
\draw[fill] (-1.7386,0.9886) circle [radius=0.04];
\draw[fill] (-1.7386,-0.9886) circle [radius=0.04];
\draw[fill] (1.7386,-0.9886) circle [radius=0.04];

\draw [color=blue](1.3603,0.6193)--(2.0818,1.3318)  node[right, scale=1] {};
\draw [color=red](1.3603,1.3318)--(2.0818,0.6193)  node[right, scale=1] {};
\draw [color=blue] [-latex] (1.7386,0.9886) -- (1.996,1.246);
\draw [color=blue] [-latex] (1.7386,0.9886) -- (1.461875,0.711625);
\draw [color=red] [-latex] (1.7386,0.9886) -- (1.996,0.711625);
\draw [color=red] [-latex] (1.7386,0.9886) -- (1.461875,1.246);

\draw [color=blue](-1.3603,0.6193)--(-2.0818,1.3318)  node[right, scale=1] {};
\draw [color=red](-1.3603,1.3318)--(-2.0818,0.6193)  node[right, scale=1] {};
\draw [color=blue] [-latex] (-2.0818,1.3318) -- (-1.9102,1.1602);
\draw [color=blue] [-latex] (-1.3693,0.6193) -- (-1.541,0.80395);
\draw [color=red] [-latex] (-1.3693,1.3318) -- (-1.5541,1.1602);
\draw [color=red] [-latex] (-2.0818,0.6193) -- (-1.9102,0.80395);

\draw [color=blue](-1.3603,-0.6193)--(-2.0818,-1.3318)  node[right, scale=1] {};
\draw [color=red](-1.3603,-1.3318)--(-2.0818,-0.6193)  node[right, scale=1] {};
\draw [color=blue] [-latex] (-2.0818,-1.3318) -- (-1.9102,-1.1602);
\draw [color=blue] [-latex] (-1.3693,-0.6193) -- (-1.541,-0.80395);
\draw [color=red] [-latex] (-1.3693,-1.3318) -- (-1.5541,-1.1602);
\draw [color=red] [-latex] (-2.0818,-0.6193) -- (-1.9102,-0.80395);

\draw [color=blue](1.3603,-0.6193)--(2.0818,-1.3318)  node[right, scale=1] {};
\draw [color=red](1.3603,-1.3318)--(2.0818,-0.6193)  node[right, scale=1] {};
\draw [color=blue] [-latex] (1.7386,-0.9886) -- (1.996,-1.246);
\draw [color=blue] [-latex] (1.7386,-0.9886) -- (1.461875,-0.711625);
\draw [color=red] [-latex] (1.7386,-0.9886) -- (1.996,-0.711625);
\draw [color=red] [-latex] (1.7386,-0.9886) -- (1.461875,-1.246);
\end {tikzpicture}
\caption{\footnotesize The jump contours $\Sigma^{(lo)}$ of $m^{(lo)}$.}
\label{Sigmalo}
\end{center}
\end{figure}
This RHP  exist jump relation but no poles and the analysis of it can be solved by the so-called Beals-Coifman operator theory.  Now we use the Beals-Coifman theory to establish the relationship between $m^{(lo)}$ and $\underset{i=1,2,3,4}\sum m_{i}^{(lo)}$,
where $m_{i}^{(lo)}$ can be constructed by parabolic cylinder equation.

On $\Sigma^{(lo)}_{ij}, i,j=1,2,3,4$, $v^{(2)}$ allows a factorization
\begin{equation}
    \left(I-w_{ij}^{-}\right)^{-1}\left(I+w_{ij}^{+}\right),
\end{equation}
\begin{equation}\label{wij}
    w_{ij}^{-}=I-[v^{(2)}]^{-1}=v^{(2)}-I, \quad w_{ij}^{+}=0,
\end{equation}
and the superscript $\pm$ indicate the analyticity in the positive/negative neighborhood of the contour.

Recall the Cauchy projection operator $C_{\pm}$ on $\Sigma^{(lo)}_{ij}$, $i,j=1,2,3,4$
\begin{equation}
    C_{\pm}f(z)=\lim_{z\leftarrow s\in\Sigma^{(lo)}_{ij,\pm}}\frac{1}{2\pi i}\int_{\Sigma^{(lo)}_{ij}}\frac{f(s)}{s-z}ds,
\end{equation}
we can define the Beals-Coifman operator on $\Sigma^{(lo)}_{ij}$, $i,j=1,2,3,4$ as follows
\begin{equation}
    C_{w_{ij}}(f):=C_{+}(fw_{ij}^{-})+C_{-}(fw_{ij}^{+}).
\end{equation}
Then we define
\begin{equation}
    w_{i}=\sum_{j=1}^{4}w_{ij},  \quad w=\sum_{i=1}^{4}w_{i}=\sum_{i,j=1}^{4}w_{ij},
\end{equation}
then we obtain $C_w=\sum_{i=1}^{4}C_{w_i}=\sum_{i,j=1}^{4}C_{w_{ij}}$.
Now we introduce the following theorem, which plays a vital role in the steepest method
\begin{theorem}\label{operatorequ}
    If $\mu\in I+L^{2}(\Sigma)$ is the solution of the singular integral equation
    \begin{equation}
        \mu=I+C_{w}(\mu),
    \end{equation}
Then there exists unique solution to the RHP for $m^{(lo)}$ written as
    \begin{equation}
     m^{(lo)}=I+C(\mu w).
    \end{equation}
\end{theorem}
Based on the above discussions, we now try to construct the Beals-Cofiman solution of $m^{(lo)}$. We start with the following lemma
\begin{lemma}\label{wjkest}
    The matrix functions $w_{ij}$ defined in \eqref{wij} admit the following estimation
    \begin{equation}
    \Vert w_{ij} \Vert_{L^{2}(\Sigma^{(lo)}_{ij})}=\mathcal{O}(t^{-1/2}).
    \end{equation}
\end{lemma}
This lemma implies that $1-C_{w}$, $1-C_{w_i}$ and  $1-C_{w_{ij}}$ exist.
Moreover, with the Theorem \ref{operatorequ}, the Beals-Cofiman solution for $m^{(lo)}$ exist unique as
\begin{equation}
    m^{(lo)}=I+\frac{1}{2\pi i}\int_{\Sigma^{(lo)}}\frac{(1-C_w)^{-1}Iw}{s-z}ds.
\end{equation}
However, the integral $I+\frac{1}{2\pi i}\int_{\Sigma^{(lo)}}\frac{(1-C_w)^{-1}Iw}{s-z}ds$ is still hard to compute. Follow
the standard procedure of Deift-Zhou \cite{NLS0}, we can separate the contributions from each saddle point. Before executing
this procedure, we need the following lemma.
\begin{lemma}\label{cwjcwk}
As $t\rightarrow +\infty$, for $i\neq j$
\begin{equation}
\Vert C_{w_i}C_{w_i} \Vert_{L^{2}(\Sigma^{(lo)})}=\mathcal{O}(t^{-1}), \quad \Vert C_{w_i}C_{w_j} \Vert_{L^{\infty}(\Sigma^{(lo)})\rightarrow L^{2}(\Sigma^{(lo)})}=\mathcal{O}(t^{-1})
\end{equation}
\end{lemma}
\begin{proof}
    Thanks to the observation of Varzugin \cite{Var}, we have
    \begin{align}
        &1-\sum_{i\neq j}C_{w_i}C_{w_j}\left(1-C_{w_j}\right)^{-1}=\left(1-C_{w}\right)\left(1+\sum_{i=1}^{4}C_{w_i}\left(1-C_{w_i}\right)^{-1}\right),\\
        &1-\sum_{i\neq j}\left(1-C_{w_i}\right)^{-1}C_{w_i}C_{w_j}=\left(1+\sum_{i=1}^{4}C_{w_i}\left(1-C_{w_i}\right)^{-1}\right)\left(1-C_{w}\right).
    \end{align}
    Take use of Lemma \ref{wjkest}, we prove this lemma.
\end{proof}
Now we can separate the contribution of Beals-Cofiman solution for $m^{(lo)}$ from each stationary phase point, which is expressed by the following proposition.
\begin{proposition} As $t\rightarrow +\infty$
\begin{equation}
    \int_{\Sigma^{(lo)}}\frac{\left(1-C_{w}\right)^{-1}Iw}{s-z}=\sum_{i=1}^{4}\int_{\Sigma^{(lo)}_{i}}\frac{\left(1-C_{w_{i}}\right)^{-1}Iw_{i}}{s-z}+O(t^{-1}).
\end{equation}
\end{proposition}
\begin{proof}
    Firstly, we can decompose the resolvent $(1-C_{w})^{-1}I$ as
\begin{equation}
    (1-C_{w})^{-1}I=I+\sum_{i=1}^{4}C_{w_j}(1-C_{w_i})^{-1}I+QPR I,
\end{equation}
where
\begin{align}
    &Q:=1+\sum_{i=1}^{4}C_{w_i}(1-C_{w_i})^{-1},\\
    &P:=\left(1-\sum_{i\neq j}C_{w_i}C_{w_k}\left(1-C_{w_j}\right)^{-1}\right)^{-1}\\
    &R:=\sum_{i\neq j}C_{w_{i}}C_{w_{j}}(1-C_{w_j})^{-1}
\end{align}
By Cauchy-Schwarz inequality
\begin{equation}
\vert\int QPRIw \vert\leqslant \Vert Q\Vert_{L^{2}(\Sigma^{(lo)}_{i})}
\Vert P\Vert_{L^{2}(\Sigma^{(lo)}_{i})}\Vert R\Vert_{L^{2}(\Sigma^{(lo)}_{i})}\Vert w \Vert_{L^{2}}\lesssim t^{-1}.
\end{equation}
The rest of the proof is trivial.
\end{proof}

Through the above conclusion, we consider to reduce above RHP \ref{rhplo} to a model RHP whose solution can be given explicitly in terms of parabolic cylinder functions on every contour $\Sigma^{(lo)}_{i}$ respectively. For briefly, we only give the details of $\Sigma_{1}^{(lo)}$ , the model of other critical point can be constructed similar.

\begin{RHP}\label{rhplo1}
Find a $2\times 2$ matrix-valued function $m_{1}^{(lo)}(z)$ such that
\begin{itemize}
\item[*] Analyticity: $m_{1}^{(lo)}(z)$ is analytic in $\mathbb{C}\backslash\Sigma_{1}^{(lo)}$.
\item[*] Jump relation: $m_{1+}^{(lo)}(z)=m_{1-}^{(lo)}(z)v_{1}^{(lo)}(z), \quad z\in\Sigma_{1}^{(lo)}$,
where
\begin{align}
      v_{1}^{(lo)}(z)=\left\{
        \begin{aligned}
    & \begin{bmatrix} 1 & -\tilde{\rho}(\zeta_{1})T^{2}(\zeta_{1})(z-\zeta_{1})^{2\mathrm{i}\nu(\zeta_{1})}e^{2\mathrm{i}t\theta(z)} \\  0 & 1 \end{bmatrix}, \quad z\in\Sigma^{(lo)}_{11},\\
    &\begin{bmatrix} 1 & 0 \\ \frac{\rho(\zeta_{1})}{1-\rho(\zeta_{1})\tilde{\rho}(\zeta_{1})}T_{-}^{-2}(\zeta_{1})(z-\zeta_{1})^{-2\mathrm{i}\nu(\zeta_{1})}e^{-2\mathrm{i}t\theta(z)}  & 1 \end{bmatrix}, \quad z\in\Sigma^{(lo)}_{12}, \\
    &\begin{bmatrix} 1 &  -\frac{\tilde{\rho}(\zeta_{1})}{1-\rho(\zeta_{1})\tilde{\rho}(\zeta_{1})}T_{+}^{-2}(\zeta_{1})(z-\zeta_{1})^{2\mathrm{i}\nu(\zeta_{1})}e^{2\mathrm{i}t\theta(z)} \\ 0  & 1 \end{bmatrix}, \quad z\in\Sigma^{(lo)}_{13}, \\
    & \begin{bmatrix} 1 & 0 \\  -\rho(\zeta_{1})T^{-2}(\zeta_{1})(z-\zeta_{1})^{-2\mathrm{i}\nu(\zeta_{1})}e^{-2\mathrm{i}t\theta(z)} & 1 \end{bmatrix}, \quad z\in\Sigma^{(lo)}_{14}.
    \end{aligned}
        \right.
    \end{align}
\item[*] Asymptotic behaviors: $m^{(lo)}_{1}(z)=I+\mathcal{O}(z^{-1}), \quad  z\rightarrow\infty$.
\end{itemize}
\end{RHP}

\begin{figure}[htbp]
	\centering
		\begin{tikzpicture}[node distance=2cm]
		\draw[->](0,0)--(2.5,1.8)node[right]{\tiny$\Sigma^{lo,1}_{4}$};
		\draw(0,0)--(-2.5,1.8)node[left]{\tiny$\Sigma^{lo,1}_{11}$};
		\draw(0,0)--(-2.5,-1.8)node[left]{\tiny$\Sigma^{lo,1}_{12}$};
		\draw[->](0,0)--(2.5,-1.8)node[right]{\tiny$\Sigma^{lo,1}_{13}$};
		\draw[dashed](-3.8,0)--(3.8,0)node[right]{\scriptsize ${\rm Re} z$};
		\draw[->](-2.5,-1.8)--(-1.25,-0.9);
		\draw[->](-2.5,1.8)--(-1.25,0.9);
		\coordinate (A) at (1,0.5);
		\coordinate (B) at (1,-0.5);
		\coordinate (G) at (-1,0.5);
		\coordinate (H) at (-1,-0.5);
		\coordinate (I) at (0,0);
		\fill (A) circle (0pt) node[right] {\tiny$\left(\begin{array}{cc}
		1 & 0\\
		-\rho(\zeta_{1})T^{-2}(\zeta_{1})(z-\zeta_{1})^{-2\mathrm{i}\nu(\zeta_{1})}e^{-2\mathrm{i}t\theta(z)} & 1
		\end{array}\right)$};
	    \fill (B) circle (0pt) node[right] {\tiny$\left(\begin{array}{cc}
		1 & -\frac{\tilde{\rho}(\zeta_{1})}{1-\rho(\zeta_{1})\tilde{\rho}(\zeta_{1})}T_{+}^{-2}(\zeta_{1})(z-\zeta_{1})^{2\mathrm{i}\nu(\zeta_{1})}e^{2\mathrm{i}t\theta(z)}\\
		0 & 1
		\end{array}\right)$};
	    \fill (G) circle (0pt) node[left] {\tiny$\left(\begin{array}{cc}
		1 & -\tilde{\rho}(\zeta_{1})T^{2}(\zeta_{1})(z-\zeta_{1})^{2\mathrm{i}\nu(\zeta_{1})}e^{2\mathrm{i}t\theta(z)}\\
		0 & 1
		\end{array}\right)$};
	    \fill (H) circle (0pt) node[left] {\tiny$\left(\begin{array}{cc}
		1 & 0\\ \frac{\rho(\zeta_{1})}{1-\rho(\zeta_{1})\tilde{\rho}(\zeta_{1})}T_{-}^{-2}(\zeta_{1})(z-\zeta_{1})^{-2\mathrm{i}\nu(\zeta_{1})}e^{-2\mathrm{i}t\theta(z)} & 1
		\end{array}\right)$};
		\fill (I) circle (1pt) node[below] {$\zeta_1$};
		\end{tikzpicture}
	\caption{\footnotesize The contour $\Sigma^{(lo)}_{1}$ and the jump matrix on it.}
\end{figure}
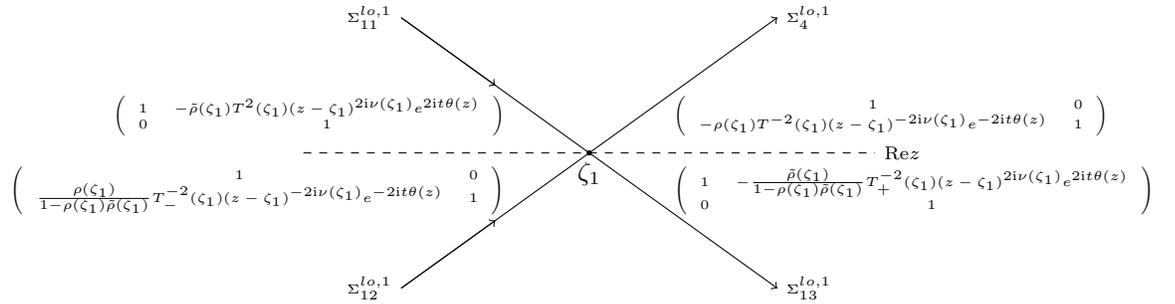

For $z$ near $\zeta_{1}$, we have
\begin{equation}
\theta(z)=\theta(\zeta_{1})+\theta^{''}(\zeta_{1})(z-\zeta_{1})^{2}+o((z-\zeta_{1})^{2}),
\end{equation}
where $\theta^{''}(\zeta_{1})$ is complex value and non-zero. To match RHP \ref{rhplo1} with a parabolic cylinder model, we introduce the following scaling transformation
\begin{equation}\label{scalingtrans}
u_{1}=2t^{\frac{1}{2}}\sqrt{\theta^{''}(\zeta_{1})}(z-\zeta_{1}).
\end{equation}
Under this scaling transformation, we define
\begin{equation}
\tilde{\rho}_{\zeta_{1}}=\tilde{\rho}(\zeta_{1})T^{2}(\zeta_{1})e^{2\mathrm{i}t\theta(\zeta_{1})-\mathrm{i}\nu(\zeta_{1}){\rm log}(4t\theta^{''}(\zeta_{1}))}.
\end{equation}
Similarly, we can define $\rho_{\zeta_{1}}$, $\frac{\rho_{\zeta_{1}}}{1-\rho_{\zeta_{1}}\tilde{\rho}_{\zeta_{1}}}$ and $\frac{\tilde{\rho}_{\zeta_{1}}}{1-\rho_{\zeta_{1}}\tilde{\rho}_{\zeta_{1}}}$.
Moreover,  the jump matrix $v_{1}^{(lo)}(z)$ approximates to the jump of a parabolic cylinder model problem as follow:
\begin{RHP}\label{rhppc1}
Find a $2\times 2$ matrix-valued function $m_{1}^{(pc)}(u_{1},\xi)$ such that
\begin{itemize}
\item[*] Analyticity: $m_{1}^{(pc)}(u_{1},\xi)$ is analytic in $\mathbb{C}\backslash\Sigma_{1}^{(pc)}$.
\item[*] Jump relation: $m_{1+}^{(pc)}(u_{1})=m_{1-}^{(pc)}(u_{1})v_{1}^{(pc)}(u_{1}), \quad u_{1}\in\Sigma_{1}^{(pc)}$,
where
\begin{align}
      v_{1}^{(pc)}(u_{1})=\left\{
        \begin{aligned}
    &\begin{bmatrix} 1 & -\tilde{\rho}_{\zeta_{1}}u_{1}^{-2\mathrm{i}\nu(\zeta_{1})}e^{\frac{\mathrm{i}}{2}u_{1}^{2}} \\ 0 & 1 \end{bmatrix}, \quad u_{1}\in\Sigma^{(pc)}_{11}, \\
    &\begin{bmatrix} 1 &  0 \\ \frac{\rho_{\zeta_{1}}}{1-\rho_{\zeta_{1}}\tilde{\rho}_{\zeta_{1}}}u_{1}^{2\mathrm{i}\nu(\zeta_{1})}e^{-\frac{\mathrm{i}}{2}u_{1}^{2}}  & 1 \end{bmatrix}, \quad u_{1}\in\Sigma^{(pc)}_{12}, \\
    & \begin{bmatrix} 1 & -\frac{\tilde{\rho}_{\zeta_{1}}}{1-\rho_{\zeta_{1}}\tilde{\rho}_{\zeta_{1}}}u_{1}^{-2\mathrm{i}\nu(\zeta_{1})}e^{\frac{\mathrm{i}}{2}u_{1}^{2}} \\  0 & 1 \end{bmatrix}, \quad u_{1}\in\Sigma^{(pc)}_{13},\\
    & \begin{bmatrix} 1 & 0 \\  -\rho_{\zeta_{1}}u_{1}^{2\mathrm{i}\nu(\zeta_{1})}e^{-\frac{\mathrm{i}}{2}u_{1}^{2}} & 1 \end{bmatrix}, \quad u_{1}\in\Sigma^{(pc)}_{14}.
    \end{aligned}
        \right.
    \end{align}
\item[*] Asymptotic behaviors: $m_{1}^{(pc)}(u_{1})=I+m_{1,1}^{(pc)}(\zeta_{1})u_{1}^{-1}+\mathcal{O}(u_{1}^{-2}), \quad  z\rightarrow\infty$.
\end{itemize}
\end{RHP}

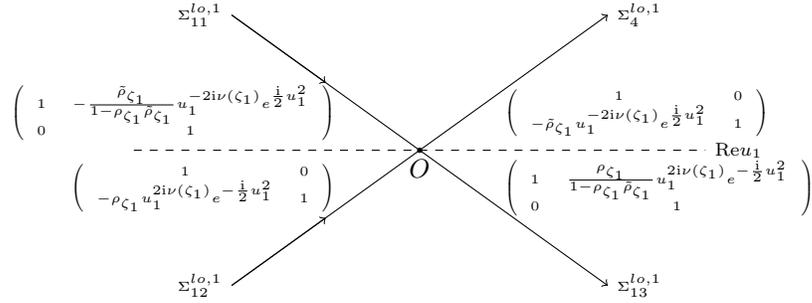
\begin{figure}[htbp]
	\centering
		\begin{tikzpicture}[node distance=2cm]
		\draw[->](0,0)--(2.5,1.8)node[right]{\tiny$\Sigma^{lo,1}_{4}$};
		\draw(0,0)--(-2.5,1.8)node[left]{\tiny$\Sigma^{lo,1}_{11}$};
		\draw(0,0)--(-2.5,-1.8)node[left]{\tiny$\Sigma^{lo,1}_{12}$};
		\draw[->](0,0)--(2.5,-1.8)node[right]{\tiny$\Sigma^{lo,1}_{13}$};
		\draw[dashed](-3.8,0)--(3.8,0)node[right]{\scriptsize ${\rm Re} u_{1}$};
		\draw[->](-2.5,-1.8)--(-1.25,-0.9);
		\draw[->](-2.5,1.8)--(-1.25,0.9);
		\coordinate (A) at (1,0.5);
		\coordinate (B) at (1,-0.5);
		\coordinate (G) at (-1,0.5);
		\coordinate (H) at (-1,-0.5);
		\coordinate (I) at (0,0);
		\fill (A) circle (0pt) node[right] {\tiny$\left(\begin{array}{cc}
		1 & 0\\
		 -\tilde{\rho}_{\zeta_{1}}u_{1}^{-2\mathrm{i}\nu(\zeta_{1})}e^{\frac{\mathrm{i}}{2}u_{1}^{2}} & 1
		\end{array}\right)$};
	    \fill (B) circle (0pt) node[right] {\tiny$\left(\begin{array}{cc}
		1 & \frac{\rho_{\zeta_{1}}}{1-\rho_{\zeta_{1}}\tilde{\rho}_{\zeta_{1}}}u_{1}^{2\mathrm{i}\nu(\zeta_{1})}e^{-\frac{\mathrm{i}}{2}u_{1}^{2}} \\
		0 & 1
		\end{array}\right)$};
	    \fill (G) circle (0pt) node[left] {\tiny$\left(\begin{array}{cc}
		1 &  -\frac{\tilde{\rho}_{\zeta_{1}}}{1-\rho_{\zeta_{1}}\tilde{\rho}_{\zeta_{1}}}u_{1}^{-2\mathrm{i}\nu(\zeta_{1})}e^{\frac{\mathrm{i}}{2}u_{1}^{2}} \\
		0 & 1
		\end{array}\right)$};
	    \fill (H) circle (0pt) node[left] {\tiny$\left(\begin{array}{cc}
		1 & 0\\ -\rho_{\zeta_{1}}u_{1}^{2\mathrm{i}\nu(\zeta_{1})}e^{-\frac{\mathrm{i}}{2}u_{1}^{2}} & 1
		\end{array}\right)$};
		\fill (I) circle (1pt) node[below] {$O$};
		\end{tikzpicture}
	\caption{\footnotesize The contour $\Sigma^{(pc)}_{1}$ and the jump matrix on it.}
\end{figure}

The solution of this RHP can be referred in as follows
\begin{equation}
m_{1}^{(pc)}(u_{1})=I+m_{1,1}^{(pc)}(\zeta_{1})u_{1}^{-1}+\mathcal{O}(u_{1}^{-2}),
\end{equation}
where
\begin{equation}
m_{1,1}^{(pc)}(\zeta_{1})=\begin{bmatrix} 0 & \beta^{\zeta_{1}}_{12} \\ -\beta^{\zeta_{1}}_{21} & 0 \end{bmatrix},
\end{equation}
and
\begin{align}
\left\{
        \begin{aligned}
    & \beta^{\zeta_{1}}_{12}=-\frac{\sqrt{2\pi}e^{\frac{\pi}{4}\mathrm{i}}e^{-\frac{\pi}{2}}\nu}{\rho_{\zeta_{1}}\Gamma(a)},\\
    &\beta^{\zeta_{1}}_{21}=-\frac{\sqrt{2\pi}e^{-\frac{\pi}{4}\mathrm{i}}e^{-\frac{\pi}{2}}\nu(1-\rho_{\zeta_{1}}\tilde{\rho}_{\zeta_{1}})}{\rho_{\zeta_{1}}\Gamma(a)}.
    \end{aligned}
        \right.
\end{align}
Substitute \eqref{scalingtrans} into above  consequence, we get
\begin{equation}
m_{1}^{(lo)}(z)=I+\frac{1}{2}\frac{t^{-\frac{1}{2}}}{\sqrt{\theta^{''}(\zeta_{1})}(z-\zeta_{1})}m_{1,1}^{(pc)}(\zeta_{1})+\mathcal{O}(t^{-1}).
\end{equation}
Moreover, the local RHP  around the other stationary phase points can be solved by the same way:
\begin{equation}
m_{i}^{(lo)}(z)=I+\frac{1}{2}\frac{t^{-\frac{1}{2}}}{\sqrt{\theta^{''}(\zeta_{i})}(z-\zeta_{i})}m_{i,1}^{(pc)}(\zeta_{i})+\mathcal{O}(t^{-1}),
\end{equation}
where
\begin{equation}
m_{i,1}^{(pc)}(\zeta_{i})=\begin{bmatrix} 0 & \beta^{\zeta_{i}}_{12} \\ -\beta^{\zeta_{i}}_{21} & 0 \end{bmatrix},
\end{equation}
and
\begin{align}
\left\{
        \begin{aligned}
    & \beta^{\zeta_{i}}_{12}=-\frac{\sqrt{2\pi}e^{\frac{\pi}{4}\mathrm{i}}e^{-\frac{\pi}{2}}\nu}{\rho_{\zeta_{i}}\Gamma(a)},\\
    &\beta^{\zeta_{i}}_{21}=-\frac{\sqrt{2\pi}e^{-\frac{\pi}{4}\mathrm{i}}e^{-\frac{\pi}{2}}\nu(1-\rho_{\zeta_{i}}\tilde{\rho}_{\zeta_{i}})}{\rho_{\zeta_{i}}\Gamma(a)},
    \end{aligned}
        \right.
\end{align}
$i=1,2,3,4$.

Thus the solution $m^{(lo)}$ can be expressed as the following proposition as $t\rightarrow\infty$.
\begin{proposition}
As $t\rightarrow\infty$,
\begin{equation}
m^{(lo)}(z)=I+\frac{1}{2}t^{-\frac{1}{2}}\sum_{i=1}^{4}\frac{m_{i,1}^{(pc)}(\zeta_{i})}{\sqrt{\theta^{''}(\zeta_{i})}(z-\zeta_{i})}+\mathcal{O}(t^{-1}).
\end{equation}
\end{proposition}
\subsubsection{The small norm RH problem for error function}\label{erro}

In this section, we consider the error matrix-function $E(z)$. Define
\begin{equation}
U(\xi)=\underset{i=1,2,3,4}\cup U_{\zeta_{i}},
\end{equation}
the $E(z)$ allows the following RHP
\begin{RHP}\label{rhpE}
    Find a $2\times 2$ matrix-valued function $E(z)$ such that
\begin{itemize}
    \item[*] $E(z)$ is analytical in $\mathbb{C}\backslash \Sigma^{(E)}$, where
    \begin{equation}
        \Sigma^{(E)}=\partial U(\xi) \cup \left(\Sigma^{(2)}\backslash U(\xi)\right);
    \end{equation}
    \item[*] $E(z)$ takes continuous boundary values $E_{\pm}(z)$ on $\Sigma^{(E)}$ and
    \begin{equation}
        E_{+}(z)=E_{-}(z)v^{(E)}(z),
    \end{equation}
    where
    \begin{equation}
      v^{(E)}(z)=\left\{
        \begin{aligned}
        &m^{(out)}(z)v^{(2)}(z)\left[{m^{(out)}(z)}\right]^{-1}, \quad z\in \Sigma^{(2)}\backslash U(\xi),\\
        &m^{(out)}(z)m^{(lo)}(z)\left[{m^{(out)}(z)}\right]^{-1}. \quad z\in \partial U(\xi).
        \end{aligned}
        \right.
    \end{equation}
    \item[*] asymptotic behavior
    \begin{equation}
     E(z)=I+\mathcal{O}(z^{-1}), \quad z\rightarrow \infty.
    \end{equation}
\end{itemize}
\end{RHP}

\begin{figure}[H]
\begin{center}
\begin{tikzpicture}[node distance=2cm]
\draw[->,dashed](-5,0)--(5,0)node[right]{ \textcolor{black}{${\rm Re} z$}};
\draw[->](0,-4)--(0,4)node[right]{\textcolor{black}{${\rm Im} z$}};
\draw[dashed](0,0)circle(2cm);

\draw(1.7386,0.9886)circle(0.4cm);
\draw(-1.7386,0.9886)circle(0.4cm);
\draw(-1.7386,-0.9886)circle(0.4cm);
\draw(1.7386,-0.9886)circle(0.4cm);

\draw (-1,-0.25)--(-2.425,-1.675)  node[right, scale=1] {};
\draw (1,-0.25)--(2.425,-1.675)  node[right, scale=1] {};
\draw (1,0.25)--(2.425,1.675)  node[right, scale=1] {};
\draw (-1,0.25)--(-2.425,1.675)  node[right, scale=1] {};
\draw (0,2)--(3.006,0.251)  node[right, scale=1] {};
\draw (0,2)--(-3.006,0.251)  node[right, scale=1] {};
\draw (0,-2)--(3.006,-0.251)  node[right, scale=1] {};
\draw (0,-2)--(-3.006,-0.251)  node[right, scale=1] {};
\draw (1,2.266) arc (90:45:2.016);
\draw (1,2.266) arc (90:120:2.016);
\draw (-1,2.266) arc (90:60:2.016);
\draw (-1,2.266) arc (90:135:2.016);
\draw (-1,-2.266) arc (270:300:2.016);
\draw (-1,-2.266) arc (270:225:2.016);
\draw (1,-2.266) arc (270:315:2.016);
\draw (1,-2.266) arc (270:240:2.016);

\draw  [-latex] (2.425,1.675) to  [out=140, in=345]  (1.45,2.215);
\draw  [-latex] (2.0818,1.3318) -- (2.2534,1.5034);
\draw  [-latex] (1.3693,0.6193) -- (1.18465,0.43465);
\draw  [-latex] (1.30395,1.24145)-- (0.8693,1.4943);
\draw  [-latex] (1.7386,0.9886) -- (2.3723,0.6198);

\draw  [-latex]  (0,2) to  [out=150, in=10]  (-1.45,2.215);
\draw  [-latex]  (-2.425,1.675)-- (-2.0818,1.3318);
\draw  [-latex] (-1,0.25) -- (-1.3693,0.6193);
\draw  [-latex] (0,2) -- (-0.8693,1.4943);;
\draw  [-latex] (-3.006,0.251) -- (-2.3723,0.6198);

\draw  [-latex]  (0,-2) to  [out=210, in=350]  (-1.45,-2.215);
\draw  [-latex]  (-2.425,-1.675)-- (-2.0818,-1.3318);
\draw  [-latex] (-1,-0.25) -- (-1.3693,-0.6193);
\draw  [-latex]  (0,-2) -- (-0.8693,-1.4943);
\draw [-latex] (-3.006,-0.251) -- (-2.3723,-0.6198);

\draw  [-latex] (2.425,-1.675) to [out=220, in=15]  (1.45,-2.215);
\draw  [-latex] (2.0818,-1.3318) -- (2.2534,-1.5034);
\draw  [-latex] (1.3693,-0.6193) -- (1.18465,-0.43465);
\draw  [-latex] (1.30395,-1.24145) -- (0.8693,-1.4943);
\draw  [-latex] (1.7386,-0.9886) -- (2.3723,-0.6198);

\draw [fill=white] (1.7386,0.9886) circle [radius=0.4];
\draw [fill=white] (-1.7386,0.9886) circle [radius=0.4];
\draw [fill=white] (-1.7386,-0.9886) circle [radius=0.4];
\draw [fill=white] (1.7386,-0.9886) circle [radius=0.4];

\node  [right]  at (2.1,1) {\footnotesize $\zeta_{1}$};
\node  [left]  at (-2.1,1) {\footnotesize $\zeta_{2}$};
\node  [left]  at (-2.1,-1) {\footnotesize $\zeta_{3}$};
\node  [right]  at (2.1,-1) {\footnotesize $\zeta_{4}$};

\draw[fill] (1.7386,0.9886) circle [radius=0.04];
\draw[fill] (-1.7386,0.9886) circle [radius=0.04];
\draw[fill] (-1.7386,-0.9886) circle [radius=0.04];
\draw[fill] (1.7386,-0.9886) circle [radius=0.04];
\end {tikzpicture}
\caption{\footnotesize The jump contours $\Sigma^{(E)}$ of $m^{(E)}$.}
\label{SigamE}
\end{center}
\end{figure}
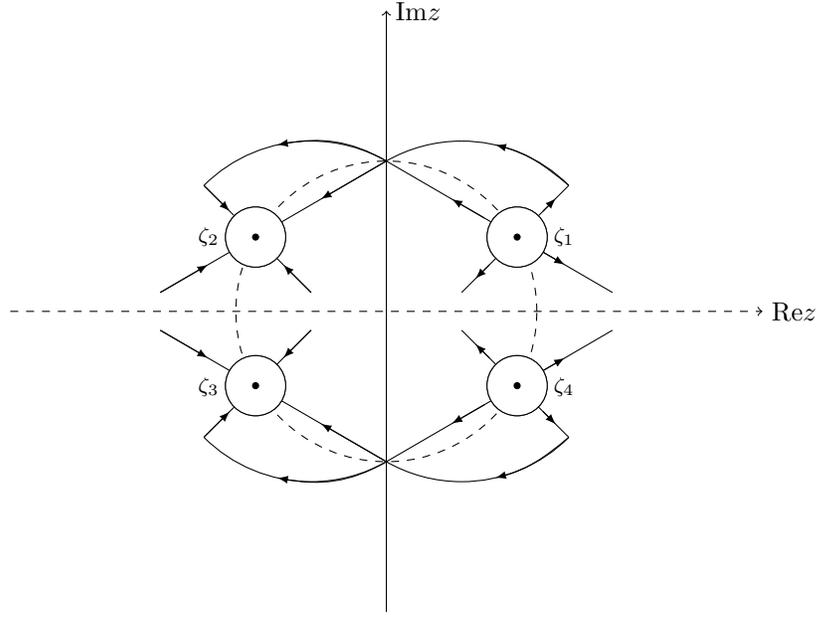

By Proposition \ref{v2guji}, we have the following estimates of $v^{(E)}$:
\begin{align}
\left|v^{(E)}-I\right|=\left\{
        \begin{aligned}
    &\mathcal{O}(e^{-\delta_{0}t}), \quad z\in\Sigma_{2}\backslash U(\xi),\\
    &\mathcal{O}(t^{-\frac{1}{2}}), \quad z\in \partial U(\xi).
    \end{aligned}
        \right.
\end{align}
According to Beals-Cofiman theory, we consider the trivial decomposition of $v^{(E)}$
\begin{equation}
v^{(E)}=b_{-}^{-1}b_{+}, \quad b_{-}=I, \quad b_{+}=v^{(E)},
\end{equation}
thus,
\begin{equation}
\begin{split}
& (w_{E})_{-}=I-b_{-}=0, \quad (w_{E})_{+}=b_{+}-I=v^{(E)}-I,\\
& w_{E}=(w_{E})_{-}+(w_{E})_{+}=v^{(E)}-I,\\
& C_{w_{E}}f=C_{-}\left(f(w_{E})_{+}\right)+C_{+}\left(f(w_{E})_{-}\right)=C_{-}\left(f(v^{(E)}-I)\right),
\end{split}
\end{equation}
where $C_{-}$ is the Cauchy projection operator
\begin{equation}
C_{-}f(z)=\lim_{z^{'}\rightarrow z\in\Sigma^{(E)}}\frac{1}{2\pi\mathrm{i}}\int_{\Sigma^{(E)}}\frac{f(s)}{s-z^{'}}ds,
\end{equation}
and $\|C_{-}\|_{L^{2}(\Sigma^{(E)})}$ is bounded. As a result, $E(z)$ in RHP \ref{rhpE} can be given by
\begin{equation}\label{E}
E(z)=I+\frac{1}{2\pi\mathrm{i}}\int_{\Sigma^{(E)}}\frac{\mu_{E}(v^{(E)}(s)-I)}{s-z^{'}}ds,
\end{equation}
where $\mu_{E}\in L^{2}(\Sigma^{(E)})$, and satisfies
\begin{equation}
(I-C_{w_{E}})\mu_{E}=I.
\end{equation}
\begin{proposition}
RHP \ref{rhpE} has an unique solution $E(z)$.
\end{proposition}
\begin{proof}
We can obtain the following estimate by the definition of $C_{w}$:
\begin{equation}
\|C_{w_{E}}\|_{L^{2}(\Sigma^{(E)})}\leq\|C_{-}\|_{L^{2}(\Sigma^{(E)})}\|v^{(E)}-I\|_{l^{\infty}(\Sigma^{(E)})}\lesssim\mathcal{O}(t^{-1/2}),
\end{equation}
which implies $I-C_{w_{E}}$ is invertible for sufficiently large $t$. Furthermore, the existence and uniqueness for $\mu$ and $E(z)$ cab be proved.
\end{proof}
In order to reconstruct the solution $q(x,t)$ of \eqref{nmkdv}, we need the asymptotic behavior of $E(z)$ as $z\rightarrow\infty$.
\begin{proposition}
As $z\rightarrow\infty$, we have
\begin{equation}
E(z)=I+\frac{E_{1}}{z}+\mathcal{O}(z^{-2}),
\end{equation}
where
\begin{equation}
E_{1}=\sum_{i=1}^{4}\frac{t^{-1/2}}{2\mathrm{i}\sqrt{\theta^{''}(\zeta_{i})}}m^{(out)}(\zeta_{i})m^{(pc)}_{i1}[m^{(out)}(\zeta_{i})]^{-1}+\mathcal{O}(t^{-1}).
\end{equation}
\end{proposition}
\begin{proof}
Recall \eqref{E}, we have
\begin{equation}
E_{1}=-\frac{1}{2\pi\mathrm{i}}\int_{\Sigma^{(E)}}\mu_{E}(s)(v^{(E)}(s)-I)ds=I_{1}+I_{2}+I_{3},
\end{equation}
where
\begin{align}
        &I_{1}=-\frac{1}{2\pi i}\oint_{\partial U(\xi)}\mu_{E}(s)\left(v^{(E)}(s)-I\right)ds,\\
        &I_{2}=-\frac{1}{2\pi i}\int_{\Sigma^{(E)}\backslash U(\xi)}\mu_{E}(s)\left(v^{(E)}(s)-I\right)ds,\\
        &I_{3}=-\frac{1}{2\pi i}\int_{\Sigma^{(E)}}(\mu_{E}(s)-I)({v^{(E)}}(s)-I)ds.
\end{align}
Take use of Proposition \ref{v2guji}, we can obtain $\vert I_2 \vert=\mathcal{O}(t^{-1})$.
As for $I_3$, we have
\begin{equation}
  |I_3|\leq c\|\mu_{E}(s)-I\|_{L^{2}(\Sigma^{(E)})}\|v^{(E)}-I\|_{L^{2}(\Sigma^{(E)})}=\mathcal{O}(t^{-1}).
\end{equation}
Finally,
\begin{align}
    I_{1}&=-\frac{1}{2\pi i}\sum_{i=1}^{4}\oint_{\partial U_{\zeta_i}}m^{(out)}(s)\left(m_{i}^{(lo)}(s)-I\right){m^{(out)}(s)}^{-1}ds \nonumber\\
    &=-\frac{1}{2\pi i}\sum_{i=1}^{4}\oint_{\partial U_{\zeta_i}}\frac{t^{-1/2}}{2\mathrm{i}\sqrt{\theta^{''}(\zeta_{i})}(z-\zeta_{i})}m^{(out)}(s)m_{i,1}^{(pc)}{m^{(out)}(s)}^{-1}ds \nonumber\\
    &=\sum_{i=1}^{4}\frac{t^{-1/2}}{2\mathrm{i}\sqrt{\theta^{''}(\zeta_{i})}}m^{(out)}(\zeta_{i})m_{i,1}^{(pc)}{m^{(out)}(\zeta_{i})}^{-1}.
\end{align}
\end{proof}

\subsection{Analysis on the pure $\bar{\partial}$-Problem}\label{puredbar}
Now we define the function
\begin{equation}\label{m3}
    m^{(3)}(z)=m^{(2)}(z)(m^{(2)}_{RHP}(z))^{-1}.
\end{equation}
Then $m^{(3)}$ satisfies the following $\bar{\partial}$-Problem.
\begin{Dbar}\label{dbarproblem}
    Find a $2\times 2$ matrix-valued function $m^{(3)}(z)$ such that
    \begin{itemize}
        \item[*] $m^{(3)}(z)$ is continuous in $\mathbb{C}$ and analytic in $\mathbb{C}\backslash\left(\underset{i,j=1,2,3,4}\cup\Omega_{ij}\right)$;
        \item[*] asymptotic behavior
        \begin{equation}
              m^{(3)}(z)=I+\mathcal{O}(z^{-1}), \quad z\rightarrow\infty ;
        \end{equation}
        \item[*]For $z\in \mathbb{C}$, we have
        \begin{equation}
            \bar{\partial}m^{(3)}(z)=m^{(3)}(z)W^{(3)}(z);
        \end{equation}
        where $W^{(3)}=m^{(2)}_{RHP}(z)\bar{\partial}R^{(2)}(z)(m^{(2)}_{RHP}(z))^{-1}$,
        \begin{align}
          W^{(3)}(z)=\left\{
        \begin{aligned}
        &m^{(2)}_{RHP}(z)\begin{bmatrix} 0 & -\bar{\partial}R_{i1}e^{2\mathrm{i}\theta} \\ 0 & 0 \end{bmatrix}(m^{(2)}_{RHP}(z))^{-1}, \quad z\in\Omega_{i1}, i=1,2,3,4,\\
        &m^{(2)}_{RHP}(z)\begin{bmatrix} 0 & 0 \\ \bar{\partial}R_{i2}e^{-2\mathrm{i}\theta} & 0 \end{bmatrix}(m^{(2)}_{RHP}(z))^{-1}, \quad z\in\Omega_{i2},  i=1,2,3,4,\\
         &m^{(2)}_{RHP}(z)\begin{bmatrix} 0 & \bar{\partial}R_{i3}e^{2\mathrm{i}\theta} \\ 0 & 0 \end{bmatrix}(m^{(2)}_{RHP}(z))^{-1}, \quad z\in\Omega_{i3}, i=1,2,3,4,\\
         &m^{(2)}_{RHP}(z)\begin{bmatrix} 0 & 0 \\ \bar{\partial}R_{i4}e^{-2\mathrm{i}\theta} & 0 \end{bmatrix}(m^{(2)}_{RHP}(z))^{-1}, \quad z\in\Omega_{i4},  i=1,2,3,4,\\
         & \begin{bmatrix} 0 & 0 \\ 0 & 0 \end{bmatrix}, \quad elsewhere.
        \end{aligned}
        \right.
\end{align}
    \end{itemize}
\end{Dbar}
\begin{proof}
The normalization condition and $\bar{\partial}$-derivative of $m^{(3)}(z)$ follow immediately from the properties of $m^{(2)}(z)$ and $m^{(2)}_{RHP}(z)$.
Then we prove the following claims.
\begin{itemize}
    \item[Claim 1:] $m^{(3)}$ has no jumps;\\
    Since $m^{(2)}$ and $m^{(2)}_{RHP}$ take the same jump matrix, we have
    \begin{equation}
     \begin{split}
      [m^{(3)}_{-}]^{-1}m^{(3)}_{+}&=m^{(2)}_{RHP-}[m^{(2)}_{-}]^{-1}m^{(2)}_{+}\left[m^{(2)}_{RHP+}\right]^{-1}\\
      &=m^{(2)}_{RHP-}v^{(2)}\left[m^{(2)}_{RHP+}\right]^{-1}=I,
     \end{split}
    \end{equation}
    \item[Claim 2:] $m^{(3)}$ has no singularity at $z=0$.\\
    Near $z=0$, we have
    \begin{equation}
        \left[m^{(2)}_{RHP}\right]^{-1}=\left(I+\frac{i}{z}\sigma_3Q_{-}\right)^{-1}=\frac{\sigma_2\left[m^{(2)}_{RHP}\right]^{\textrm{T}}\sigma_2}{1+z^{-2}}.
    \end{equation}
    Thus
    \begin{equation}
        \lim_{z\rightarrow 0}m^{(3)}=\lim_{z\rightarrow 0}\frac{m^{(2)}\sigma_2\left[m^{(2)}_{RHP}\right]^{\textrm{T}}\sigma_2}{1-z^{-2}}=(\mathrm{i}\sigma_3Q_{-}\sigma_2)^2=I=\mathcal{O}(1).
    \end{equation}
    \item[Claim 3:] $m^{(3)}$ has no singularities at $z=\pm \mathrm{i}$.\\
    This follows from observing that the symmetries of RH problem applied to the local expansion of $m^{(2)}(z)$ and $m^{(2)}_{RHP}(z)$ imply that
    \begin{equation}
    \begin{split}
       & m^{(2)}(z)=\begin{bmatrix} c_{1\pm} & \pm\sigma c_{1\pm} \\ c_{2\pm} & \pm\sigma c_{2\pm} \end{bmatrix}+\mathcal{O}(z\mp\mathrm{i}), \\
        &\left[m^{(2)}_{RHP}(z)\right]^{-1}=\frac{\pm 1}{2\mathrm{i}(z\mp \mathrm{i})}\sigma_2\begin{bmatrix} \gamma_{1\pm} & \pm\sigma \gamma_{1\pm} \\ \gamma_{2\pm} & \pm\sigma \gamma_{2\pm} \end{bmatrix}^{\textrm{T}}\sigma_2+\mathcal{O}(1),
    \end{split}
    \end{equation}
    where $c_{j\pm}=m^{(2)}_{j1}(\pm \mathrm{i}), \gamma_{j\pm}=m^{(2)}_{RHPj1}(\pm \mathrm{i}), j=1,2$. Taking the product it's immediately clear the singular part of $m^{(3)}(z)$ vanishes at $z=\pm \mathrm{i}$.
    \item[Claim 4:]  $m^{(3)}$ has no singularities at $\mathcal{Z}\cup\hat{\mathcal{Z}}$.\\
    For $\eta\in\mathcal{Z}\cup\hat{\mathcal{Z}}$, let $\mathcal{N}_\eta$ denote the nilpotent matrix which appears in the residue condition of $m^{(2)}$ and $m^{(2)}_{RHP}$,
    we have Laurent expansions at $z=\eta$
    \begin{align}
        m^{(2)}(z)=a(\eta)\left(\frac{\mathcal{N}_\eta}{z-\eta}+I\right)+\mathcal{O}\left(\left(z-\eta\right)\right),\\
        m^{(2)}_{RHP}(z)=A(\eta)\left(\frac{\mathcal{N}_\eta}{z-\eta}+I\right)+\mathcal{O}\left(\left(z-\eta\right)\right),
    \end{align}
    then, we have
    \begin{equation}
        m^{(3)}=m^{(2)}\left[m^{(2)}_{RHP}\right]^{-1}=\left[a(\eta)\left(\frac{\mathcal{N}_\eta}{z-\eta}+I\right)\right]
        \left[\left(-\frac{\mathcal{N}_\eta}{z-\eta}+I\right)\sigma_2A^{\rm T}(\eta)\sigma_2\right]=\mathcal{O}(1).
    \end{equation}
\end{itemize}
\end{proof}
Now we consider the long time asymptotic behavior of $m^{(3)}$. The solution of $\bar{\partial}$-Problem \ref{dbarproblem} can be solved
by the following integral equation
\begin{equation}\label{intm3}
    m^{(3)}(z)=I-\frac{1}{\pi}\iint_{\mathbb{C}}\frac{m^{(3)}(s)W^{(3)}(s)}{s-z}dA(s),
\end{equation}
where $A(s)$ is the Lebesgue measure on $\mathbb{C}$. Denote $S$ as the Cauchy-Green integral operator
\begin{equation}\label{C-Gop}
    S[f](z)=-\frac{1}{\pi}\iint\frac{f(s)W^{(3)}(s)}{s-z}dA(s),
\end{equation}
then \eqref{intm3} can be written as the following operator equation
\begin{equation}
    (1-S)m^{(3)}(z)=I.
\end{equation}
To prove the existence of the operator at large time, we present the following lemma.
\begin{lemma}\label{estS}
The norm of the integral operator $S$ decay to zero as $t\rightarrow\infty$, and
\begin{equation}
\|S\|_{L^{\infty}\rightarrow L^{\infty}}\leq ct^{-1/4}.
\end{equation}
\end{lemma}
\begin{proof}
Taking $z\in\Omega_{11}$ for example, the other cases are similar. Denote
\begin{equation}
s=u+\mathrm{i}v-\zeta_{1}=(u-a_{1})+\mathrm{i}(v-b_{1}), \quad z=\alpha+\mathrm{i}\eta,
\end{equation}
then$\left|e^{(2\mathrm{i}t\theta)}\right|=e^{{\rm Re}(2\mathrm{i}t\theta)}\leq e^{-t(v-b_{1})}$ and
\begin{equation}
\begin{split}
\left\|\frac{1}{s-z}\right\|^{2}_{L^{2}(a_{1}+b_{1}-v,2a_{1})}&=\int_{a_{1}+b_{1}-v}^{2a_{1}}\frac{1}{|s-z|^{2}}du \leq\int_{0}^{+\infty}\frac{1}{|s-z|^{2}}\\
&=\int_{0}^{+\infty}\frac{1}{(u-a_{1}-\alpha)^{2}+(v-b-\eta)^{2}}du\\
&=\frac{1}{|v-b_{1}-\eta|}\int_{0}^{+\infty}\frac{1}{1+y^{2}}dy\\
&=\frac{\pi}{2|v-b_{1}-\eta|},
\end{split}
\end{equation}
where $ y=\frac{u-a_{1}-\alpha}{v-b_{1}-\eta}$. Therefore, we obtain
\begin{equation}
\begin{split}
\left|S(f)\right|&\leq\frac{1}{\pi}\iint_{\Omega_{11}}\frac{\left|fm^{(2)}_{RHP}\bar{\partial}\mathcal{R}^{(2)}[m^{(2)}_{RHP}]^{-1}\right|}{|s-z|}dA(s)\\
&\lesssim\iint_{\Omega_{11}}\frac{\bar{\partial}\chi_{\mathcal{Z}}(z)e^{2\mathrm{i}t\theta}}{|s-z|}dA(s) \\
&\lesssim\int_{b_{1}}^{+\infty}\left\|\frac{1}{s-z}\right\|^{2}_{L^{2}(0,+\infty)}e^{-t(v-b_{1})}dv\\
&\lesssim \int_{b_{1}}^{+\infty}\frac{e^{-t(v-b_{1})}}{\sqrt{v-b_{1}-\eta}}dv\\
&=\int_{0}^{\eta}\frac{e^{-t\kappa}}{\sqrt{\kappa-\eta}}d\kappa+\int_{\eta}^{+\infty}\frac{e^{-t\kappa}}{\sqrt{\kappa-\eta}}dv,
\end{split}
\end{equation}
where
\begin{equation}
\int_{0}^{\eta}\frac{e^{-t\kappa}}{\sqrt{\kappa-\eta}}d\kappa =\int_{0}^{1}\frac{\sqrt{\eta}e^{-tw\eta}}{\sqrt{w-1}}dw\lesssim t^{-1/4}\int_{0}^{1}\frac{1}{\sqrt{\sqrt{w}(w-1)}}dw\lesssim t^{-1/4},
\end{equation}
\begin{equation}
\int_{\eta}^{+\infty}\frac{e^{-t\kappa}}{\sqrt{\kappa-\eta}}d\kappa\lesssim\int_{0}^{+\infty}\frac{e^{-tw}}{\sqrt{w}}dw
=t^{-1/2}\int_{0}^{+\infty}\frac{e^{-\lambda}}{\lambda}d\lambda\lesssim t^{-1/2}.
\end{equation}
\end{proof}
Based on the above discussion, we have the following proposition.
\begin{proposition}
As $t\rightarrow\infty$, $(I-S)^{-1}$ exists, which implies $\bar{\partial}$ Problem \ref{dbarproblem} has an unique solution.
\end{proposition}
Aim at the asymptotic behavior of $m^{(3)}$, we make the asymptotic expansion
\begin{equation}
    m^{(3)}=I+z^{-1}m^{(3)}_{1}(x,t)+\mathcal{O}(z^{-2}), \quad {\rm as} \quad z\rightarrow \infty
\end{equation}
where
\begin{equation}\label{m31}
    m^{(3)}_{1}(x,t)=\frac{1}{\pi}\iint_{\mathbb{C}}m^{(3)}(s)W^{(3)}(s)dA(s).
\end{equation}
To recover the solution of \eqref{nmkdv}, we shall discuss the asymptotic behavior of $m^{(3)}_{1}(x,t)$, thus we have
the following proposition.
\begin{proposition} \label{m31guji}As $t\rightarrow \infty$,
    \begin{equation}
    \vert m^{(3)}_{1}(x,t) \vert \lesssim t^{-1}.
    \end{equation}
\end{proposition}
\begin{proof}
\begin{equation}
\begin{split}
\left|m_{1}^{(3)}\right|&\leq\frac{1}{\pi}\iint_{\Omega_{11}}\left|m^{(2)}_{RHP}\bar{\partial}\mathcal{R}^{(2)}[m^{(2)}_{RHP}]^{-1}\right|dA(s)\\
&\lesssim\int_{2a_{1}}^{b_{1}+1}\int_{a_{1}+b_{1}-v}^{2a_{1}}|\bar{\partial}\chi_{\mathcal{Z}}(z)|e^{-t(v-b_{1})}dudv\\
&\lesssim t^{-1}.
\end{split}
\end{equation}
\end{proof}
\hspace*{\parindent}

\section{Long time asymptotics for nonlocal mKdV equation}
\begin{proposition}
For giving  reflectionless scattering data
\begin{equation}
\sigma_{d}^{\Lambda}=\{(\eta_{k},\hat{c}_{k}), k\in\Lambda\},
\end{equation}
its corresponding $N(\Lambda)$-solution for nonlocal mKdV equation \eqref{nmkdv} allows the following long time asymptotics
\begin{equation}
q(x,t)=cq_{sol}(x,t;\sigma_{d}^{\Lambda})-t^{-1/2}f+\mathcal{O}(t^{-1}), \quad t\rightarrow\infty,
\end{equation}
where
\begin{equation}
f=\sum_{i=1}^{4}\frac{1}{2\sqrt{\theta^{''}(\zeta_{i})}}\left(m_{11}^{2}(\zeta_{i})\beta_{12}^{\zeta_{i}}+m_{12}^{2}(\zeta_{i})\beta_{21}^{\zeta_{i}}\right).
\end{equation}
\end{proposition}
\begin{proof}
Recall all the transformations for $m(x,t;z)$, we obtain
\begin{equation}
m(z)=T(\infty)^{-\sigma_{3}}m^{(3)}(z)E(z)m^{(out)}(z)\mathcal{R}^{(2)}(z)^{-1}T(z)^{\sigma_{3}},
\end{equation}
then for $z\rightarrow\infty$ out $\underset{j=1,2,3,4}\cup\overline{\Omega}_{ij}$, $i=1,2,\cdot\cdot\cdot,6,$
\begin{equation}
m=\left(I+\frac{m_{1}^{(3)}}{z}+\cdot\cdot\cdot\right)\left(I+\frac{E_{1}}{z}+\cdot\cdot\cdot\right)\left(I+\frac{m_{1}^{(out)}}{z}+\cdot\cdot\cdot\right),
\end{equation}
obviously,
\begin{equation}
m_{1}=m_{1}^{(out)}+m_{1}^{(3)}+E_{1}.
\end{equation}
Take use the potential recovering formulae \eqref{resconstructm}, we have
\begin{equation}
\begin{split}
q(x,t)&=-\mathrm{i}\left[m_{1}^{(out)}+m_{1}^{(3)}+E_{1}\right]_{12}\\
&=cq_{sol}(x,t;\sigma_{d}^{\Lambda})-t^{-1/2}f+\mathcal{O}(t^{-1}).
\end{split}
\end{equation}
\end{proof}

\noindent\textbf{Acknowledgements}

This work is supported by  the National Natural Science
Foundation of China (Grant No. 11671095,  51879045).


\begin{thebibliography}{99}

   \bibitem{GGKM}C. S. Gardner, J. M. Green, M. D. Kruskal and R. M. Miurra,
    \newblock{Method for solving the Kortweg-de Vries equation},
    Phys. Rev. Lett., 19(1967), 1095-1097.

   \bibitem{ve1974}V. E. Zakharov and A.B. Shabat,
   \newblock{A scheme for integrating the nonlinear
    equations of mathematical physics by the method of the inverse scattering problem},
    Funk. Anal. Pril., 6(1974), 43-53.

    \bibitem{ve1973}V.E. Zakharov and A.B. Shabat,
    \newblock{A scheme for integrating the nonlinear equations of mathematical physics by the method of the inverse scattering
     problem. II},
     Funk. Anal. Pril., 13(1979), 13-22.

    \bibitem{Manakov1974} S. V. Manakov,
    \newblock{Nonlinear Fraunhofer diffraction},
    \newblock{Sov. Phys. JETP}, 38(1974), 693-696.

    \bibitem{DZ1993} P. Deift, X. Zhou,
    \newblock{A steepest descent method for oscillatory Riemann-Hilbert prblems. Asymptotics for the MKdV equation},
    Ann. Math., 137(1993), 295-368.

   \bibitem{NLS0} X. Zhou and P. Deift,
   \newblock{Long-time behavior of the non-focusing nonlinear
    Schr$\ddot{o}$dinger equation-a case study},
    Lectures in Mathematical Sciences, Graduate School of Mathematical Sciences, University of Tokyo, 1994.

    \bibitem{NLS1} P. Deift and X. Zhou,
    \newblock{Long-time asymptotics for solutions of the NLS equation with initial data in a weighted Sobolev space},
    Comm. Pure Appl. Math., 56(2003), 1029-1077.

    \bibitem{KDV} K. Grunert and G. Teschl,
    \newblock{Long-time asymptotics for the Korteweg de Vries equation via noninear steepest descent},
    Math. Phys. Anal. Geom., 12(2009), 287-324.

    \bibitem{CH} A. B. de Monvel, A. Kostenko, D. Shepelsky and G. Teschl,
    \newblock{Long-time asymptotics for the Camassa-Holm equation},
     SIAM J. Math. Anal, 41(2009), 1559-1588.

    \bibitem{DP} A. B. de Monvel, J. Lenells and D. Shepelsky,
    \newblock{Long-time asymptotics for the Degasperis-Procesi equation on the half-line},
     Ann. Inst. Fourier, 69(2019), 171-230.


    \bibitem{FL} J. Xu, E. G. Fan,
    \newblock{Long-time asymptotics for the Fokas-Lenells equation
    with decaying initial value problem: Without solitons},
    J. Differential Equations, 259(2015), 1098-1148.

    \bibitem{MM1} K. T. R. McLaughlin and P. D. Miller,
    \newblock{The $\bar{\partial}$-steepest descent method and the asymptotic behavior of polynomials orthogonal on the unit circle with
     fixed and exponentially varying non-analytic weights},
     Int. Math. Res. Not., (2006), Art. ID 48673.

    \bibitem{MM2} K. T. R. McLaughlin and P. D. Miller,
    \newblock{The $\bar{\partial}$-steepest descent method for orthogonal polynomials on the real line with varying weights},
    Int. Math. Res. Not., (2008), Art. ID 075

    \bibitem{DM} M. Dieng, K. D. T. R. McLaughlin,
    \newblock{Dispersive asymptotics for linear and integrable equations by the Dbar steepest descent method,
    Nonlinear dispersive partial differential equations and inverse scattering},
    Fields Inst. Commmun., Springer, New York, 2019, 253-291.

    \bibitem{CJ} S. Cuccagna, R. Jekins,
    \newblock{On asymptotic stability $N$-solitons of the defocusing nonlinear Schr\"odinger equation}
    \newblock{Comm. Math. Phys.}, 343(2016), 921-969.

    \bibitem{BJ} M. Borghese, R. Jenkins, K. D. T. R. McLaughlin, P. Miller,
    \newblock{Long-time aysmptotic behavior of the focusing nonlinear Schr$\ddot{o}$dinger equation},
    Ann. I. H. Poincar\'e Anal, 35(2018), 997-920.

   \bibitem{JL} R. Jenkins, J. Liu, P. Perry and C. Sulem,
   \newblock{Soliton resolution for the derivative nonlinear Schr$\ddot{o}$dinger equation},
   Commun. Math. Phys., 363(2018), 1003-1049.

   \bibitem{nmkdv1} M.J. Ablowitz and Z.H. Musslimani,
   \newblock{Inverse scattering transform for the integrable nonlocal nonlinear Schr$\ddot{o}$dinger equation},
   Nonlinearity, 29 (2016) 915-946.

   \bibitem{nmkdv2} M.J. Ablowitz, Z.H. Musslimani,
   \newblock{Integrable nonlocal nonlinear equations},
    Stud. Appl. Math. 139 (2017) 7-59.

  \bibitem{eg} M.J. Ablowitz, P.A. Clarkson,
  \newblock{Soliton, Nonlinear Evolution Equations and Inverse Scattering},
   Cambridge Univeristy Press, Cambridge, 1991.

   \bibitem{PT} C.M. Bender, S. Boettcher,
   \newblock{Real spectra in non-Hermitian Hamiltonians having PT symmetry},
   Phys. Rev. Lett. 80 (1998) 5243-5246.

   \bibitem{eg2} X. Y. Tang, Z. F. Liang and X. Z. Hao,
   \newblock{Nonlinear waves of a nonlocal modified KdV equation in the atmospheric and oceanic dynamical system},
   Commun. Nonlinear Sci. Numer. Simul. 60 (2018) 62-71.



    \bibitem{Darboux} J. L. Ji, Z. N. Zhu,
    \newblock{On a nonlocal modified Korteweg-de Vries equation: Integrability, Darboux transformation and soliton solutions},
    Commun. Nonlinear Sci. Numer. Simul., 42 (2017) 699.

    \bibitem{IST} J. L. Ji, Z. N. Zhu,
    \newblock{Soliton solutions of an integrable nonlocal modified Korteweg-de Vries equation through inverse scattering transform},
     J. Math. Anal. Appl., 453 (2017) 973-984.

   \bibitem{ZY} G. Zhang, Z. Yan,
    \newblock{Inverse scattering transforms and soliton solutions of focusing and
     defocusing nonlocal mKdV equations with non-zero boundary conditions},
    \newblock{\em Pys. D}, 402(2020), 132170.


     \bibitem{CF} Q. Y. Cheng, E. G. Fan,
     \newblock{Soliton resolution for the focusing Fokas-Lenells equation with weighted Sobolev initial data},
     Arxiv, arXiv: 2010.08714, 2020.

     \bibitem{YF} Y. L. Yang, E. G. Fan,
     \newblock{On asymptotic approximation of the modified Camassa-Holm equation in different space-time solitonic regions}
     Arxiv, arXiv:2101.02489v2, 2021.

     \bibitem{XZF} T. Y. Xu, Z. C. Zhang and E. G. Fan,
     \newblock{Long time asymptotics for the defocusing mKdV equation with finite density initial data in different solitonic regions}
      Arxiv, arXiv:2108.06284v3, 2021.

      \bibitem{ZXF} Z. C. Zhang, T. Y. Xu and E. G. Fan,
     \newblock{Soliton resolution and asymptotic stability of N-soliton solutions for the defocusing mKdV equation with finite density type initial data}
      Arxiv, arXiv:2108.03650v3, 2021.

     \bibitem{HF} F. J. He, E. G. Fan and J. Xu,
    \newblock{Long-time asymptotics for the nonlocal mKdV equation},
     Commun. Theor. Phys. 71 (2019) 475-488.

    \bibitem{BK} G. Biondini, G. Kova$\breve{c}$i$\breve{c}$,
    \newblock{Inverse scattering transform for the focusing nonlinear Schr$\ddot{o}$dinger equation with nonzero boundary conditions},
    J. Math. Phys. 55 (2014) 031506.

   \bibitem{Var} G. Varzugin,
    \newblock{Asymptotics of oscillatory Riemann-Hilbert problems}.
    J. Math. Phys, 37(1996), 5869-5892.




\end{thebibliography}
\end{document}